\documentclass{amsart}

\usepackage{hyperref}
\usepackage{stmaryrd}
\usepackage[all]{xy}
\usepackage{tikz-cd}

\usepackage{leftidx}
\RequirePackage{amsmath}
\RequirePackage{amssymb}
\RequirePackage{amsxtra}
\RequirePackage{amsfonts}
\RequirePackage{latexsym}
\RequirePackage{euscript}
\RequirePackage{amscd}
\RequirePackage{amsthm}

\usepackage{mathtools}
\usepackage{tikz-cd}
\usepackage[shortlabels]{enumitem}

\newtheorem{prop}{Proposition}[section]
\newtheorem{defi}[prop]{Definition}
\newtheorem{conj}[prop]{Conjecture}
\newtheorem{lem}[prop]{Lemma}
\newtheorem{thm}[prop]{Theorem}

\newtheorem{cor}[prop]{Corollary}

\theoremstyle{remark}
\newtheorem{examp}[prop]{Example}
\newtheorem{remar}[prop]{Remark}

\DeclareMathOperator{\Spf}{Spf}

\DeclareMathAlphabet{\mathpzc}{OT1}{pzc}{m}{it}
\DeclareMathOperator{\Aut}{Aut}
\DeclareMathOperator{\End}{End}
\DeclareMathOperator{\Hom}{Hom}

\DeclareMathOperator{\Ind}{Ind}

\DeclareMathOperator{\Res}{Res}
\DeclareMathOperator{\Sym}{Sym}

\DeclareMathOperator{\GL}{GL}

\DeclareMathOperator{\Ker}{Ker}

\DeclareMathOperator{\Gal}{Gal}

\DeclareMathOperator{\tr}{tr}
\DeclareMathOperator{\Stab}{Stab}
\DeclareMathOperator{\Irr}{Irr}

\DeclareMathOperator{\Spec}{Spec}
\DeclareMathOperator{\mSpec}{m-Spec}

\DeclareMathOperator{\id}{id}
\DeclareMathOperator{\Mod}{Mod}

\DeclareMathOperator{\Sp}{Sp}
\DeclareMathOperator{\Rep}{Rep}

\DeclareMathOperator{\Ext}{Ext}

\newcommand{\Q}{\mathbb{Q}}
\newcommand{\Qp}{\mathbb {Q}_p}
\newcommand{\Zp}{\mathbb{Z}_p}
\newcommand{\Qpbar}{\overline{\mathbb{Q}}_p}
\newcommand{\Zpbar}{\overline{\mathbb{Z}}_p}

\newcommand{\ZZ}{\mathbb Z}

\newcommand{\QQ}{\mathbb Q}

\newcommand{\ev}{\mathrm{ev}}

\newcommand{\RR}{\mathbb R}
 
\newcommand{\mm}{\mathfrak m}

\newcommand{\Fpbar}{\overline{\mathbb{F}}_p}

\newcommand{\OO}{\mathcal O}

\newcommand{\TT}{\mathbb T}

\DeclareMathOperator{\wtimes}{\widehat{\otimes}}

\newcommand{\cV}{\check{\mathbf{V}}}
\DeclareMathOperator{\Ad}{Ad}

\newcommand{\md}{\mathrm m}

\newcommand{\la}{\mathrm{la}}

\newcommand{\pro}{\mathrm{pro}}

\newcommand{\alg}{\mathrm{alg}}
\newcommand{\cont}{\mathrm{cont}}

\newcommand{\Z}{\mathbb{Z}}
\newcommand{\F}{\mathbb{F}}
\newcommand{\R}{\mathbb{R}}

\newcommand{\C}{\mathbb{C}}

\newcommand{\slt}{\mathfrak{sl}_2}

\DeclareMathOperator{\Frob}{Frob}

\newcommand{\PGL}{\mathrm{PGL}}

\newcommand{\rig}{\mathrm{rig}}

\newcommand{\Cp}{\mathbb C_p}

\newcommand{\rhobar}{\bar{\rho}}

\newcommand{\univ}{\mathrm{univ}}

\newcommand{\cyc}{\mathrm{cyc}}

\newcommand{\Ghat}{\widehat{G}}
\newcommand{\LG}{\leftidx{^L}G}
\newcommand{\LPad}{\leftidx{^L}P^{\mathrm{ad}}}
\newcommand{\CPad}{\leftidx{^C}P^{\mathrm{ad}}}
\newcommand{\That}{\widehat{T}}
\newcommand{\Bhat}{\widehat{B}}

\newcommand{\CG}{\leftidx{^C}G}
\newcommand{\LGf}{\LG_f}

\newcommand{\Xrig}{\mathfrak{X}^{\mathrm{rig}}}
\newcommand{\ghat}{\widehat{\mathfrak g}}
\newcommand{\that}{\widehat{\mathfrak t}}
\DeclareMathOperator{\Lie}{Lie}

\newcommand{\Rla}{R\text{-}\la}

\newcommand{\Sla}{S\text{-}\la}

\newcommand{\Gm}{\mathbb{G}_m}
\newcommand{\tw}{\mathrm{tw}}
\newcommand{\Sen}{\mathrm{Sen}}
\newcommand{\Pad}{P^{\mathrm{ad}}}

\usepackage{mathtools}

\DeclarePairedDelimiter{\set}{\{}{\}}

\newcommand{\Spl}{\mathrm{Spl}}

\newcommand{\Ytilde}{\widetilde{Y}}
\newcommand{\Mtilde}{\widetilde{\mathcal M}}

\title{Infinitesimal characters in arithmetic families}
\author{Gabriel Dospinescu}
\address{
Ecole Normale Sup\'erieure de Lyon,
46 all\'ee d'Italie,
69 364 Lyon Cedex 07, France}
\email{gabriel.dospinescu@ens-lyon.fr}

\author{Vytautas Pa\v{s}k\={u}nas}
\address{Universit\"at Duisburg-Essen,
Fakult\"at f\"ur Mathematik, Thea-Leymann-Str.\,9,
45127 Essen, Germany}
\email{paskunas@uni-due.de}

\author{Benjamin Schraen}
\address{Universit\'e Paris--Saclay, CNRS, Laboratoire de math\'ematiques d'Orsay, 91405, Orsay, France}
\email{benjamin.schraen@universite-paris-saclay.fr}

\date{\today.}
\begin{document} 
\maketitle

\begin{abstract} We associate infinitesimal characters to (twisted)
  families of $L$-parameters and $C$-parameters of $p$-adic reductive
  groups.  We use the construction to study the action of the centre
  of the universal enveloping algebra on the locally analytic vectors
  in the Hecke eigenspaces in the completed cohomology.
\end{abstract}
\tableofcontents

\section{Introduction} 
  
Let $G$ be a connected reductive group over $\mathbb{R}$ and let
$\mathfrak{g}=\Lie(G(\R))$.  The center $Z(\mathfrak{g})$ of the
enveloping algebra of $\mathfrak{g}_{\C}$ has a strong influence on
the representation theory of $G$, as we will briefly recall. Let
$\widehat{G(\R)}$ be the unitary dual of $G(\R)$.  By Segal's
theorem\footnote{This crucially uses the unitarity hypothesis, and the
  result fails for irreducible Banach representations, making it
  impossible to adapt the proof in the $p$-adic world.} there is an
``infinitesimal character map''
\[\widehat{G(\R)}\to {\rm Hom}_{\mathbb{C}\text{-}\rm
    alg}(Z(\mathfrak{g}), \mathbb{C}), \, \pi\mapsto \chi_{\pi},\] and
by Harish-Chandra's finiteness theorem this map has finite
fibres. Moreover, the study of the differential equations hidden in
the existence of $\chi_{\pi}$ yields important information about the
asymptotic behaviour of the matrix coefficients of $\pi$, and this can
be used to prove Casselman's subrepresentation theorem and the
Langlands classification. Going somewhat in the opposite direction,
one can use the existence of an infinitesimal character to deduce
finiteness results: another classical result of Harish-Chandra ensures
that any admissible Banach representation of $G(\R)$ with an
infinitesimal character has finite length.
   
It is both natural and tempting to investigate whether similar
phenomena happen in the $p$-adic world, but so far the situation is
far less rosy due to our rather poor understanding of these
representations. The results of this paper and its sequel \cite{DPS2}
point to some striking similarities with the above picture, as well as
significant differences. Our results are most definite in the cases
when a connection to Galois representations can be made.
   
\subsection{Quick overview}
  
There are essentially three main results in this paper. 

The first is a
very general construction (see sections \ref{Definf} and
\ref{sec_Cgr}) attaching infinitesimal characters to (twisted)
families of $L$-parameters and $C$-parameters of $p$-adic reductive
groups. This is a $p$-adic analogue of a classical construction for
$L$-parameters of real reductive groups (see section \ref{sec:real}
for a review). For an $L$-parameter
$\rho: \Gal_F\to \LG(\overline{\Q}_p)$ the associated infinitesimal
character $\zeta_{\rho}$ is obtained from the conjugacy class of the
semisimple part of the Sen operator attached to $\rho$.  

The second is
an interpolation result (theorems \ref{infinitesimal} and
\ref{infinitesimal_central}), stating that if we can produce
sufficiently many\footnote{In a sense which can and will be made
  precise later on.} locally analytic vectors in sufficiently many
members of a family of admissible Banach representations of a $p$-adic
reductive group $G$, with the property that the center of the
universal enveloping algebra of $\Lie(G)$ acts on them by characters
in a compatible way, then these characters glue and all Banach
representations of the family have an infinitesimal character obtained
by specialisation. 

The third exploits these two results to show
(theorems \ref{lzero} and \ref{lzero_central}) that many locally
analytic representations arising ``in nature'' have infinitesimal
characters, which can be explicitly computed from $p$-adic Hodge
theoretic data.  Informally, if we can attach Galois representations
to Hecke eigenspaces in completed cohomology, so that at classical
points the Hodge--Tate weights of the Galois representation match the
infinitesimal character of the corresponding classical automorphic
form (i.e.~a weak form of local-global compatibility at $p$), then
this property propagates by analytic continuation to all Hecke
eigenspaces. Thus in favorable settings if the Hecke eigenspace is
non-zero, then its locally analytic vectors have an infinitesimal
character, which can be related to the generalized Hodge--Tate weights
of the associated Galois representation.
      
Examples of situations covered by our results are given by
(sufficiently non-degenerate) Hecke eigenspaces in the completed
cohomology of modular curves, or more generally Shimura curves over
totally real fields, or of compact unitary Shimura varieties studied
by Caraiani--Scholze, or of definite unitary groups over totally real
fields (see sections \ref{sec_modular}, \ref{sec_shimura},
\ref{sec_unitary} and \ref{CS}). The formalism developed in this paper
also applies to ``patched'' modules obtained by patching these
completed cohomology groups as in \cite{6auth}. In particular, we show
(see theorem \ref{patched}) that the candidates for the $p$-adic local
Langlands correspondence for $p$-adic $\GL_n(F)$ constructed in
\cite{6auth} have infinitesimal characters depending only on the local
Galois representation one starts with, giving further nontrivial
evidence that these candidates are independent of the choices made in
their construction.
   
\subsection{Problems in the $p$-adic world}
   
Fix a prime number $p$ and a connected reductive group $G$ over
$\Qp$. Let $L$ be the coefficient field of our representations, a
sufficiently large finite extension of $\Qp$. We assume that $G$ is
split over $L$.  Let $\mathfrak{g}$ be the Lie algebra of the $p$-adic
Lie group $G(\Qp)$ (or equivalently of $G$).
    
Before discussing the problems arising for $p$-adic representations,
it is convenient to introduce some notation and recall a certain
number of basic results.  Let $H$ be a $p$-adic Lie group and let
$\mathfrak{h}=\Lie(H)$. Let ${\rm Ban}_{L} (H)$ be the category of
admissible $L$-Banach space representations of $H$ and let
$\widehat{H}_L$ be the set of isomorphism classes of absolutely
irreducible objects of this category. Contrary to the world of real
groups, admissibility of irreducible unitary representations does not
come for free, thus it is better to impose it from the very
beginning. Let ${\rm Ban}_{L} (H)^{\rm unit}$ be the full subcategory
of ${\rm Ban}_{L} (H)$ consisting of unitary representations,
i.e.~those having an $H$-invariant norm defining the Banach space
topology, and let $\widehat{H}^{\rm unit}_L$ be the set of isomorphism
classes of absolutely irreducible objects of
${\rm Ban}_{L} (H)^{\rm unit}$. If $\Pi\in {\rm Ban}_{L} (H)$ then we
let $\Pi^{\la}$ be the subspace of locally analytic vectors in $\Pi$,
i.e.~vectors whose orbit map $g\mapsto g.v$ is a locally analytic map
from $H$ to $\Pi$. Then $\Pi^{\la}$ is a dense subspace of $\Pi$ by
\cite{ST}, on which
$U(\mathfrak{h})_L\coloneqq U(\mathfrak{h})\otimes_{\Qp} L$ naturally
acts ($U(\mathfrak{h})$ is the universal enveloping algebra of
$\mathfrak{h}$). Let $Z(\mathfrak{h})_L$ be the center of
$U(\mathfrak{h})_L$ and let $\chi: Z(\mathfrak{h})_L\to L$ be an
$L$-algebra homomorphism. If $\Pi \in {\rm Ban}_L(H)$, we say that
$\Pi^{\la}$ has infinitesimal character $\chi$ if $Xv=\chi(X)v$ for
all $X\in Z(\mathfrak{h})_L, v\in \Pi^{\la}$. We will also abuse
language and say that $\Pi$ has infinitesimal character $\chi$ when
$\Pi^{\la}$ does so.

A first major problem, which is unfortunately unsolved (to our
knowledge) for {\it{any}} $G$ beyond tori, is the existence of an
infinitesimal character map
\begin{equation}\label{infcharmap}
  \widehat{G(\Qp)}_L\to {\rm Hom}_{L\text{-}\rm alg} (
  Z(\mathfrak{g})_L, L),
\end{equation}
i.e.~whether $\Pi^{\la}$ has an infinitesimal character for each
$\Pi\in \widehat{G(\Qp)}_L$. It is conjectured in \cite{DS} that this
is the case, and it is proved there (based on deep results of Ardakov
and Wadsley \cite{AW}) that this is so if we further\footnote{This is
  not automatic; actually it is not even known if $\Pi^{\la}$ always
  has finite length, and it is not even clear that this is to be
  expected.} assume that $\Pi^{\la}$ is an absolutely irreducible
locally analytic representation of $G(\Qp)$. If we replace
$\widehat{G(\Qp)}$ by $\widehat{G(\Qp)}^{\rm unit}$, then essentially
the only group for which the existence of this map is known is
$G=\GL_2$, and the argument in \cite{GDinf} fully uses the $p$-adic
local Langlands correspondence for $\GL_2(\Qp)$. We will give a
different proof in this paper, which still uses this input, as well as
global results, but sheds some more light on what could happen for
other groups. Contrary to the case of real groups, already for
$\GL_2(\Qp)$ all fibres of the map \eqref{infcharmap} are
uncountable. Still, the fibres of the restriction to
$\widehat{\GL_2(\Qp)}^{\rm unit}$ have nice geometric structures: they
are the set of $L$-points of some (non quasi-compact) rigid analytic
spaces.
  
Inspired by the situation for real groups, one may ask whether every
admissible $L$-Banach representation $\Pi$ of $G(\Qp)$ with an
infinitesimal character has finite length. One cannot expect such a
result to hold in great generality (i.e. with $G(\Qp)$ replaced by any
$p$-adic Lie group), since there are admissible $L$-Banach
representations $\Pi$ of $\GL_2(\Zp)$ with an infinitesimal character
and of infinite length. However, we will show in \cite{DPS2} that this
holds for unitary representations of $\GL_2(\Qp)$ (the proof does
{\it{not}} use the $p$-adic Langlands correspondence for
$\GL_2(\Qp)$), and that a similar result holds for the group of units
of a quaternion division algebra $D$ over $\Qp$, if we further assume
(in the case of $D^{\times}$) that the infinitesimal character is not
the one of an irreducible algebraic representation. When this
hypothesis on the infinitesimal character is no longer satisfied, it
is not clear to us whether the $\GL_2(\mathbb Z_p)$ or the
$\GL_2(\Qp)$ phenomenon prevails, since $D^{\times}$ is compact modulo
the centre.  This in turn has interesting local and global
applications, for instance to the study of eigenspaces in the
completed cohomology of certain Shimura curves or to the (still
hypothetical, despite a lot of recent progress \cite{ScholzeLT,
  Ludwig, howe, Knight}) $p$-adic Jacquet--Langlands correspondence for
the pair $(\GL_2(\Qp), D^{\times}$). We leave these applications for
\cite{DPS2}, but they should suggest that the existence of an
infinitesimal character on a Banach representation has rather strong
consequences on its structure, and it is precisely with this existence
problem that we are dealing in this paper.
         
\subsection{Analytic continuation and infinitesimal characters}
The basic idea is very simple: suppose that a representation $\Pi$
lives in a family of Banach representations $\Pi_{X}$, parameterized
by some rigid analytic space $X$, and suppose that we can produce
sufficiently ``many'' locally analytic vectors in sufficiently
``many'' members of a family of admissible Banach representations,
with the property that $Z(\mathfrak{g})_L$ acts on them by characters
in a compatible way, so that these characters interpolate to a
character $\chi: Z(\mathfrak{g})_L\to \mathcal{O}_X(X)$.  A density
argument then shows that $ Z(\mathfrak{g})_L$ acts by $\chi$ on
$\Pi_{X}^{\la}$, and thus all members of the family have infinitesimal
characters, obtained as specialisations of $\chi$.
   
In order to maximize its flexibility, it is convenient to state the
analytic continuation argument in a rather general and abstract
setting.  Let $\mathcal{O}$ be the ring of integers of $L$ and let
$(R, \mm)$ be a complete local noetherian $\mathcal{O}$-algebra with
the same residue field as $L$. Let $\Xrig$ be the generic fibre of the
formal scheme $\Spf R$ and let
$R^{\rig}=H^0(\Xrig, \mathcal{O}_{\Xrig})$.
    
Let $G$ be a connected reductive group over $\Qp$ and let $K$ be a
compact open subgroup of $G(\Qp)$. Finally, let $M$ be a finitely
generated $R[\![K]\!]$-module. There is a canonical topology on $M$
making it a compact topological $R[\![K]\!]$-module. Endowed with the
supremum norm, the (not necessarily admissible) unitary Banach space
representation of $K$
\[\Pi\coloneqq {\rm Hom}_{\mathcal{O}}^{\rm cont}(M,L)\]
can be thought of as a family of objects of
${\rm Ban}_L(K)^{\rm unit}$, parameterized by the maximal spectrum
$\mSpec(R[1/p])$: for any $x\in \mSpec(R[1/p])$ the subspace
$\Pi[\mm_x]$ of $\Pi$ consisting of elements killed by the maximal
ideal $\mm_x$ is an object of ${\rm Ban}_{\kappa(x)}(K)^{\rm unit}$
(and also an object of ${\rm Ban}_L(K)^{\rm unit}$ since the residue
field $\kappa(x)$ of $x$ is finite over $L$).
  
Choose the coefficient field $L$ large enough so that $G$ splits over
$L$. Let ${\rm Irr}_G(L)$ be the set of isomorphism classes of
irreducible algebraic representations of $G_L$.  For any
$V\in {\rm Irr}_G(L)$, by evaluating at $L$ we get an action of
$G(\Qp)\subset G(L)$ on $V$, and $V$ is absolutely irreducible as a
representation of $G(\Qp)$ (since $G$ splits over $L$), thus it has an
infinitesimal character. For such $V$ consider the $R[1/p]$-module
\[M(V)={\rm Hom}_K^{\rm cont}(V, \Pi)^\prime,\] where $W'$ is the
topological $L$-dual of the topological $L$-vector space $W$.  It is
not difficult to see\footnote{One can also present
  $M(V)=V\otimes_{\mathcal{O}[\![K]\!]} M$, for the right
  $\mathcal{O}[\![K]\!]$-module structure on $V$ induced by the
  anti-automorphism of $\mathcal{O}[\![K]\!]$ sending $k\in K$ to
  $k^{-1}$.}  that $M(V)$ is a finitely generated $R[1/p]$-module. Let
$R_V$ be the quotient of $R$ that acts faithfully on $M(V)$.
   
We refer the reader to theorems \ref{infinitesimal} and
\ref{infinitesimal_central} for variations on the next basic theorem.
   
\begin{thm}\label{main1}
  Let $R$, $M$ be as above and suppose that there is an $L$-algebra
  homomorphism $\chi: Z(\mathfrak{g})_L\to R^{\rig}$ such that the
  following hold
  \begin{enumerate}
  \item there is an $M$-regular sequence $y_1,...,y_h\in \mm$ such
    that $M/(y_1,...,y_h)M$ is a finitely generated projective
    $\mathcal{O}[\![K]\!]$-module.
     
  \item for all $V\in {\rm Irr}_G(L)$ the ring $R_{V}$ is reduced.
     
  \item for all $V\in {\rm Irr}_G(L)$ and all
    $x\in \mSpec(R_{V}[1/p])$ the infinitesimal character of $V$ is
    the specialisation $\chi_x: Z(\mathfrak{g})_L\to \kappa(x)$ of
    $\chi$ at $x$.
     
  \end{enumerate}
  Then for all $y\in \mSpec(R[1/p])$ the representation $\Pi[\mm_y]$
  has infinitesimal character $\chi_y$.
\end{thm}
   
We end this paragraph with a detailed explanation of the key
ingredients in the proof of this result, since this also explains the
origin of the somewhat exotic hypotheses in the statement of the
theorem.  Let $D(K,L)$ be the algebra of $L$-valued distributions on
$K$.  The representation $\Pi$ gives rise to an
$R^{\rig}\wtimes_L D(K,L)$-module $\Pi^{\Rla}$, roughly speaking the
space of vectors $v\in \Pi$ that are locally analytic both for the
action of $K$ and that of $R$. More precisely, we have an isomorphism
\[(\Pi^{\Rla})^\prime\simeq (R^{\rig}\wtimes_L
  D(K,L))\otimes_{R[\![K]\!]} M\] describing $\Pi^{\Rla}$ as the
topological $L$-dual of the right-hand side. Moreover, the inclusion
$\Pi^{\Rla}[\mm_x]\to \Pi[\mm_x]^{\la}$ is an isomorphism for all
$x\in \mSpec(R[1/p])$.
     
It suffices therefore to prove that
$1\otimes D-\chi(D)\otimes 1\in R^{\rig}\wtimes_L D(K,L)$ kills
$\Pi^{\Rla}$ for all $D\in Z(\mathfrak{g})_L$. By continuity, it
suffices to prove this for a dense subspace of $\Pi^{\Rla}$. We will
explain how to find such a subspace under the hypotheses of the
theorem.  There are natural embeddings
\begin{equation}\label{firstemb}
  \bigoplus_{[V]\in {\rm Irr}_G(L)} {\rm Hom}_K(V, \Pi^{\Rla})\otimes_L V\to \Pi^{\Rla}
\end{equation}
and
\begin{equation} \label{secondemb} \bigoplus_{x\in \mSpec(R_{V}[1/p])}
  {\rm Hom}_K(V, \Pi[\mm_x]^{\la})\otimes_L V\to {\rm Hom}_K(V,
  \Pi^{\Rla})\otimes V,
\end{equation}
the first one identifying the left-hand side with the space of
$K$-algebraic vectors in $\Pi^{\Rla}$. We prove that under the
assumptions of the theorem both embeddings have dense image. Since the
third hypothesis implies that $1\otimes D-\chi(D)\otimes 1$ kills the
image of the second embedding for all $[V]\in {\rm Irr}_G(L)$ and
$D\in Z(\mathfrak{g})_L$, this will finish the proof.
     
The proof that \eqref{firstemb} and \eqref{secondemb} have dense image
crucially uses the first hypothesis, which allows us to replace $R$
first by $S=\mathcal{O}[\![x_1,...,x_h]\!]$ and then (this requires
changing the groups $K$ and $G$) by $\mathcal{O}$, while
simultaneously reducing the proof to the case
$M=\mathcal{O}[\![K]\!]$.  Sending $x_i$ to $y_i$ yields an action of
$S$ on $M$, and a key remark is that the first hypothesis forces $M$
to be a finite projective module over $S[\![K]\!]$.
   
Let us explain the proof that \eqref{firstemb} has dense image.  It is
proved in \cite{BHS} that $\Pi(M)^{\Rla}=\Pi(M)^{\Sla}$, thus we may
forget about $R$.  But $S[\![K]\!]$ is identified with
$\mathcal{O}[\![K']\!]$, where $K'=(1+2p\Zp)^h\times K$ is a compact
open subgroup of $G'=\Gm^h\times G$. The transparent link between
${\rm Irr}_G(L)$ and ${\rm Irr}_{G'}(L)$ reduces the proof to the case
$R=\mathcal{O}$ and $M$ is finite projective over
$\mathcal{O}[\![K]\!]$ (up to replacing $G$ and $K$ by $G'$ and
$K'$). This further reduces to the case $M=\mathcal{O}[\![K]\!]$. In
this case the result we want to prove is equivalent to the density of
the image of the natural map
$\oplus_{V} V^*\otimes V\to \mathcal{C}^{\la}(K, L)$. Since
$\oplus_{V} V^*\otimes V\simeq \mathcal{O}_G(G)$, the required density
is a consequence of the following more general result, a simple
application of Amice's theorem: for any smooth affine scheme $X$ over
$\Spec(\Qp)$ and for any open and closed subset $U$ of $X(\Qp)$, the
natural map
$L\otimes_{\Qp} \mathcal{O}_X(X)\to \mathcal{C}^{\la}(U, L)$, obtained
by restricting regular functions on $X$ to $U$, has dense image.

Finally, let us explain the proof of the density of the map
\eqref{secondemb}.  We argue by duality, using the Hahn--Banach
theorem. Thus we want to show that the dual of the right-hand side in
\eqref{secondemb} embeds in the dual of the left-hand side.  The
choice of a $K$-stable lattice $\Theta$ in $V$ induces an integral
structure $M(\Theta)$ on $M(V)$. Simple arguments show that
\[{\rm Hom}_K(V, \Pi^{\Rla})^\prime\simeq M(\Theta)\otimes_{R_{V}}
  R_{V}^{\rig}\] and for all $x\in \mSpec(R_{V}[1/p])$
\[{\rm Hom}_K(V, \Pi[\mm_x]^{\la})^\prime\simeq
  M(\Theta)\otimes_{R_{V}} \kappa(x).\] One can prove, using again the
first hypothesis, that $M(\Theta)$ is a finite free $S$-module and
deduce that it is a Cohen--Macaulay $R_{V}$-module. The result follows
now from the second hypothesis and the following commutative algebra
result: if $R$ is a complete local noetherian $\mathcal{O}$-algebra,
which is reduced and $\mathcal{O}$-torsion free, then for any faithful
Cohen--Macaulay $R$-module $M$ the natural map
\[M\otimes_R R^{\rig}\to \prod_{x\in (\Spf R)^{\rig}} M\otimes_R
  \kappa(x)\] is injective. Note that this is equivalent to the
density of the image of \eqref{secondemb} in the case when $G$ is
trivial. The proof is a simple application of the Noether
normalisation lemma and basic properties of faithful Cohen--Macaulay
modules.

\subsection{Infinitesimal characters and local Galois representations}
         
The previous paragraph gives a systematic way of proving the existence
of an infinitesimal character for a Banach representation, but one
needs to be able to find the objects $R,M$ and $\chi$ as in theorem
\ref{main1}. This is not at all a trivial problem. As we will explain
in the next paragraph, the existence of $R$ and $M$ can be proved in
many cases using global methods and the existence of Galois
representations attached to $p$-adic automorphic forms, combined with
suitable local-global compatibility results. In this paragraph we want
to focus on a key result of this paper, namely the construction of a
character $\chi$ starting from a family of local Galois
representations.  The construction uses Sen theory in families (as
discussed in \cite{BC}), Tannakian formalism, Chevalley's restriction
theorem and the Harish-Chandra isomorphism.

We find it useful to consider first the simplest possible case. Let
$\rho: \Gal_{\QQ_p}\to \GL_n(L)$ be a continuous Galois
representation, where $\Gal_{\QQ_p}$ is the absolute Galois group of
$\QQ_p$. We would like to attach to $\rho$ a character
$\zeta_{\rho}: Z(\mathfrak{g})_L\to L$, where $\mathfrak{g}$ is the
Lie algebra of $\GL_n(\QQ_p)$. By the Harish-Chandra isomorphism
$Z(\mathfrak{g})_L$ is a polynomial ring in $n$ variables, thus giving
$\zeta_{\rho}$ is equivalent to giving $n$ elements of $L$. These are
simply the coefficients of the characteristic polynomial of the Sen
operator of $\rho$. In other words $\zeta_{\rho}$ encodes the
generalized Hodge--Tate weights of $\rho$. When $\rho$ is de Rham with
regular Hodge--Tate weights, one can attach to $\rho$ an algebraic
representation of $\GL_n/\QQ_p$ over $L$, and $\zeta_{\rho}$ encodes
the highest weight of this representation.
      
Let us discuss now the general case. Let $F$ be a a finite extension
of $\QQ_p$ and let $G$ be a connected reductive group over $F$ with
Lie algebra $\mathfrak{g}$. Let $\Ghat$ be the dual group of $G$
(defined over $\mathbb{Z}$). This group comes with an action
$\mu_G: \Gal_F\to \Ghat$ of the absolute Galois group of $F$, and we
let $\LG=\Ghat\rtimes \Gal_F$ be the Langlands dual group of $G$. The
map $g\mapsto (1,g)$ identifies $\Ker(\mu_G)$ with a normal subgroup
of $\LG$, and we let $\LG_f=\LG/\Ker(\mu_G)=\Ghat\rtimes \Gamma$,
where $\Gamma=\Gal_F/\Ker(\mu_G)$.
   
Next, let $X$ be a rigid analytic variety and let $P$ be a
$\Ghat_X$-torsor, locally trivial for the \'etale
topology\footnote{One could allow other Grothendieck topologies as
  well.}. Let $\Pad$ be the sheaf of $\Ghat_X$-equivariant
automorphisms of $P$, which leave the base $X$ fixed. We define in
this generality the $L$-group $\LPad(X)$ of $G$ twisted by $P$, as
well as a notion of admissible representations
$\rho: \Gal_F\to \LPad(X)$, see definitions \ref{twistL} and
\ref{defi_twisted}. For instance, if $P=\Ghat_X$ is the trivial
$\Ghat_X$-torsor, then
$\LPad(X)=\LG_f(\OO_X(X))= \Ghat(\OO_X(X))\rtimes \Gamma$ and
admissibility of a continuous representation
$\rho: \Gal_F\to \Ghat(\OO_X(X))\rtimes \Gamma$ is the usual notion,
i.e. we ask that the induced map $\Gal_F\to \Gamma$ is the natural
projection. Our main construction associates to any admissible
representation $\rho: \Gal_F\to \LPad(X)$ an $L$-algebra homomorphism
\begin{equation}\label{zetarhoL}
  \zeta_{\rho}: Z(\Res_{F/\Qp}\mathfrak g)\otimes_{\Qp} L \rightarrow \OO_X(X).
\end{equation}
 
One may naturally wonder if working in the previous generality really
has some interest.  A motivating example appears in Chenevier's thesis
\cite{chenevier_thesis}, where he constructs an eigenvariety $X$ for a
unitary group compact at $\infty$ and split at $p$ associated to a
division algebra over a quadratic extension $E$ of $\QQ$, and a
pseudo-representation $t: \Gal_E\rightarrow \OO_X(X)$, which at points
corresponding to classical automorphic forms interpolates the traces
of corresponding Galois representations. The locus $X_{\mathrm{irr}}$,
where $t$ is absolutely irreducible, is an open rigid subvariety of
$X$ and Chenevier shows that over this locus, $t$ gives rise to a
representation $\rho: \Gal_E\rightarrow \mathcal A^*$, where
$\mathcal A$ is an Azumaya algebra over $X_{\mathrm{irr}}$, see
\cite[Thm.\,E]{chenevier_thesis}. A similar type of example arising in
the deformation theory of pseudo-representations appears in
\cite[\S4.2]{chenevier_det}.  Although, in these cases
$\Ghat= \LG_f=\GL_n$, the Galois representation takes values not in
$\GL_n(\OO_X(X_{\mathrm{irr}}))$, but in an inner form of it. We would
like our setup to cover this example and also the situations where the
action of $\Gal_F$ on $\Ghat$ is non-trivial.

All previous constructions can be adapted if we work with the
$C$-group $\CG$ of $G$ instead of the $L$-group. Recall that $\CG$ is
related to the $L$-group by an exact sequence
\begin{equation}\label{LGCG_intro}
 1\rightarrow \LG \rightarrow \CG \overset{d}{\rightarrow}\Gm \rightarrow 1
\end{equation}
and that $\CG$ is the $L$-group of a canonical central extension $G^T$
of $G$ by $\Gm$ over $F$. In particular, one can define (see remark
\ref{CPad}) a twisted $C$-group $\CPad(X)$ attached to a
$\Ghat_X$-torsor $P$, as well as a notion of admissibility for a
continuous representation $\rho: \Gal_F\rightarrow \CPad(X)$. And to
each such admissible representation $\rho$ we can associate (see
definition \ref{def_ringhom_C}) a character
\[\zeta_{\rho}^C: Z(\Res_{F/\Qp} \mathfrak g)\otimes_{\Qp} L
  \rightarrow \OO_X(X).\] If we consider $\CG$ as the $L$-group of
$G^T$ then $\zeta_{\rho}$ and $\zeta_{\rho}^C$ are related by a
twisting construction: taking square roots at the Lie algebra level is
just dividing by $2$ and we can always do that in characteristic zero.
We will only consider $L$-groups in the introduction, to avoid too
many technicalities, but we emphasize\footnote{We thank Peter Scholze
  for pointing out this.} that it is better to consider
representations valued in $C$-groups instead of $L$-groups. The first
issue is that one expects that the Galois representations attached to
$C$-algebraic automorphic forms take values in the $C$-group and not
the $L$-group, see \cite[Conj.\,5.3.4]{beegee}. Secondly, if we take
$L=\Qp$, $G=\GL_2$, so that $\LGf= \GL_2$, and
\[\rho: \Gal_{\Qp} \rightarrow \GL_2(\Qp), \quad g\mapsto \bigl
  (\begin{smallmatrix} \chi_{\cyc}(g)^{a} & 0 \\ 0 &
    \chi_{\cyc}(g)^b\end{smallmatrix}\bigr),\] where $a > b$ are
integers, then the Sen operator is the matrix
$\bigl (\begin{smallmatrix} a & 0 \\ 0 & b\end{smallmatrix}\bigr)$,
but the character $\zeta_{\rho}: Z(\mathfrak g)\rightarrow \Qp$ is not
an infinitesimal character of an algebraic representation of $\GL_2$.
The problem is caused is by the shift by the half sum of positive
roots appearing in the Harish-Chandra isomorphism.  If we choose a
twisting element
$\tilde{\delta}(t)\coloneqq \bigl (\begin{smallmatrix} t^{n+1} & 0 \\
  0 & t^n\end{smallmatrix}\bigr)$ for some $n\in \ZZ$ then to $\rho$
one may attach a Galois representation with values in
$\leftidx{^C}{\GL_2}(\Qp)$, which we denote by $\rho^C$. In this case,
$\zeta^C_{\rho^C}$ is equal to the infinitesimal character of
$\Sym^{a-b-1}\otimes \det^{b-n}$. See section \ref{sec_Cgr} for more
details.

The construction and the study of the characters $\zeta_{\rho}$ and
$\zeta_{\rho}^C$ occupies the first four chapters of the paper and is
a bit subtle since one needs to pay special attention to the action of
the group $\Gamma=\Gal_F/\Ker(\mu_G)$. One important result we prove
is the compatibility of our construction with the Buzzard--Gee
conjecture, cf. proposition \ref{C-alg-inf}. This roughly says that
for $C$-algebraic automorphic forms $\pi$, if $\rho_{\pi}$ is ``the''
Galois representation conjecturally attached to $\pi$ as in conjecture
$5.3.4$ in \cite{beegee}, then one can relate the infinitesimal
characters attached to the restriction of $\rho_{\pi}$ at places above
$p$ and the infinitesimal characters of the archimedean components of
$\pi$. This plays a crucial role in global applications, and it was
also a major motivation for our construction of $\zeta_{\rho}$ and
$\zeta_{\rho}^C$.
    
We end this long paragraph by sketching the construction of
$\zeta_{\rho}$ in the case when $X$ is an affinoid over $L$, $P$ is
the trivial $\Ghat_X$-torsor and $F=\Qp$. This case already covers the
key difficulties. Thus we start with an admissible representation
$\rho: \Gal_{\Qp}\rightarrow \LG_f(A)$, where $A$ is an $L$-affinoid
algebra, and we want to construct a character
$\zeta_{\rho}: Z(\mathfrak{g})_L\to A$.
       
Pick a finite Galois extension $E/\Qp$ splitting $G$ and pick an
embedding $\tau: E\to L$. The map $\zeta_{\rho}$ will be a composition
of two maps, each depending on the choice of $\tau$, but whose
composition is independent of any choice
\[Z(\mathfrak{g})_L\overset{\kappa_{\tau}}{\longrightarrow}
  S(\ghat^*)^{\Ghat}\overset{\theta_{\tau}}{\longrightarrow} A.\] We
write $S(W)$ for the symmetric algebra of an $L$-vector space $W$,
thought of as the ring of polynomial functions on the dual of $W$.
                   
The map $\kappa_{\tau}$ is an isomorphism, and is obtained by
combining the Harish-Chandra isomorphism and the Chevalley restriction
theorem. More precisely, let $T$ be a maximal torus in $G_E$ and let
$\widehat{T}$ be the dual torus in $\Ghat$. Let $\mathfrak{t}$ and
$\widehat{\mathfrak{t}}$ be the associated Lie algebras of $T$ and
$\widehat{T}$, and let $W$ be the Weyl group of the root system of
$(G,T)$, which is the same as the Weyl group of the dual root
system. Then $\kappa_{\tau}$ is the composition
\[ Z(\mathfrak{g})_L\simeq S(\widehat{\mathfrak{t}}^*)^W\simeq
  S(\widehat{\mathfrak{g}}^*)^{\Ghat},\] the first isomorphism being
obtained by base change along $\tau: E\to L$ of the (normalised)
Harish-Chandra isomorphism
$Z(\mathfrak{g}_E)\simeq S(\mathfrak{t})^W$, where we also use $\tau$
to identify
$\mathfrak{t}\otimes_{E,\tau} L\simeq \widehat{\mathfrak{t}}^*$, and
the second map being the inverse of the isomorphism
$S(\widehat{\mathfrak{g}}^*)^{\Ghat}\simeq
S(\widehat{\mathfrak{t}}^*)^W$ induced by restriction (this is
Chevalley's restriction theorem).
          
On the other hand, the map $\theta_{\tau}$ is the composition
\[S(\ghat^*)^{\Ghat}\to E\otimes A\to A,\] the second map being simply
$x\otimes a\to \tau(x)a$ and the first map being the evaluation of
$\Ghat$-invariant polynomial functions at a special element
\[\Theta_{\Sen, \rho}\in (\mathbb{C}_p\wtimes A)\otimes_{L}
  \widehat{\mathfrak{g}}.\] A priori the evaluation map lands in
$\mathbb{C}_p \wtimes A$, but using the Ax--Sen--Tate theorem and
special $\Gal_F$-equivariance properties of $\Theta_{\Sen, \rho}$ and
of the Harish-Chandra and Chevalley isomorphisms, we show that it
lands in $E\otimes A$, and moreover that
$\theta_{\tau}\circ \kappa_{\tau}$ is independent of $\tau$.  The
element $\Theta_{\Sen, \rho}$ is obtained using the Tannakian
formalism and the results of Berger--Colmez \cite{BC} on Sen theory in
families of Galois representations.
                  
\begin{remar}
 Let $R$ be a complete local noetherian $\OO$-algebra with residue field $k$. We expect that 
  even if  $\rho$ takes values in $\CG_f(R)$ the
  character $\zeta_\rho$ will take values in $R^{\rig}$ and not in $R[\tfrac{1}{p}]$, in general. 
  
  For example, if $G=\GL_n$ and  $\zeta_{\rho}$ takes values in $R[\tfrac{1}{p}]$ then this would imply that the  Sen polynomial of $\rho$ has coefficients in $R[\tfrac{1}{p}]$ and this would 
 impose  restrictions on the Hodge--Tate weights of the Galois representations obtained by specialising 
 $\rho$ at the maximal ideals of $R[\tfrac{1}{p}]$.  
 
 Specializing the example further, if $p\nmid 2n$ and $\rho^{\univ}: \Gal_{\Qp}\rightarrow \GL_n(R_{\overline{\rho}})$ is the universal deformation of an absolutely irreducible representation
  $\overline{\rho} : \Gal_{\Qp}\rightarrow\GL_n(k)$, which is  not isomorphic to
  its twist by the cyclotomic character,   then 
  the density results of \cite{EP} or \cite{Hellmann_Schraen}  imply that 
  the   locus in $\mSpec R_{\overline{\rho}}[\tfrac{1}{p}]$ corresponding to  de Rham representations with some fixed Hodge--Tate weights is dense in $\Spec R_{\overline{\rho}}$. Thus if 
  $\zeta_{\rho^{\univ}}$ takes values in $R_{\overline{\rho}}[\tfrac{1}{p}]$ then this would imply 
  that all the specialisations of $\rho^{\univ}$ have the same Hodge--Tate weights, 
  which would contradict \cite[Thm.\,D]{ghls}.
  \end{remar}

\subsection{Global input and applications}

We use Theorem \ref{main1} to study Hecke eigen\-spaces in the
completed cohomology. Our work is motivated by Emerton's ICM talk
\cite{emerton_icm} and his paper \cite{interpolate}. The theorem
\ref{mainglobal1} below requires a number of hypotheses to be
satisfied, most serious of which is the existence of Galois
representations attached to automorphic forms and satisfying a weak
form of local-global compatibility at $p$.  We explain in sections
\ref{sec_modular} to \ref{CS} that these hypotheses are satisfied for
modular curves, or more generally Shimura curves over totally real
fields, as well as in the setting of definite unitary groups over
totally real fields and of compact unitary Shimura varieties studied
by Caraiani--Scholze \cite{CS}.

Let $G$ be a connected reductive group over $\QQ$ and let $A$ be the
maximal split torus in the centre $Z$ of $G$.  Choose a maximal
compact subgroup $K_{\infty}$ of $G(\RR)$, as well as a sufficiently
small\footnote{See the discussion preceding lemma \ref{control_Gamma}
  for the precise hypotheses.} compact open subgroup
$K_f^p=\prod_{\ell\ne p} K_{\ell}$ of $G(\mathbb A^{p}_f)$, where
$K_{\ell}$ is a compact open subgroup of $G(\QQ_{\ell})$.  Let $S$ be
a finite set of prime numbers, containing $p$ and such that $G$ is
unramified over $\QQ_{\ell}$ and $K_{\ell}$ is hyperspecial for
$\ell\notin S$, and consider the universal spherical Hecke algebra
outside of $S$ (a polynomial ring over $\OO$ in infinitely many
variables)
\[\mathbb T^{\univ}=\bigotimes_{\ell\notin S} \mathcal H_{\ell},\]
where $\mathcal H_{\ell}=\OO[K_\ell\backslash G(\QQ_{\ell})/K_\ell]$
and the tensor product is taken over $\OO$.
                   
If $K_f$ is a compact open subgroup of $G(\mathbb A_f)$ let
\[\Ytilde(K_f)=G(\QQ)\backslash G(\mathbb A)/Z(\RR)^{\circ}
  K_{\infty}^{\circ} K_f,\] where $H^{\circ}$ denotes the neutral
connected component of the Lie group $H$. Let
\[\widetilde{H}^i=\varprojlim_{s} \varinjlim_{K_p} H^i(\Ytilde(K_f^p
  K_p), \OO/\varpi^s)\] be the associated completed cohomology groups,
the inductive limit being taken over compact open subgroups $K_p$ of
$G(\Qp)$.
     
The algebra $\mathbb T^{\univ}$ acts on the various $\widetilde{H}^i$
by spherical Hecke operators, and following \cite{emerton_icm} one can
define a suitable completion $\mathbb T$ of a quotient of
$\mathbb T^{\univ}$, which still acts on $\widetilde{H}^i$ for all
$i\ge 0$ (see the discussion preceding lemma \ref{neat} for the
precise definition). The algebra $\mathbb T$ is profinite and has only
finitely many open maximal ideals, but it is not known in this level
of generality whether it is Noetherian.

Let $x: \mathbb{T}[1/p]\rightarrow \Qpbar$ be a continuous
homomorphism of $\OO$-algebras with kernel $\mm_x$, such that the
image of $x$ is a finite extension of $\Qp$ (this condition is
satisfied if $\mathbb T$ is Noetherian). We conjecture that for all
$n\geq 0$ the $G(\Qp)$-representation
$$(\widetilde{H}^n\otimes_{\OO} L)[\mm_x]^{\la}\otimes_{\mathbb T,
  x}\Qpbar$$ has an infinitesimal character, and we give a conjectural
recipe for this character.  More precisely, motivated by
\cite[Conj.\,5.3.4]{beegee} we conjecture that one can associate to
$x$ an admissible Galois representation
\[ \rho : \Gal_{\QQ}\longrightarrow \CG_f(\Qpbar)\] which is
unramified outside of $S$ and such that for $\ell\notin S$ the
semisimplification of $\rho(\Frob_{\ell})$ matches the homomorphism
$x_{\ell}: \mathcal H_{\ell}\rightarrow \mathbb
T\overset{x}{\longrightarrow} \Qpbar$ via a suitable form of Satake
isomorphism for the $C$-group, defined in \cite{zhu}, see section
\ref{eyeswideshut} for the precise definitions.  Of course, there is
no reason for $\rho$ to be unique, but we prove that all such
representations $\rho$ (if they exist) have the same associated (by
the recipe described in the previous section) infinitesimal character
$\zeta^C_{\rho}$.  Let us note that if we fix an isomorphism
$\iota: \Qpbar \cong \C$ and if we suppose that
$\iota \circ x: \mathbb T\rightarrow \C$ is associated to a
$C$-algebraic automorphic form, our conjecture on the existence of
Galois representations becomes \cite[Conj.\,5.3.4]{beegee}.
  
\begin{conj}\label{dream_padic_intro}
  Let $\rho$ be an admissible representation associated to
  $x: \mathbb T[1/p] \rightarrow \Qpbar$ as in Conjecture
  \ref{dreamconj}, and let $n$ be a non-negative integer. Then
  $Z(\mathfrak{g})$ acts on
  $(\widetilde{H}^n\otimes_{\OO} L)[\mm_x]^{\la}\otimes_{\mathbb T,
    x}\Qpbar$ via $\zeta^C_{\rho}$.
\end{conj}

As it stands, the conjecture seems out of reach, but we prove it in a
certain number of cases, cf. theorem \ref{mainglobal1} below. We now
explain the extra assumptions one needs to make in our theorem. First,
we assume for simplicity of the exposition that $Z(\RR)/A(\RR)$ is
compact, cf. theorem \ref{lzero_central} for a version with a fixed
central character which allows one to suppress this hypothesis (this
is useful in practice, for instance in the case of Shimura curves over
totally real fields). Next, we choose an open maximal ideal $\mm$ and
we make the crucial hypothesis that $\mm$ is weakly non-Eisenstein,
i.e.  that there is an integer $q_0$ such that
\[H^i(\Ytilde(K^p_f K_p), \OO/\varpi)_{\mm}=0\] for all $i\ne q_0$ and
all sufficiently small compact open subgroups $K_p$ of
$G(\QQ_p)$. This assumption is automatically satisfied if
$G(\RR)/Z(\RR)$ is compact. When this quotient is not compact, the
assumption is rather strong, since one expects it to be satisfied only
when the rank of $G(\RR)$ is the same as the rank of
$Z(\RR) K_{\infty}$, i.e. the defect of the derived group of $G(\RR)$
is $0$, in which case $q_0$ is half the dimension of the symmetric
space $G(\RR)/Z(\RR)^{\circ}K_{\infty}^{\circ}$. As mentioned in
\cite{emerton_icm} (see also lemma \ref{need_this}) the hypothesis on
$\mm$ has very strong consequences, for instance the continuous
$\OO$-dual of $\widetilde{H}^n$ is finite projective over
$\OO[\![K_p]\!]$ for all sufficiently small compact open subgroups
$K_p$ of $G(\QQ_p)$, and $\mathbb{T}_{\mm}$ acts faithfully on this
dual.
     
Fix a supercuspidal type $\tau$, i.e.~a smooth absolutely irreducible
representation of $K_p$ on an $L$-vector space such that for any
smooth irreducible non supercuspidal $\bar{L}$-representation $\pi_p$
of $G(\QQ_p)$ we have ${\rm Hom}_{K_p} (\tau, \pi_p)=0$. If $V$ is an
irreducible algebraic representation of $G$ over $L$, let
$V(\tau)=V\otimes_L \tau$ and let $\mathbb T_{\mm, V(\tau)}$ be the
quotient of the localization $\mathbb T_{\mm}$ acting faithfully on
${\rm Hom}_{K_p}(V(\tau), \widetilde{H}^{q_0}_{\mm}[1/p])$. As in
\cite{emerton_icm}, one can (thanks to Franke's theorem \cite{franke})
associate to any $L$-algebra homomorphism
$x: \mathbb T_{\mm, V(\tau)}\to \Qpbar$ a cuspidal automorphic
representation $\pi_x=\otimes'_{v} \pi_{x,v}$ of $G(\mathbb A)$ such
that (among other things, cf. lemma \ref{TmVtau}) $\mathbb T_{\mm}$
acts on $\pi_x^{K_f^p}$ via $x$ and $\pi_{x,\infty}$ has the same
infinitesimal character as $V$, thus $\pi_x$ is cohomological and in
particular $C$-algebraic. Conjecture 5.3.4 of \cite{beegee} asserts
that there should be an admissible Galois representation
$\rho_x: \Gal_{\QQ}\to \CG(\Qpbar)$ attached to $\pi_x$. This
representation is not necessarily unique (for instance because $\pi_x$
is not necessarily unique), but one of the properties imposed on
$\rho_x$ is a relation between the Hodge--Tate cocharacter of $\rho_x$
at places above $p$ and the infinitesimal character of $\pi_x$ at the
archimedean places, as discussed in \cite[Rem.\,5.3.5]{beegee}.  In
particular $\zeta^C_{\rho, x}$ is well-defined, i.e.~independent of
the choice of $\pi_x$ or $\rho_x$. Our main result is (we keep the
previous hypotheses):
     
\begin{thm}\label{mainglobal1}
  We assume that the following hold:
  \begin{itemize}
  \item[(i)] $\mathbb T_{\mm}$ is noetherian and $\mm$ is weakly
    non-Eisenstein;
  \item[(ii)] there is an admissible representation
    $\rho: \Gal_{\QQ}\rightarrow \CG_f(\mathbb T_{\mm}^{\rig})$, such
    that for all $V \in \Irr_{G}(L)$ and all
    $x: \mathbb T_{\mm, V(\tau)}[1/p]\rightarrow \Qpbar$, the
    specialisation of $\rho$ at $x$ matches $\pi_x$ according
    \cite[Conj.\,5.3.4]{beegee};
  \item[(iii)] the composition $d\circ \rho$ is equal to the $p$-adic
    cyclotomic character.
  \end{itemize}
  Then for all $y\in \mSpec \mathbb T_{\mm}[1/p]$ the centre
  $Z(\mathfrak g)$ acts on
  $(\widetilde{H}^{q_0}_{\mm} \otimes_{\OO} L)[\mm_y]^{\la}$ by
  $\zeta^C_{\rho, y}$.
\end{thm}
  
We apply this theorem in the settings of modular curves, definite
unitary groups over totally real fields, compact unitary Shimura
varieties, studied by Caraiani--Scholze, and we apply a version of the
theorem with a fixed central character in the setting of Shimura
curves.  We note that Lue Pan has proved a version of the theorem
above in the setting of modular curves in \cite[prop.\,6.1.5]{LuePan}
using the geometry of the Hodge--Tate period map and relative Sen
theory (as opposed to the classical one, used in this paper).
    
\subsection{$p$-adic Jacquet--Langlands and infinitesimal characters}
With a view towards the $p$-adic Jacquet--Langlands correspondence,
one particularly interesting example is that of a quaternion algebra
$D$ over $\mathbb Q$ split at $\infty$, and the first completed
cohomology group of the tower of Shimura curves associated to $D$ (of
a fixed tame level), localised at a maximal ideal corresponding to an
absolutely irreducible $2$-dimensional mod $p$ Galois representation
of $\Gal(\overline{\mathbb{Q}}/ \mathbb Q)$. If the ideal $\mm_x$ in
the theorem corresponds to a $p$-adic Galois representation which is
not Hodge--Tate at $p$ then $\zeta^C_{\rho, x}$ cannot be an
infinitesimal character of an irreducible finite dimensional
$U(\slt)_L$-module. This implies that if $K$ is a compact open
subgroup of $(D\otimes \Qp)^\times$ then
$(\widetilde{H}^1_{\mm} \otimes L)[\mm_x]$ does not have a finite
dimensional $K$-invariant subquotient.  By a different argument (see
\cite{Ludwig} when locally at $p$ the residual Galois representation
is reducible and generic and \cite{DPS2} for the general case) we can
bound the Gelfand--Kirilov dimension of this Banach space
representation by $1$ and, putting the two ingredients together,
conclude that $(\widetilde{H}^1_{\mm} \otimes L)[\mm_x]$ is of finite
length as a Banach space representation of $K$. This is interesting
because there should be a $p$-adic Jacquet--Langlands correspondence
realised by the completed cohomology, and the result suggests that the
objects on the division algebra side should be (some) admissible
unitary $L$-Banach space representations of $(D\otimes\Qp)^\times$ of
finite length and of Gelfand--Kirilov dimension $1$. The case when the
$p$-adic Galois representation is Hodge--Tate at $p$ seems quite a bit
more involved, in particular we do not know how to prove that the
corresponding Banach representation has finite length.

\subsection{$p$-adic Langlands and infinitesimal characters}
In the applications discussed so far we applied theorem \ref{main1}
with $h=0$ in part (i), since the completed cohomology is
admissible. However, we may also apply it to the modules coming out of
Taylor--Wiles patching as done in \cite{6auth}, as we will now
explain.  Let $F$ be a finite extension of $\Qp$ and let $n\geq1$ be
such that $p\nmid 2n$. Consider a mod $p$ Galois representation
$\rhobar : \Gal_F\rightarrow\GL_n(k)$, with universal framed
deformation
$\rho^\Box: \Gal_F\rightarrow\GL_n(R_{\overline{\rho}}^\Box)$.  Let
$R_{\infty}$ be the patched deformation ring, which is a flat
$R_{\overline{\rho}}^{\Box}$-algebra, and let
$\rho:\Gal_F \rightarrow \GL_n(R_{\infty})$ be the Galois
representation obtained from $\rho^\Box$ by extending scalars to
$R_{\infty}$.

Let $K=\GL_n(\mathcal{O}_F)$ and let $M_\infty$ be the compact
$R_\infty[\![K]\!]$-module with a compatible $R_\infty$-linear action
of $G$ constructed in \cite{6auth} by patching automorphic forms on
definite unitary groups. Let
$\Pi_{\infty}\coloneqq \Hom_{\OO}^{\cont}(M_{\infty}, L)$. Any
$y\in \mSpec R_{\infty}[1/p]$ gives rise to a Galois representation
$\rho_x: \Gal_F \rightarrow \GL_n(\kappa(x))$, where $x$ is the image
of $y$ in $\mSpec R_{\overline{\rho}}^\Box[1/p]$, as well as to an
admissible unitary $\kappa(y)$-Banach space representation
$\Pi_{\infty}[\mm_y]$ of $G\coloneqq \GL_n(F)$. It is expected but not
known beyond the case $n=2$ and $F=\Qp$, that $\Pi_{\infty}[\mm_y]$
depends only on $\rho_x$, more precisely that $\Pi_{\infty}[\mm_y]$
and $\rho_x$ should be related by the hypothetical $p$-adic Langlands
correspondence, see \cite[\S 6]{6auth}. Theorem \ref{patched1} below
shows that the infinitesimal character of $\Pi_{\infty}[\mm_y]^{\la}$
depends only on $\rho_x$, thus adding nontrivial evidence that the
expectation should be true.  Using a twisting element we may associate
to $\rho$ an admissible Galois representation $\rho^C$ with values in
the $C$-group of $\GL_n$, see lemma \ref{rhoC} for more details.

\begin{thm}\label{patched1}
  Let $\mathfrak{g}$ be the $\Qp$-linear Lie algebra of $G$. The
  algebra $Z(\mathfrak{g})$ acts on $\Pi_{\infty}[\mm_y]^{\la}$
  through the character $\zeta_{\rho^C_x}^C$.
\end{thm}
If $F=\Qp$ and $n=2$ then the $p$-adic Langlands correspondence
$\rho\mapsto \Pi(\rho)$, envisioned by Breuil, has been established by
local means in \cite{CDP}, \cite{image}, building on the monumental
work of Colmez \cite{colmez}, and the infinitesimal character of
$\Pi(\rho)^{\la}$ has been calculated by one of us (G.D.) in
\cite{GDinf} in terms of the characteristic polynomial of the Sen
operator of $\rho$. We give a new proof of this result in section
\ref{padicLL}, using the patched module and recent results of Tung
\cite{Tung1}, \cite{Tung2}.

\subsection{Beyond $l_0=0$ case}
We end this long introduction by noting that one cannot hope to prove
the conjecture \ref{dream_padic_intro} by applying theorems
\ref{infinitesimal} and \ref{infinitesimal_central} to the completed
cohomology directly, since the localisation of completed homology
$\widetilde{H}_n$ at a maximal ideal of the Hecke algebra is not
expected to be projective over $\OO[\![K_p]\!]$. However, one might
hope to be able to apply our results to the patched homology groups
obtained via the patching method of Calegary--Geraghty
\cite{cale-ger}. The most accessible case, when weakly non-Eisenstein
maximal ideal are not expected to exist, is when $G= \PGL_2$ over a
quadratic imaginary field $F$, such that $p$ splits completely in $F$,
studied by Gee--Newton in \cite{gee-newton}. It follows from
\cite[Prop.\,5.3.1]{gee-newton} and its proof that under the
assumptions made there the patched homology satisfies the conditions
of Theorem \ref{infinitesimal}. We do not pursue the matter here,
since in this specific setting instead of applying Theorem
\ref{infinitesimal} it might be easier to use local--global
compatibility at $p$ and appeal to the results on the infinitesimal
character in the $p$-adic Langlands correspondence for $\GL_2(\Qp)$,
see Theorem \ref{tung}, but hope to come back to these questions in
future work.

\subsection{Overview of sections} We will review the content of each section. 

In section \ref{Lgroup} we review $L$-groups, $C$-groups, admissible
Galois representations and introduce their variants twisted by a
torsor in subsection \ref{sec_balaur}.

In section \ref{sec_Sen_ban} we review Sen theory for Galois
representations valued in $\GL_n$ of a Banach algebra.

In section \ref{families_gal} we define the characters $\zeta_{\rho}$
and $\zeta_{\rho}^C$. The main difficulty in this section is to come
up with the right definition, when $G$ is not split over $F$. At first
reading one should ignore that and assume that $G$ is split over $F$
and the torsor $P$ is trivial.

In section \ref{sec_HT} we relate $\zeta_{\rho}$ and $\zeta_{\rho}^C$
to the Hodge--Tate cocharacter, when $\rho$ is Hodge--Tate. We also
explain how to attach algebraic representations to $\rho$, when the
Hodge--Tate weights are regular in a setting, when $G$ is (possibly)
not split over $F$. In subsection \ref{sec_auto_reps} we check that if
$\rho$ is a Galois representation associated to a cohomological
automorphic representation $\pi$ then $\zeta_{\rho}^C$ is essentially
the infinitesimal character of the archimedean component of $\pi$.

In section \ref{FibreCM} we prove a commutative algebra result, which
is used in section \ref{sec_fam_Ban}.

In section \ref{sec_density_algebraic} we prove a density theorem,
which is used in section \ref{sec_fam_Ban}.

In section \ref{sec_fam_Ban} we prove the existence of infinitesimal
character in an abstract setting, modelled on the representations that
appear in completed cohomology, when the mod $p$ Hecke eigenvalues
contribute to only one cohomological degree, and its patched version.

In section \ref{sec_global} we first review completed cohomology
following \cite{emerton_icm}.  However, we do not make the assumption
that the maximal $\QQ$-split torus and the maximal $\R$-split torus in
the centre of $G$ coincide, since we want to handle Shimura
curves. The main result is proved in subsection \ref{main_result} by
showing that we may apply the results of section \ref{sec_fam_Ban} if
certain conditions are satisfied. We then check that these conditions
are satisfied in the settings of modular curves, Shimura curves,
definite unitary groups, compact unitary Shimura varieties by
appealing to the results already available in the literature.  In
subsection \ref{sec_patch} we show that we may apply the results of
section \ref{sec_fam_Ban} in the setting of the patched module defined
in \cite{6auth}. In subsection \ref{padicLL} we reprove results on the
infinitesimal character in the $p$-adic local Langlands correspondence
for $\GL_2(\Qp)$.  In subsection \ref{eyeswideshut} we formulate a
general conjecture regarding the action of the center of the universal
enveloping algebra on the locally analytic vectors in the Hecke
eigenspaces in completed cohomology.

\subsection{Acknowledgements} We would like to acknowledge the
influence of the ideas of Michael Harris, regarding the connections
between the representation theory of real Lie groups and the $p$-adic
representation theory of $p$-adic groups. In particular, the question
about the infinitesimal character was raised by him in 2003 in Luminy,
where VP was a participant.  We would like to thank Ana Caraiani,
Matthew Emerton, Toby Gee, Florian Herzig, David Loeffler and Timo
Richarz for the stimulating correspondence regarding various aspects
of this paper. VP would like to thank especially his colleague Jochen
Heinloth for several illuminating discussions regarding geometric
group theory. Parts of the paper were written, when VP visited the
Morning Side Center for Mathematics in Beijing in March 2019.  VP
would like to thank Yongquan Hu for the invitation and stimulating
discussions and the Morning Side Center for Mathematics for providing
excellent working conditions.  GD and BS would like to thank
Christophe Breuil, Laurent Clozel, Pierre Colmez and Olivier Ta\"ibi
for useful discussions on the topics of this paper. BS would like to acknowledge the support of the A.N.R.~through the project CLap-CLap ANR-18-CE40-0026.
We would like to 
thank Toby Gee and James Newton for their comments on the earlier version of 
the paper. 

%BS is member of the A.N.R.\ project CLap-CLap ANR-18-CE40-0026.

\section{The $L$-group}\label{Lgroup}

We fix algebraic closure $\Qpbar$ of $\Qp$ and a finite extension $F$
of $\Qp$ contained in $\Qpbar$.  Let $G$ be a connected reductive
group defined over $F$. Let $E\subset \Qpbar$ be a finite Galois
extension of $F$ such that $G$ splits over $E$.  We follow \cite[\S
I.2]{borel_corvalis}, \cite[\S 2.1]{beegee} and \cite{zhu} to define
the dual group $\Ghat$, the $L$-group $\LG$ and the $C$-group
$\CG$. We discuss admissible representations of
$\Gal_F\coloneqq \Gal(\Qpbar/F)$ with values in these groups. In the
last subsection we consider versions of these constructions twisted
along a torsor.

We fix a maximal split torus $T$ of $G$ defined over $E$. Let
$X= X^*(T)$ be the group of characters of $T$, $\Phi$ the set of
roots, $X^{\vee}=X_*(T)$ the group of cocharacters of $T$ and
$\Phi^{\vee}$ the set of coroots. We denote the natural pairing
$X\times X^{\vee} \rightarrow \ZZ$ by $\langle \cdot, \cdot\rangle$.
The $4$-tuple $R(G,T)=(X, \Phi, X^{\vee}, \Phi^{\vee})$ together with
the pairing $\langle \cdot, \cdot\rangle$ is a reduced \textit{root
  datum} in the sense of \cite[Def.\,1.3.3]{bcnrd} by
\cite[Prop.\,5.1.6]{bcnrd}.

We further fix a Borel subgroup $B$ of $G$ defined over $E$ and
containing $T$.  To the triple $(G_E, T, B)$ one may attach a
\textit{based root datum}, which is a $6$-tuple
$(X, \Phi, \Delta, X^{\vee}, \Phi^{\vee}, \Delta^{\vee})$, where
$\Delta$ is the set of positive simple roots corresponding to $B$,
$\Delta^{\vee}$ is the set of coroots $a^{\vee}$ for $a\in \Delta$,
see \cite[\S 1.5]{bcnrd}.

Let $\Aut(G)$ be the group of automorphisms of $G$ as an algebraic
group over $E$. Then $\Aut(G)$ acts on the set of such pairs $(B, T)$
as above. If $\phi\in \Aut(G)$ then the pairs $(\phi(B),\phi(T))$ and
$(B, T)$ are conjugate by an element $g_\phi\in G(E)$, which is
uniquely determined by $\phi$ up to an element of $T(E)$, see
\cite[Prop.\,6.2.11 (2)]{bcnrd} and its proof.  This induces a group
homomorphism $\Aut(G)\rightarrow \Aut(R(G, T), \Delta)$, defined by
$\phi\mapsto\Ad(g_\phi)\circ \phi$, where $\Aut(R(G, T), \Delta)$ is
the subgroup of the group of automorphisms of the root data $R(G,T)$,
consisting of those automorphisms, which map $\Delta$ to itself. It
follows from \cite[Prop. 7.1.6]{bcnrd} that this induces an exact
sequence:
\begin{equation}\label{exact}
  1 \rightarrow (G/Z_G)(E) \rightarrow \Aut(G)\rightarrow \Aut(R(G,T), \Delta)\rightarrow 1,
\end{equation}
where $Z_G$ is the centre of $G$ and the first non-trivial arrow is
given by conjugation. Let $\Aut(G, T, B)$ be the set of
$\phi\in \Aut(G)$ such that $\phi(T)=T$ and $\phi(B)=B$.  If
$\phi\in \Aut(G, T, B)$ then
$\phi(a)\coloneqq a\circ\phi^{-1}\in \Delta$ for all $a\in \Delta$ and
$\phi$ induces an isomorphism $\phi: U_a \rightarrow U_{\phi(a)}$ for
all $a\in \Delta$, where $U_a$ is the root subgroup corresponding to
$a\in \Delta$. For each $a\in \Delta$ we fix an isomorphism
$p_a: \mathbb{G}_a\overset{\cong}{\longrightarrow}U_{a}$ over $L$. The
data $\{p_a\}_{a\in \Delta}$ is called a \textit{pinning} of
$(G, T, B)$. The group $(T/Z_G)(E)$ acts simply transitively on the
set of pinnings of $(G, T,B)$.  Let
$\Aut(G, T, B, \{p_a\}_{a\in \Delta})$ be the set of
$\phi\in \Aut(G, T, B)$ such that $p_{\phi(a)} = \phi\circ p_a$ for
all $a\in \Delta$. It is shown in \cite[Prop.\,7.1.6]{bcnrd} that
\eqref{exact} induces an isomorphism
$\Aut(G, T, B, \{p_a\}_{a\in \Delta}) \overset{\cong}{\rightarrow}
\Aut(R(G,T), \Delta)$.  Thus if we fix a pinning then there is a
natural isomorphism:
\begin{equation}\label{splitting}
  \Aut(G)\cong (G/Z_G)(E) \rtimes \Aut(R(G,T), \Delta).
\end{equation}
 
By \cite[Thm.\,6.1.17]{bcnrd} there is a unique (up to unique
$\ZZ$-isomorphism) pinned split reductive group
$(\Ghat, \That, \Bhat, \{p'_{a^{\vee}}\}_{a^{\vee}\in \Delta^{\vee}})$
over $\ZZ$ such that its based root datum
$(R(\Ghat, \That), \Delta^{\vee})$ is isomorphic to
$(X^{\vee}, \Phi^{\vee}, \Delta^{\vee}, X, \Phi, \Delta)$, which is
the based root datum dual to $(R(G, T), \Delta)$.  We will refer to
$\Ghat$ as the dual group.
  
We will now give a linear algebra definition of the action
\begin{equation}\label{muG}
  \mu_G: \Gal_F\rightarrow \Aut(R(G,T), \Delta),
\end{equation}
defined in \cite[\S I.1.3]{borel_corvalis}. Let $\mathfrak g$ be the
Lie algebra of $G$ as an algebraic group over $F$. Then
$\mathfrak g_E\coloneqq \mathfrak g \otimes_F E$ is the Lie algebra of
$G_E$. For $\gamma\in \Gal_F$ let
$\gamma: \mathfrak g_E \rightarrow \mathfrak g_E$ be the map
$x\otimes \lambda \mapsto x\otimes \gamma(\lambda)$ for
$x\in \mathfrak g$ and $\lambda\in E$. Let
$\mathfrak{t}\subset \mathfrak{b} \subset \mathfrak{g}_E$ be the Lie
algebras of $T$ and $B$ respectively and let $\Ad$ denote the action
of $G$ on $\mathfrak g$ by conjugation. Then $\gamma(\mathfrak t)$ is
a Cartan subalgebra of $\mathfrak{g}_E$ and $\gamma(\mathfrak b)$ is a
Borel subalgebra of $\mathfrak g_E$.  Thus there is $g\in (G/Z_G)(E)$
such that $\Ad(g)(\gamma(\mathfrak t))= \mathfrak t$ and
$\Ad(g)(\gamma(\mathfrak b))= \mathfrak b$. The map
$\Ad(g) \circ \gamma$ maps the set of $\mathfrak t$-eigenspaces for
the adjoint action on $\mathfrak b$ to itself, thus there is a unique
$\sigma\in \Aut(R(G,T), \Delta)$ such that
$\Ad(g)(\gamma(\mathfrak g_{E, \alpha}))= \mathfrak g_{E,
  \sigma(\alpha)}$ for all $\alpha\in \Delta$. To show that the map
$\gamma\mapsto \sigma$ is a group homomorphism it is enough to check
that $\sigma$ does not depend on the choice of $g$. This is the case,
because any two choices differ by an element of $T(E)$, which does not
change the eigenspaces. It follows from the construction that
\eqref{muG} factors through $\Gal(E/F)$.  By using the splitting in
\eqref{splitting} we obtain a group homomorphism
$\Gal_F\rightarrow \Aut(G)$, which induces an $E$-linear action of
$\Gal_F$ on $\mathfrak g_E$.

The automorphism groups of $(R(G, T),\Delta)$ and
$(R(\Ghat, \That), \Delta^{\vee})$ are canonically isomorphic. Using
the pinning and \cite[Thm.\,7.1.9 (3)]{bcnrd} we obtain a group
homomorphism $\Gal_F\rightarrow \Aut(\Ghat)$.  We note that the
resulting action of $\Gal_F$ on $\Ghat$ is defined over $\ZZ$.  The
$\Ghat/ Z_{\Ghat}$-conjugacy class of this homomorphism is canonical
and depends only on $G$. We define the $L$-group of $G$ as a
semidirect product:
\[\LG\coloneqq  \Ghat\rtimes \Gal_F.\]
In particular, $\LG$ is a split reductive group over $\ZZ$ with
identity component $\Ghat$ and the component group $\Gal_F$.  The map
$g\mapsto(1, g)$ identifies $\Ker \mu_G$ with a normal subgroup
$\LG$. We let $\LGf\coloneqq \LG/ \Ker \mu_G$.

\subsection{$C$-groups}\label{sec_Cgr_sub}
We follow the exposition in \cite{zhu}. Let $\delta$ be a half sum of
positive roots, thus $2\delta\in X^*(T)= X_*(\That)$. Since
$\Ghat/Z_{\Ghat}$ is of adjoint type there is a unique
$\delta_{\mathrm{ad}}\in X_*(\That/Z_{\Ghat})$, such that the image of
$2\delta$ in $X_*(\That/Z_{\Ghat})$ is equal to
$2\delta_{\mathrm{ad}}$. We define an action of $\Gm$ on $\Ghat$ by
\[ \Ad \delta_{\mathrm{ad}}: \Gm
  \overset{\delta_{\mathrm{ad}}}{\longrightarrow} \That/Z_{\Ghat}
  \overset{\Ad}{\longrightarrow} \Aut(\Ghat).\] Since
$\delta_{\mathrm{ad}}$ is $\Gal_F$-invariant, the action of $\Gm$
commutes with the action of $\Gal_F$. We let
\[\Ghat^T\coloneqq \Ghat \rtimes \Gm, \quad \CG\coloneqq \Ghat^T
  \rtimes \Gal_F\cong \Ghat \rtimes( \Gm \times \Gal_F).\] Let
$\CG_f\coloneqq \CG/\Ker \mu_G$ and let $d: \CG\rightarrow \Gm$ denote
the projection map. Then we have an exact sequence of group schemes
over $\ZZ$:
\begin{equation}\label{LGCG}
  1\rightarrow \LG \rightarrow \CG \overset{d}{\rightarrow}\Gm \rightarrow 1
\end{equation}
and a similar sequence with $\LG_f$ and $\CG_f$.

Following Remark 1 in \cite{zhu}, we may regard $\Ghat^T$ as the dual
group of of a reductive group $G^T$, which is a central extension of
$G$ by $\Gm$ over $F$, and then regard $\CG= \leftidx{^L}G^T$ as the
usual Langlands dual group of $G^T$.  This is the definition given in
\cite[Def.\,5.3.2]{beegee}.

\subsection{Admissible representations}\label{adm_reps}
Let $N$ be a topological group and let $\Aut(N)$ be the group of
continuous automorphisms of $N$, which we equip with discrete
topology. Let $\theta: \Gal_F \rightarrow \Aut(N)$ be a continuous
group homomorphism. We may then form the semi-direct product
$N \rtimes \Gal_F$ and let $\pi: N\rtimes \Gal_F \rightarrow \Gal_F$
denote the projection.
\begin{defi} A continuous representation
  $\rho: \Gal_F \rightarrow N\rtimes \Gal_F$ is admissible if
  $\pi \circ \rho$ induces the identity on $\Gal_F$. Two admissible
  representations $\rho$, $\rho'$ are equivalent if there is $n\in N$
  such that $\rho'(\gamma)= n \rho(\gamma) n^{-1}$, for all
  $\gamma\in \Gal_F$.
\end{defi}

Let $Z^1_{\cont}(\Gal_F, N)$ be the set of continuous functions
$c: \Gal_F \rightarrow N$, $\gamma\mapsto c_\gamma$ satisfying the
cocycle condition $c_{\gamma \beta}= c_\gamma( \gamma \cdot c_\beta)$,
where $\cdot$ denotes the action of $\Gal_F$ on $N$. Two cocycles $c$
and $c'$ are cohomologous if there is $n\in N$ such that
$c'_{\gamma}=n c_\gamma (\gamma\cdot n^{-1})$ for all
$\gamma \in \Gal_F$. This relation is an equivalence relation and we
denote the set of equivalence classes by $H^1_{\cont}(\Gal_F, N)$.

\begin{lem}\label{cocycle} The rule 
  $\rho(\gamma)= (c_{\gamma}, \gamma)$ defines a bijection between the
  set of admissible representations
  $\rho: \Gal_F\rightarrow N\rtimes \Gal_F$ and
  $Z^1_{\cont}(\Gal_F, N)$, which induces a bijection between the
  respective equivalence classes.
\end{lem}
\begin{proof} Let $\rho: \Gal_F\rightarrow N\rtimes \Gal_F$ be an
  admissible representation. Then we may write
  $\rho(\gamma)= (c_\gamma, \gamma)$ with $c_{\gamma} \in N$ for all
  $\gamma\in \Gal_F$. The condition
  $\rho(\gamma\beta)=\rho(\gamma) \rho(\beta)$ translates into the
  cocycle condition. Conversely, if $c\in Z^1_{\cont}(\Gal_F, N)$ then
  $\rho(\gamma)\coloneqq (c_{\gamma}, \gamma)$ defines an admissible
  representation $\Gal_F\rightarrow N\rtimes \Gal_F$.  If $n\in N$
  then
  $(n, 1)(c_{\gamma}, \gamma)(1, n)^{-1}= ( n c_{\gamma} (\gamma\cdot
  n^{-1}), g)$. Thus two admissible representations are equivalent if
  and only if the corresponding cocycles are cohomologous.
\end{proof}
 
If $A$ is a topological algebra then the topology on $A$ induces a
topology on $\Ghat(A)$ and $\LG(A)= \Ghat(A)\rtimes \Gal_F$, where the
topology on $\Gal_F$ is the Krull topology. Thus we may apply the
above discussion to $N=\Ghat(A)$. The description of admissible
representations in terms of cocycles shows that admissible
representations with values in $\LG_f(A)$ coincide with admissible
representations with values in $\LG(A)$. We will exploit that $\LG_f$
is of finite type over $\ZZ$. The same discussion applies to $\CG(A)$
and $\CG_f(A)$.
  
\subsection{Twisted families of $L$- and
  $C$-parameters}\label{sec_balaur}
We want to allow more general families of Galois representations than
considered in the previous section. A motivating example appears in
Chenevier's thesis \cite{chenevier_thesis}, where he constructs an
eigenvariety $X$ for a unitary group compact at $\infty$ and split at
$p$ associated to a division algebra over a quadratic extension $E$ of
$\QQ$, and a pseudo-representation $t: \Gal_E\rightarrow \OO_X(X)$,
which at points corresponding to classical automorphic forms
interpolates the traces of corresponding Galois representations. The
locus $X_{\mathrm{irr}}$, where $t$ is absolutely irreducible, is an
open rigid subvariety of $X$ and Chenevier shows that over this locus,
$t$ gives rise to a representation
$\rho: \Gal_E\rightarrow \mathcal A^*$, where $\mathcal A$ is an
Azumaya algebra over $X_{\mathrm{irr}}$, see
\cite[Thm.\,E]{chenevier_thesis}. A similar type of example arising in
the deformation theory of pseudo-representations appears in
\cite[\S4.2]{chenevier_det}.  Although, in these cases
$\Ghat= \LG_f=\GL_n$, the Galois representation takes values not in
$\GL_n(\OO_X(X_{\mathrm{irr}}))$, but in an inner form of it. We would
like our setup to cover this example and also the situations where the
action of $\Gal_F$ on $\Ghat$ is non-trivial. In this more general
setting it is not obvious what \textit{admissibility} for a Galois
representation should mean.
 
We keep the discussion deliberately very general, since it applies in
different contexts. Let $X$ be a space with a sheaf of topological
rings $\OO_X$ for some Grothendieck topology on $X$. (In the
application, ''space" will mean a rigid analytic space over $\Qp$, but
it could also be, for example, schemes over $\mathbb F_\ell$, rigid
spaces over $\mathbb{Q}_\ell$, or differentiable manifolds.) Let $H$
be a group over $X$ and let $P$ be an $H$-torsor, by which we mean a
space $P\rightarrow X$, together with the left $H$-action
$H\times_X P\rightarrow P$, such that the map
\[ H\times_X P \rightarrow P\times_X P, \quad (g,p)\mapsto (gp, p)\]
is an isomorphism. We require in addition that there is a family of
local sections $s_i:U_i\rightarrow P$ for some open cover
$\mathcal U=\{U_i\}_{i\in I}$ of $X$. We note that this implies that
the map $H\times U_i \rightarrow P\times_X U_i$ is an isomorphism, so
that the restriction of $P$ to $U_i$ is a trivial
$H$-torsor. Following \cite{breen_messing}, we let $\Pad$ be the sheaf
of $H$-equivariant automorphisms of $P$, which leave the base $X$
fixed.
 
Recall that $\LG_f= \Ghat \rtimes \Gamma$, where
$\Gamma= \Gal_F/\Ker \mu_G$. Since $\Ghat$ and the action
$\mu_G: \Gal_F \rightarrow \Aut(\Ghat)$ are defined over $\ZZ$, we may
view $\Ghat$, $\LG_f$, $\Gamma$ as groups over $X$, so that
$\Ghat_X(U)=\Ghat(\OO_X(U))$. We have an isomorphism of groups over
$X$:
\begin{equation}\label{semidirect}
  \LG_{f, X} \cong \Ghat_X \rtimes \Gamma_X.
\end{equation}
Let $P$ be a $\Ghat$-torsor and let
$P_f\coloneqq \LG_{f, X} \times^{\Ghat_X} P$. It follows from
\eqref{semidirect} that the quotient $\Ghat_X \backslash P_f$ is a
$\Gamma_X$-torsor on $X$ together with a global trivialisation
$t: \Ghat_X \backslash P_f \cong \Gamma_X$.

\begin{lem} We have an exact sequence of sheaves of groups on $X$:
  \[ 1\rightarrow \Pad \rightarrow P_f^{\mathrm{ad}} \rightarrow
    (\Ghat_X \backslash P_f)^{\mathrm{ad}}\rightarrow 1, \] where we
  consider $P$ as a $\Ghat_X$-torsor, $P_f$ as an $\LG_{f, X}$-torsor,
  $\Ghat_X \backslash P_f$ as a $\Gamma_X$-torsor.
\end{lem}
\begin{proof} If $P$ is a trivial torsor then the assertion follows
  from \eqref{semidirect}. In general, one may reduce to this case by
  choosing a covering which trivialises $P$.
\end{proof}
By taking global sections we get an exact sequence of pointed sets
\begin{equation}\label{pointed}
  1\rightarrow \Pad(X)\rightarrow P_f^{\mathrm{ad}}(X) \overset{\pi}{\longrightarrow} \Gamma_X(X) \rightarrow H^1(X, \Pad),
\end{equation}
where we used the trivialisation $t$ to identify $\Gamma_X(X)$ with
$(\Ghat_X \backslash P_f)^{\mathrm{ad}}(X)$.  There is an injective
group homomorphism
$\Gamma\hookrightarrow \Gamma_X(X), \gamma \mapsto r_{\gamma}$, where
$r_{\gamma}( (\beta, x))= (\beta\gamma^{-1}, x)$.
 
\begin{defi}\label{twistL} Let $P$ be a $\Ghat_X$-torsor over $X$. We
  define the $L$-group of $G$ twisted by $P$ to be
  \[\LPad(X)\coloneqq \{ \varphi\in P_f^{\mathrm{ad}}(X): \pi(\varphi)
    \in \Gamma \subset \Gamma_X(X)\},\] where $\pi$ is the projection
  in \eqref{pointed}, equipped with the topology as explained below.
\end{defi}
We have an exact sequence of groups
\begin{equation}
  1\rightarrow \Pad(X)\rightarrow \LPad(X)\overset{\pi}{\longrightarrow} \Gamma.
\end{equation}
Since $\Gamma$ is finite it is enough to define a topology on
$\Pad(X)$. If $P$ is a trivial $\Ghat_X$-torsor over $X$ then we
topologize $\Pad(X)$ by identifying it with $\Ghat(\OO_X(X))$.  Any
two such identifications differ by a conjugation by an element of
$\Ghat(\OO_X(X))$, hence the topology on $\Pad(X)$ does not depend on
a choice of section.  In general, we choose a covering
$\mathcal U=\{U_i\}_{i\in I}$ of $X$ trivialising $P$, topologize
$\Pad(U_i)$ as above and use the sheaf property to put a subspace
topology on $\Pad(X)$. In Remark \ref{torsor_top} we show that if $X$
is a rigid analytic space over $\Qp$ and $P$ is a torsor on $X$ for
the \'etale or the analytic topology then the topology on $\Pad(X)$
does not depend on a choice of a covering.
 
\begin{examp} If $P=\Ghat_X$ then
  $\LPad(X)=\LG_f(\OO_X(X))= \Ghat(\OO_X(X))\rtimes \Gamma$.
\end{examp}
\begin{remar} One may also describe $\LPad(X)$ as the set of pairs
  $(\varphi, \gamma)\in P_f^{\mathrm{ad}}(X) \times \Gamma$, such that
  the diagram
  \[ \xymatrix@1{ P_f\ar[r]\ar[d]_-{\varphi} & \Ghat_X \backslash P_f\ar[d]^-{r_{\gamma}}\\
      P_f\ar[r] & \Ghat_X \backslash P_f}\] commutes.
\end{remar}
\begin{defi}\label{defi_twisted} Let $P$ be $\Ghat_X$-torsor over
  $X$. A continuous representation
  \[ \rho: \Gal_F \rightarrow \LPad(X)\] is admissible, if
  $\pi(\rho(\gamma))= \gamma (\Ker \mu_G)$ for all $\gamma\in \Gal_F$.
 
 Two admissible representations $\rho$ and $\rho'$ are equivalent if
 there is $\varphi\in \Pad(X)$ such that
 $\rho'(\gamma) = \varphi \circ \rho(\gamma)\circ \varphi^{-1}$ for
 all $\gamma \in \Gamma$.
\end{defi}
 
\begin{remar} If $P=\Ghat_X$ then we recover the definition of
  admissible representations and their equivalence classes for Galois
  representations valued in $\LG_f(\OO_X(X))$.
\end{remar}
 
\begin{remar} For a general $P$ the map
  $\pi: \LPad(X)\rightarrow \Gamma$ might not be surjective; the
  obstruction will lie in $H^1(X, \Pad)$. In this case, there will be
  no admissible representations of $\Gal_F$ with values in
  $\LPad(X)$. However, we expect that as the theory develops examples
  of admissible representations with values in $\LPad(X)$ will arise
  in the context of eigenvarieties, similar to
  \cite[Thm.\,E]{chenevier_thesis}.
\end{remar}
\begin{defi}\label{base_change} If $Y\rightarrow X$ is a map of spaces
  and $\rho: \Gal_F \rightarrow \LPad(X)$
  is an admissible representation then we define an admissible representation
  \[\rho_Y: \Gal_F\rightarrow \LPad(Y)\]
  by composing $\rho$ with the natural map
  $P_f^{\mathrm{ad}}(X)\rightarrow (P_{f}\times_X
  Y)^{\mathrm{ad}}(Y)$.
\end{defi}
 
\begin{lem}\label{fairwell_balaur} Let
  $\mathcal{U}=\{U_i\rightarrow X\}_{i\in I}$ be a cover of $X$
  together with local sections $\{s_i: U_i\rightarrow P\}_{i\in
    I}$. Let $g_{ij} \in \Ghat(\OO_X(U_{ij}))$ be such that
  $s_i= g_{ij} s_j$ for all $i, j\in I$. Then to give an admissible
  representation $\rho: \Gal_F \rightarrow \LPad(X)$ is equivalent to
  giving a family of admissible representations
\[\rho_{U_i}: \Gal_F\rightarrow \LG(\OO_X(U_i)), \quad \forall i\in I,\]
such that 
\begin{equation}\label{transform_balaur}
  \rho_{U_i}(\gamma)= g_{ij} \rho_{U_j}(\gamma) g_{ij}^{-1},  \quad \forall \gamma\in \Gal_F, \quad \forall i,j\in I,
\end{equation} 
where the equality takes place in $\LG(\OO_X(U_{ij}))$.
\end{lem}
\begin{proof} We note that the cover $\mathcal U$ also trivialises
  $P_f$ and the gluing data for $P_f$ is also given by
  $\{g_{ij}\}_{i, j\in I}$. The assertion follows from the description
  of $P^{\mathrm{ad}}_f(X)$ in terms of the glueing data in \cite[\S
  1, (1.1.4)]{breen_messing}. Let us just indicate how to obtain
  $\rho_{U_i}$ from $\rho$ and vice versa, leaving the rest of the
  details for the interested reader.
 
  If $u_i\in P^{\mathrm{ad}}_f(U_i)$ then there is a unique
  $h_i\in \LG_f(\OO_X(U_i))$ such that $u_i(s_i)= h_i^{-1} s_i$. The
  map
  $P_f^{\mathrm{ad}}(U_i)\rightarrow \LG_f(\OO_X(U_i)), u_i\mapsto
  h_i$ is a group isomorphism, which is in fact a homeomorphism (by
  construction). We let $\rho_{U_i}$ be the composition
  \[\rho_{U_i}: \Gal_F\overset{\rho}{\longrightarrow} \LPad(X) \subset
    P^{\mathrm{ad}}_f(X)\rightarrow
    P^{\mathrm{ad}}_f(U_i)\overset{\cong}{\longrightarrow}
    \LG_f(\OO_X(U_i)).\] Conversely, if we start with the family
  $\{\rho_{U_i}\}_{i\in I}$ then for each $\gamma\in \Gal_F$,
  \eqref{transform_balaur} implies that
  $\{\rho_{U_i}(\gamma)\}_{i\in I}$ glue to
  $\rho(\gamma)\in P^{\mathrm{ad}}_f(X)$. Since the glueing data for
  $P_f$ is given by elements of $\Ghat(\OO_X(U_i))$, we get that
  $\rho(\gamma)$ lies in $\LPad(X)$. The admissibility of $\rho_{U_i}$
  then implies that $\rho: \Gal_F \rightarrow \LPad(X)$,
  $\gamma \mapsto \rho(\gamma)$ is admissible.
\end{proof}

\begin{remar}\label{CPad} We make the analogous definitions for
  $C$-groups as follows.  If $P$ is a $\Ghat_X$-torsor on $X$ then
  $P^T\coloneqq \Ghat^T_X\times^{\Ghat_X} P$ is $\Ghat^T_X$-torsor on
  $X$. By interpreting the $C$-group $\CG_f$ as an $L$-group $\LG^T_f$
  and using $P^T$ instead of $P$ we may make the same definitions for
  $C$-groups, so that $\CPad(X)={}^L P^{T, \mathrm{ad}}(X)$, and an
  admissible representation $\rho: \Gal_F\rightarrow \CPad(X)$ is just
  a representation
  $\rho: \Gal_F\rightarrow {}^L P^{T, \mathrm{ad}}(X)$, which is
  admissible in the sense of Definition \ref{defi_twisted}. Moreover,
  the exact sequence \eqref{LGCG} induces an exact sequence of groups
  \begin{equation}\label{twisted_d}
    1 \rightarrow \LPad(X)\rightarrow \CPad(X)\overset{d}{\longrightarrow} \OO_X(X)^{*}.
  \end{equation}
\end{remar}
 
\section{Sen theory}\label{sec_Sen_ban}

Let $F$ be a finite extension of $\Q_p$ and let
$F_{\infty}=F(\mu_{p^{\infty}})=\cup_{n} F_n$, where
$F_n=F(\mu_{p^n})$. Let $H_F={\rm Gal}(\bar{F}/F_{\infty})$ and
$\Gamma_F={\rm Gal}(F_{\infty}/F)\simeq {\rm Gal}_F/H_F$.
    
Let $A$ be a $\Qp$-Banach algebra and let $V$ be a finite free
$A$-module with a continuous $A$-linear action of ${\rm Gal}_F$ (one
could as well assume that $V$ is locally free over $A$, but in
applications $V$ will actually be free).  Define
\[\widetilde{D}_{\rm Sen}(V)\coloneqq (\mathbb{C}_p\wtimes_{\Q_p}
  V)^{H_F}\] and define $D_{\rm Sen}(V)$ as the set of
$\Gamma_F$-finite vectors in $\widetilde{D}_{\rm Sen}(V)$, i.e.~those
$v\in \widetilde{D}_{\rm Sen}(V)$ for which $A[\Gamma_F]v$ is a
finitely generated $A$-module. Observe that
$\widetilde{D}_{\rm Sen}(V)$ is an
$\widehat{F}_{\infty}\widehat{\otimes}_{\Q_p} A$-module and that
$D_{\rm Sen}(V)$ is an $F_{\infty}\otimes_{\Q_p} A$-submodule of
$\widetilde{D}_{\rm Sen}(V)$. The following result is classical when
$A=L$, but somewhat curiously it is not easy to find it in the
literature in the form below:
    
\begin{thm}\label{Sen}
  With the above notations the $F_{\infty}\otimes_{\Q_p} A$-module
  $D_{\rm Sen}(V)$ is free and the natural map
  \begin{equation}\label{Sen_compare}
    (\mathbb{C}_p\wtimes_{\Q_p} A)\otimes_{F_{\infty}\otimes_{\Q_p} A} D_{\rm Sen}(V) \to \mathbb{C}_p\wtimes_{\Qp} V
  \end{equation} is an isomorphism. Moreover, 
  for all $x\in D_{\rm Sen}(V)$ the limit 
  \begin{equation}\label{theta_limit}
    \theta(x)=\lim_{\gamma\to 1} \frac{\gamma.x-x}{\chi_{\rm cyc}(\gamma)-1}
  \end{equation}
  exists in $D_{\rm Sen}(V)$ and
  $\theta\in {\rm End}_{F_{\infty}\otimes_{\Q_p} A} (D_{\rm Sen}(V))$.
\end{thm}

\begin{proof} We will deduce this by descent from the results of Berger--Colmez \cite[Prop.\,4.1.2]{BC}. For this we need the following lemma:

\begin{lem}\label{zwischen}
  If $E$ is a finite Galois extension of $F$ and $V_{E}$ is the
  restriction of $V$ to ${\rm Gal}_E$, then the natural map
  $(E_{\infty}\otimes_{\Q_p} A)\otimes_{F_{\infty}\otimes_{\Q_p} A}
  D_{\rm Sen}(V)\to D_{\rm Sen}(V_E)$ is an isomorphism.
\end{lem}

\begin{proof} Consider the finite group
  $G={\rm Gal}(E_{\infty}/F_{\infty})$. We clearly have
  $\widetilde{D}_{\rm Sen}(V)=\widetilde{D}_{\rm Sen}(V_E)^{G}$.  We
  claim that a vector $v\in \widetilde{D}_{\rm Sen}(V)$ is
  $\Gamma_F$-finite if and only if it is $\Gamma_E$-finite, or
  equivalently $v$ is a $G$-invariant $\Gamma_E$-finite vector of
  $\widetilde{D}_{\rm Sen}(V_E)$. In other words
  $D_{\rm Sen}(V)\simeq D_{\rm Sen}(V_E)^{G}$. Thus given the claim,
  the result follows then by Galois descent for the finite Galois
  extension $E_{\infty}/F_{\infty}$.

  To prove the claim, we observe that one direction is trivial and the
  other follows from Cayley--Hamilton. Namely, let $e_1, \ldots, e_n$
  be some generators of $A[\Gamma_F]v$ as an $A$-module, let $\gamma$
  be a topological generator of $\Gamma_F$ and $M=(a_{ij})$ be a
  matrix such that $\gamma e_i= \sum_{j=1}^n a_{ij} e_j$ for
  $1\le i \le n$. It follows from Cayley--Hamilton that, for all
  $k\ge 0$, $M^{mk}$ is an $A$-linear combination of $M^{mi}$ for
  $0\le i \le n-1$. If $\gamma^m$ is a topological generator of
  $\Gamma_E$, we deduce that $\gamma^{mi} v$ for $0\le i \le n-1$
  generate $A[\Gamma_E]v$ as an $A$-module.
\end{proof}

Coming back to the proof of the Theorem, let $A^+$ be the unit ball in
$A$. By continuity, there is a finite Galois extension $E$ of $F$ and
a free $A^+$-module $T\subset V$ such that $A\otimes_{A^+} T\simeq V$,
$T$ is stable under ${\rm Gal}_E$ and ${\rm Gal}_E$ acts trivially on
$T/(12pT)$. By a result of Berger--Colmez \cite[Prop.\,4.1.2]{BC}, for
$n$ sufficiently large we can find a free $E_n\otimes_{\Q_p} A$-module
$D_n\subset (\mathbb{C}_p\wtimes_{\Q_p} A)\otimes_A V$ such that
\begin{itemize}
 \item $D_n$ is stable under ${\rm Gal}_E$, fixed by $H_E$ and has a
   basis $\mathcal{B}$ over $E_n\otimes_{\Q_p} A$ in which the matrix
   $M_{\gamma}$ of each $\gamma\in \Gamma_E$ satisfies
   $||M_{\gamma}-1||<1$, i.e. the entries of $M_{\gamma}-1$ have
   positive valuations.

 \item the natural map
   $ (\mathbb{C}_p\wtimes_{\Q_p} A)\otimes_{ E_n\otimes_{\Q_p} A}
   D_n\to (\mathbb{C}_p\wtimes_{\Q_p} A)\otimes_A V$ is an
   isomorphism.
\end{itemize}
We claim that
$D_{\rm Sen}(V_E)=(E_{\infty}\otimes_{\Qp}
A)\otimes_{E_n\otimes_{\Q_p} A} D_n$. This immediately implies that
the theorem holds with $F$ replaced by $E$. But then the previous
lemma combined with the fact that $E_{\infty}\otimes_{\Q_p} A$ is
faithfully flat over $F_{\infty}\otimes_{\Q_p} A$ allows one to
conclude.
 
To prove the claim, first take $H_E$-invariants in the
isomorphism
\[(\mathbb{C}_p\wtimes_{\Q_p} A)\otimes_{ E_n\otimes_{\Q_p} A}
  D_n\simeq (\mathbb{C}_p\wtimes_{\Q_p} A)\otimes_A V\] to deduce that
$\widetilde{D}_{\rm Sen}(V_E)\simeq
(\widehat{E}_{\infty}\widehat{\otimes}_{\Q_p}
A)\otimes_{E_n\otimes_{\Q_p} A} D_n$, where $\widehat{E}_{\infty}$ is
the closure of $E_{\infty}$ in $\Cp$.  This implicitly uses that
$ (\mathbb{C}_p\wtimes_{\Q_p} A)\otimes_A V\simeq
\Cp\widehat{\otimes}_{\Q_p} V$ (since $V$ is finite over $A$) and
$(\mathbb{C}_p\widehat{\otimes}_{\Q_p} A)^{H_E}\simeq
\widehat{E}_{\infty}\widehat{\otimes}_{\Q_p} A$ (pick an orthonormal
basis of $A$ over $\Q_p$ and use Ax-Sen-Tate).  Now
\cite[Lem.\,4.2.7]{bellovin} implies that $\Gamma_E$-finite vectors in
$\widetilde{D}_{\mathrm{Sen}}(V_E)$ coincide with
$(E_{\infty}\otimes_{\Qp} A)\otimes_{E_n\otimes_{\Q_p} A}
D_n=E_{\infty}\otimes_{E_n} D_n$.

We will now show that the operator $\theta$ is well defined. Since 
\[D_{\Sen}(V)= D_{\Sen}(V_E)^{\Gal(E_{\infty}/F_{\infty})}\] by Lemma
\ref{zwischen}, it is enough to compute the limit inside
$D_{\Sen}(V_E)$. Since
\[D_{\Sen}(V_E)=E_{\infty} \otimes_{E_n} D_n = \varinjlim_{m\ge n} E_m
  \otimes_{E_n} D_n\] it is enough to compute the limit inside $D_n$.
If $\gamma\in \Gamma_{E_n}$ then the series
\[\log(M_{\gamma})= - \sum_{m=1}^{\infty} \frac{(1-M_{\gamma})^m}{m}\]
converges in $\End_{E_n \otimes A}(D_n)$. Since
$\log(M_{\gamma^k})=\log(M_{\gamma}^k)=k \log(M_{\gamma})$ and
$\Gamma_E$ is pro-cyclic the operator
$\theta'\coloneqq \log(M_{\gamma})(\log_p (\chi_{\cyc}(\gamma))^{-1}$
is independent of $\gamma$. Moreover, the $\Gamma_{E_n}$-action on
$D_n$ is given by
$\gamma \mapsto \exp (\log_p(\chi_{\cyc}(\gamma))
\theta')=M_{\gamma}$.  Thus the limit exists and $\theta=\theta'$.
\end{proof} 

\begin{defi}\label{Sen_op} The Sen operator of $V$ is the
  $\mathbb{C}_p\widehat{\otimes}_{\Q_p} A$-linear endomorphism
  $\Theta_{\Sen, V}$ of $\mathbb{C}_p\widehat{\otimes}_{\Q_p} V$
  obtained from the endomorphism $\theta$ of $D_{\rm Sen}(V)$ by
  linearity using \eqref{Sen_compare}.
\end{defi}
\begin{remar} It follows from the proof of Theorem \ref{Sen} that the
  Sen operator does not change if we restrict the Galois
  representation to the Galois group of a finite extension.
\end{remar}
    
\begin{lem}\label{transf_tannaka} Let $V$ and $W$ be free $A$-modules
  of finite rank with continuous $\Gal_F$-action.
  Then
  \[\Theta_{\Sen, V\oplus W}= \Theta_{\Sen, V} \oplus \Theta_{\Sen,
      W}\] and
  \[\Theta_{\Sen, V\otimes_A W}= \Theta_{\Sen, V} \otimes \id_W +
    \id_V \otimes \Theta_{\Sen,W}.\]
\end{lem}
\begin{proof} The first assertion is trivial. To prove the second
  assertion we note that the natural map
  $D_{\Sen}(V)\otimes_{F_{\infty}\otimes_{\Qp} A} D_{\Sen}(W)
  \rightarrow D_{\Sen}(V\otimes_A W)$ becomes an isomorphism after
  extending scalars to $\Cp\wtimes_{\Qp} A$. Since
  $\Cp\wtimes_{\Qp} A$ is faithfully flat over
  $F_{\infty}\otimes_{\Qp} A$ we conclude that the map is an
  isomorphism. The assertion then follows from \eqref{theta_limit}
  applied to an element of the form $x\otimes y$ with
  $x\in D_{\Sen}(V)$ and $y\in D_{\Sen}(W)$.
\end{proof}

\begin{lem}\label{extend_scalars} Let $A\rightarrow B$ be a continuous
  map of $\Qp$-Banach algebras. There is a natural isomorphism
  $D_{\Sen}(B\otimes_{A} V)\simeq
  D_{\Sen}(V)\otimes_{F_{\infty}\otimes A} (F_{\infty}\otimes B)$
  inducing an identification
  \[\Theta_{\Sen, V\otimes_A B}= \Theta_{\Sen, V}\otimes \id.\]
\end{lem}
\begin{proof} This is proved in the same way as the previous lemma.
\end{proof}
    
We may identify
$\End_{\Cp\wtimes A}(\Cp\wtimes V)= \Cp\wtimes \End_A(V)$ and define
two commuting, $A$-linear actions of $\Gal_F$ on it via
\[\gamma*(z\otimes X)=\gamma(z)\otimes X,\quad \gamma \cdot(z\otimes
  X)=z\otimes \Ad(\rho(\gamma))(X),\] for all $z\in \Cp$ and
$X\in \End_A(V)$, where $\rho$ denotes the action of $\Gal_F$ on $V$.

\begin{lem}\label{hallelujah} We have
  $\gamma\cdot (\gamma \ast \Theta_{\Sen, V})= \Theta_{\Sen, V}$, for
  all $\gamma\in \Gal_F$.
\end{lem}
\begin{proof} We observe that the $*$-action maybe described as
  follows: every $\gamma\in \Gal_F$ induces a semi-linear map
  $\gamma\otimes\id: \Cp\wtimes V\rightarrow \Cp \wtimes V$,
  $z\wtimes v\mapsto \gamma(z)\wtimes v$, and
\[\gamma* \varphi= (\gamma\otimes \id) \circ \varphi\circ (\gamma \otimes \id)^{-1}\]
for all $\varphi \in \End_{\Cp\wtimes A}(\Cp\wtimes V)$. Thus the
assertion of the lemma is equivalent to the equality
\[ (\gamma \otimes \rho(\gamma))\circ \Theta_{\Sen, V}\circ (\gamma
  \otimes \rho(\gamma))^{-1}= \Theta_{\Sen, V}.\] Note that
$\gamma \mapsto \gamma \otimes \rho(\gamma)$ is just the diagonal
action of $\Gal_F$ on $\Cp\wtimes V$. Since both operators are
$\Cp\wtimes A$-linear, it is enough to show that they agree on
$D_{\Sen}(V)$. Since the action of $\Gal_F$ on $D_{\Sen}(V)$ commutes
with the limit in \eqref{theta_limit}, we obtain the assertion.
\end{proof} 
  
\section{Families of Galois representations and infinitesimal
  characters}\label{families_gal}

Recall that we fix algebraic closure $\Qpbar$ of $\Qp$ and a finite
extension $F$ of $\Qp$ contained in $\Qpbar$.  Let $G$ be a connected
reductive group defined over $F$. Let $E\subset \Qpbar$ be a finite
Galois extension of $F$ such that $G_E$ is split and and let
$\Gamma=\Gal(E/F)$. Let $L$ be a further finite extension of $\Qp$,
let $\OO$ be the ring of integers of $L$ with residue field $k$ and a
uniformizer $\varpi$. We assume that $L$ is large enough so that there
are $[E:\Qp]$ field embeddings of $E$ into $L$. In particular,
$G\times_{F, \sigma} L$ is split for any embedding
$\sigma: F\hookrightarrow L$.

\subsection{Topology on the coefficients}\label{sec_top} Let
$\mathcal T$ be a Grothendieck topology on the category of rigid
analytic spaces
over $L$, such that the following hold:
\begin{itemize}
\item the functor $X\mapsto \OO_X(X)$ is a sheaf on $\mathcal T$;
\item every $X$ admits a covering in $\mathcal T$ by affinoids;
\item every covering $\{U_i\rightarrow X\}_{i\in I}$ with $U_i$ and
  $X$ affinoid admits a finite subcover.
\end{itemize}
An example of such topology is the analytic topology, \cite[\S
9.3.1]{BGR}, or the \'etale topology, \cite[\S 8.2]{FP}. We want to
topologize the rings $\OO_X(X)$.  If $X=\Sp(A)$ then $\OO_X(X)=A$,
which is naturally an $L$-Banach space. Moreover, a map between
affinoids induces a continuous map on the ring of functions.

\begin{lem}\label{Banach_open} Let $\{U_i\}_{i\in I}$ be a covering of
  $X$ in $\mathcal T$. If $X$ and $U_i$ for $i\in I$ are affinoid then
  the map
\[\OO_X(X)\rightarrow \prod_{i\in I} \OO_{U_i}(U_i)\] 
is a homeomophism onto its image for the product topology on the
target and Banach space topology for the rings of global sections.
\end{lem} 
\begin{proof} 
  For each $i,j\in I$ we choose a covering by affinoids
  $\{V_{ijk} \rightarrow U_i \times_X U_j\}_k$.  The maps in the
  equalizer diagram
\begin{equation} 
\begin{tikzcd}
  \OO_X(X) \arrow[r] & \prod\limits_i \OO_{U_i }(U_i)
  \arrow[r,yshift=-2pt] \arrow[r,yshift=2pt]&
  \prod\limits_{i,j,k}\OO_{V_{ijk}}(V_{ijk})
\end{tikzcd}
\end{equation}
are continuous and hence $\OO_X(X)$ is closed in
$\prod_{i\in I} \OO_{U_i}(U_i)$ for the subspace topology, which we
denote by $\sigma$. Let $F\subset I$ be finite such that
$\{U_i\}_{i\in F}$ is a cover of $X$. Then by the same argument
$\OO_X(X)$ is closed in $\prod_{i\in F} \OO_{U_i}(U_i)$ for the
subspace topology, which we denote by $\sigma'$. Since $F$ is finite,
$\sigma'$ is a Banach space topology. Since by the Open Mapping
theorem every continuous bijection between Banach spaces is a
homeomorphism, $\beta$ and $\sigma'$ coincide. The projection map
$\prod_{i\in I} \OO_{U_i}(U_i)\rightarrow \prod_{i\in F}
\OO_{U_i}(U_i)$ induces continuous maps
$(\OO_X(X), \beta)\overset{\id}{\rightarrow} (\OO_X(X),
\sigma)\overset{\id}{\rightarrow} (\OO_X(X), \sigma')$, thus $\sigma$
and $\beta$ coincide.
\end{proof}

For an arbitrary $X$ we define the topology $\tau$ on $\OO_X(X)$ to be
the coarsest topology such that for all coverings
$\{U_i\rightarrow X\}_i$ by affinoids the map
$\OO_X(X) \rightarrow \prod_i \OO_{U_i}(U_i)$ is continuous, where we
put the Banach space topology on each $\OO_{U_i}(U_i)$ and the product
topology on the target.

\begin{lem}\label{topology} If $\{U_i\rightarrow X\}_{i\in I}$ is any covering of $X$ in $\mathcal T$ by affinoids then the subspace topology on $\OO_X(X)$
  induced by $\OO_X(X)\rightarrow \prod_{i\in I} \OO_{U_i}(U_i)$
  coincides with $\tau$, for the for the Banach space topology on
  $\OO_{U_i}(U_i)$ and the product topology on the target.
\end{lem}
\begin{proof} Let $\{U'_j\rightarrow X\}_{j\in J}$ be another covering
  in $\mathcal T$ of $X$ by affinoids. Then
  $\{U_i \times_X U_j' \rightarrow U_i\}_i$ and
  $\{U_i \times_X U_j' \rightarrow U_j'\}_j$ are both coverings in
  $\mathcal T$. For each $i,j$ we let $\{V_{ijk}\}_k$ be a cover of
  $U_i \times_X U_j'$ by affinoids. It follows from Lemma
  \ref{Banach_open} that the maps
  \[ \prod_{i} \OO_{U_i}(U_i)\rightarrow
    \prod_{i,j,k}\OO_{V_{ijk}}(V_{ijk}), \quad \prod_{j}
    \OO_{U'_j}(U'_j)\rightarrow \prod_{i,j,k}\OO_{V_{ijk}}(V_{ijk})\]
  are homeomorphisms onto their images. Since both maps coincide, when
  restricted to $\OO_X(X)$, we deduce that the subspace topologies on
  $\OO_X(X)$ induced by the coverings are the same.
\end{proof}

\begin{cor} If $X$ is an affinoid then $\tau$ coincides with the
  Banach space topology on $\OO_X(X)$.
\end{cor} 
\begin{proof} This follows from the fact that
  $X\overset{\id}{\rightarrow} X$ is a covering of $X$ by affinoids.
\end{proof} 

\begin{remar} If a covering of $X$ by admissible open affinoids for
  the analytic topology is also a covering in $\mathcal T$ then Lemma
  \ref{topology} implies that both topologies on $\OO_X(X)$
  coincide. In particular, there is no difference for the topology on
  $\OO_X(X)$ in the analytic and \'etale topologies.
\end{remar} 

\begin{remar} \label{torsor_top} We may identify $\Ghat$ with a
  Zariski closed subset of $\mathbb{A}^n$ for some $n$.  Then
  $\Ghat(\OO_{U_i}(U_i))$ is a closed subset of
  $\OO_{U_i}(U_i)^{\oplus n}$. If $P$ is a $\Ghat$-torsor over $X$ for
  $\mathcal T$ and $\mathcal U=\{U_i\}_{i\in I}$ is a covering of $X$
  by affinoids trivialising $P$, then we may identify $\Pad(U_i)$ with
  $\Ghat(\OO_{U_i}(U_i))$.  This identification is canonical up to
  conjugation by elements of $\Ghat(\OO_{U_i}(U_i))$, thus the induced
  topology on $\Pad(U_i)$ does not depend on choices.  The argument in
  Lemma \ref{topology} shows that the subspace topology on $\Pad(X)$
  induced by the inclusion $\Pad(X) \subset\prod_i \Pad(U_i)$ does not
  depend on the choice of the cover $\mathcal U$.
\end{remar}

\subsection{The Sen operator in the affinoid case}\label{sec_Sen}

Let $U=\Sp(A)$ be an affinoid and let
\[\rho_U: \Gal_F \rightarrow \LG_f(A)=\Ghat(A) \rtimes \Gamma\] be an
admissible representation.  If $r: \LG_f \rightarrow \GL(V)$ is an
algebraic representation of $\LG_f$ on a finite dimensional $L$-vector
space $V$, we let
\[\Theta_{\mathrm{Sen}, r, U}\in \End_{\Cp \wtimes A}( (\Cp \wtimes A)\otimes_{L} V)\]
be the Sen operator $\Theta_{\Sen, V\otimes_L A}$ defined in
Definition \ref{Sen_op} for the representation
$r\circ \rho_U: \Gal_F\to \GL(V)$.

Let $\ghat$ be the Lie algebra of $\Ghat$, which we identify with the
Lie algebra of $\LG_f$.
 
\begin{lem}\label{tannaka} There is a unique
  $\Theta_{\mathrm{Sen}, U}\in (\Cp \wtimes A)\otimes_{L} \ghat$ such
  that
  \[ \mathrm{Lie}(r)( \Theta_{\mathrm{Sen}, U})= \Theta_{\mathrm{Sen},
      r, U}\] for every algebraic representation $(r, V)$ of $\LG_f$
  over $L$. Moreover, if $U'=\Sp(A') \rightarrow U$ is a map of
  affinoid varieties then the map
\[(\Cp \wtimes A)\otimes_{L} \ghat \rightarrow (\Cp \wtimes
  A')\otimes_{L} \ghat,\] induced by extension of scalars, sends
$\Theta_{\mathrm{Sen}, U}$ to $\Theta_{\mathrm{Sen}, U'}$.
\end{lem} 
\begin{proof} As explained in \cite[Lem.\,2.2.5]{patrikis}, given the
  transformation properties of the Sen operator with respect to the
  direct sum and tensor product proved in Lemma \ref{transf_tannaka},
  the assertion follows from the Tannaka duality for Lie
  algebras. Patrikis refers for this to \cite{HC}. We note that
  alternatively one could follow Milne \cite[Prop.\,6.11]{milne1} and
  deduce the statement from the Tannaka duality for algebraic
  groups. Informally, it says that if $R$ is an $L$-algebra then to
  give an element of $\LG_f(R)$ is equivalent to giving a compatible
  family $\{ g_V\in \Aut_R(R\otimes_L V)\}_V$ indexed by algebraic
  representation $V$ of $\LG_f$ defined over $L$.  The statement is
  made precise in the references cited below. We may identify
 \begin{equation}
   R \otimes_{L} \ghat = \Ker( \LG_f(R[\varepsilon])\rightarrow \LG_f(R)),
 \end{equation}
 where $R[\varepsilon]=R[X]/(X^2)$ is the algebra of dual numbers over
 $R$, and consider
 \begin{equation}\label{dual_numbers}
   1+ \varepsilon \Theta_{\mathrm{Sen}, r, U} \in \Aut_{ (\Cp \wtimes A)[\varepsilon]}(  (\Cp \wtimes A)[\varepsilon] \otimes_{L} V).
 \end{equation}
 The existence and uniqueness of $\Theta_{\mathrm{Sen}, U}$ follows
 from \cite[Prop\,.2.8]{DM}, see also \cite[\S 2.10]{milne1}. The last
 assertion follows from Lemma \ref{extend_scalars}.
\end{proof}

The natural $A$-linear action of $\Gal_F$ on $\Cp \wtimes A$ induces
an action of $\Gal_F$ on $(\Cp \wtimes A)\otimes_{L} \ghat$, defined
by
\begin{equation} \label{staraction} \gamma*(z\otimes
  X)=\gamma(z)\otimes X
\end{equation}
for $z\in \Cp\wtimes A$, $X\in \ghat$ and $\gamma\in \Gal_F$.
     
On the other hand, the group $\LG_f(\Cp \wtimes A)$ acts on
$(\Cp \wtimes A)\otimes_{L} \ghat$ via the adjoint action of $\LG_f$
on its Lie algebra $\ghat$. Since $\LG_f$ is defined over $\Qp$, the
action of $\Gal_F$ on $\Cp \wtimes A$ induces an action of $\Gal_F$ on
$\LG_f(\Cp \wtimes A)$, denoted $(\gamma, g)\to \gamma(g)$.  These
three actions are related by
\begin{equation}\label{adj}
  \gamma*({\rm Ad}(g)(X))={\rm Ad}(\gamma(g))(\gamma*X)
\end{equation}
for $\gamma\in \Gal_F$, $g\in \LG_f(\Cp \otimes A)$,
$X\in (\Cp \otimes A)\otimes_{L} \ghat$.  Indeed, it suffices to check
this after an $L$-embedding of $\LG_f$ in some $\GL_n$, where it
becomes obvious. The above relation \eqref{adj} shows in particular
that the adjoint action of
$\LG_f(A)\subset \LG_f(\Cp \wtimes A)^{\Gal_F}$ commutes with the
action of $\Gal_F$.
   
Finally, we define another action of $\Gal_F$ on
$(\Cp \otimes A)\otimes_{L} \ghat$, via
\begin{equation} \label{dotaction} \gamma\cdot X= {\rm
    Ad}(\mu_G(\bar{\gamma}))(X),
\end{equation}
where $\bar{\gamma}$ is the image of $\gamma$ in $\Gamma$. Since
$\mu_G(\bar{\gamma})\in \LG_f(A)$, the discussion in the previous
paragraph shows that this action commutes with the star action
$(\gamma, X)\to \gamma*X$ of $\Gal_F$ on
$(\Cp \otimes A)\otimes_{L} \ghat$.
   
For all $\gamma\in \Gal_F$ we write
$\rho_U(\gamma)= (c_{\gamma},\bar{\gamma})$ with
$c_{\gamma} \in \widehat{G}(A)$ as in subsection \ref{adm_reps}.

\begin{lem}\label{equiv}
  For all $\gamma\in \Gal_F$ we have
  \[{\rm Ad}(c_{\gamma})(\gamma*(\gamma\cdot\Theta_{\mathrm{Sen},
      U}))=\Theta_{\mathrm{Sen}, U}.\]
\end{lem}

\begin{proof}  

  For simplicity let $\Theta=\Theta_{\mathrm{Sen}, U}$ and let
  $\rho=\rho_U$. We identify $\Gamma$ and $\Ghat$ with their images in
  $\LG_f$, so that $\rho(\gamma)=c_{\gamma} \bar{\gamma}$. By the
  discussion preceding the lemma
  \[{\rm Ad}(c_{\gamma})(\gamma*(\gamma\cdot\Theta))={\rm
      Ad}(c_{\gamma}) (\gamma\cdot (\gamma*\Theta))={\rm
      Ad}(c_{\gamma} \mu_G(\bar{\gamma}))(\gamma*\Theta)= {\rm
      Ad}(\rho(\gamma))(\gamma*\Theta).\]

  Thus we need to check that
  ${\rm Ad}(\rho(\gamma))(\gamma*\Theta)=\Theta$.  It suffices to
  check the equality after applying $\mathrm{Lie}(r)$ to both sides,
  where $r: \LG_f\to \GL(V)$ is a faithful representation of
  $\LG_f$. Note that for all $g\in \LG_f, \gamma\in \Gal_F$ and
  $u\in (\Cp \wtimes A)\otimes_{L} \ghat$ we have
  \[\mathrm{Lie}(r)({\rm Ad}(g)(u))={\rm
      Ad}(r(g))(\mathrm{Lie}(r)(u)), \quad
    \mathrm{Lie}(r)(\gamma*u)=\gamma*\mathrm{Lie}(r)(u).\] Letting
  $\Theta_r=\Theta_{\mathrm{Sen}, r, U}= \mathrm{Lie}(r)(\Theta)$, we
  need to check that
  ${\rm Ad}(r\circ \rho(\gamma))(\gamma*\Theta_r)= \Theta_r$ for
  $r\circ \rho: \Gal_F\to \GL(V\otimes A)$.  This assertion is proved
  in Lemma \ref{hallelujah}.
\end{proof}

Let $S(\ghat^*)$ be the symmetric algebra in the $L$-linear dual of
$\ghat$. We think of $S(\ghat^*)$ as polynomial functions on
$\ghat$. Let $S(\ghat^*)^{\Ghat}$ be the subring $\Ghat$-invariant
polynomial functions of $\ghat$.

If $V$ is a finite dimensional $L$-vector space and $R$ is an
$L$-algebra then
\[\Hom_{L\text{-}\alg}(S(V^*), R)\cong \Hom_{L}(V^*, R)\cong V\otimes_L R.\]
If $\lambda \in V\otimes_L R$ then we will denote the corresponding
$L$-algebra homomorphism by $\ev_{\lambda}$.  To give a homomorphism
of $L$-algebras from $S(\ghat^*)$ to an $L$-algebra $R$ is the same as
to give an $L$-linear map from $\ghat^*$ to $R$, which is the same as
to give an element of $\ghat\otimes_L R$.  Thus the elements
$\Theta_{\mathrm{Sen}, U}$ give us a compatible system of homomorphism
of $L$-algebras $\theta_U: S(\ghat^*)\rightarrow \Cp \wtimes A$, which
we may restrict to the subrings considered above.

\begin{lem}\label{thetan} The map
  $\theta_U: S(\ghat^*)^{\Ghat} \rightarrow \Cp \wtimes A$ takes
  values in $E\otimes A$. Moreover,
\begin{equation}\label{transform2}
  \gamma(\theta_U(f)) = \theta_U ( \gamma\cdot f), \quad \forall \gamma\in \Gal_F, \quad \forall f\in S(\ghat^*)^{\Ghat},
\end{equation}
where $\gamma\cdot$ denotes the $L$-linear action of $\Gal_F$ on
$S(\ghat^*)$ via $\mu_G$.
\end{lem}

\begin{proof}
  Let $\Theta=\Theta_{\mathrm{Sen}, U}$.  Thinking of elements of
  $S(\ghat^*)^{\Ghat}$ as functions on $\ghat$, we need to prove that
  $\gamma(f(\Theta))=(\gamma\cdot f)(\Theta)$ for all
  $f\in S(\ghat^*)^{\Ghat}$.  Since
  $(\gamma\cdot f)(u)=f(\gamma^{-1}\cdot u)$, for all
  $u\in \Cp\wtimes A\otimes_L\ghat$, it is enough to prove that
  $\gamma(f(\Theta))=f(\gamma^{-1}\cdot\Theta)$.  By observing that
  $f(\gamma*u)=\gamma(f(u))$ for all
  $u\in \Cp\wtimes A\otimes_L\ghat$, we obtain
  \[f(\Theta)=f(c_{\gamma}\cdot (\gamma*(\gamma\cdot
    \Theta)))=f(\gamma*(\gamma\cdot
    \Theta))=\gamma(f(\gamma\cdot\Theta)),\] where the first equality
  follows from Lemma \ref{equiv} and the second from the
  $\Ghat$-invariance of $f$.  Replacing $\gamma$ by $\gamma^{-1}$
  yields the result.
\end{proof}

\subsection{The Sen operator and twisted families of
  $L$-parameters}\label{sen_gauge}
Let $X$ be a rigid analytic variety and let $P$ be a $\Ghat$-torsor on
$X$ with respect to a Grothendieck topology satisfying the conditions
in Section \ref{sec_top}.  Let $\rho: \Gal_F\rightarrow \LPad(X)$ be
an admissible representation, see section \ref{sec_balaur}.

Let $\mathcal U=\{U_i\rightarrow X\}_{i\in I}$ be a covering of $X$ by
affinoids together with local sections $s_i: U_i\rightarrow P$.  Lemma
\ref{fairwell_balaur} gives a family of admissible representations
$\rho_{U_i}: \Gal_F \rightarrow \LG(\OO_X(U_i))$ satisfying
\[\rho_{U_i}(\gamma)= g_{ij} \rho_{U_j}(\gamma) g_{ij}^{-1}, \quad \forall \gamma\in \Gal_F, \quad \forall i,j\in I,\]
where $g_{ij}\in \Ghat(\OO_X(U_{ij}))$ come from the glueing data for
the torsor.  For each $U_i$, Lemma \ref{tannaka} gives us
$\Theta_{\Sen, U_i}\in (\Cp\wtimes \OO_X(U_i))\otimes_L \ghat$, which
behave well under maps between affinoid varieties. The uniqueness of
the Sen operator implies that
\[\Theta_{\Sen, U_i}= g_{ij} \Theta_{\Sen, U_j} g_{ij}^{-1}, \quad
  \forall i,j\in I\] in $(\Cp\wtimes \OO_X(U_{ij}))\otimes_L
\ghat$. Hence, they glue to
$\Theta_{\Sen, \rho}\in \Cp \wtimes (\Lie \Pad)(X)$, use
\eqref{dual_numbers}.
\begin{lem} If $f:Y\rightarrow X$ is a map of rigid varieties over
  $L$, and $\rho_Y: \Gal_F \rightarrow \LPad(Y)$ is the extension of
  scalars of $\rho$ to $Y$, see Definition \ref{base_change}, then
  $\Theta_{\Sen, \rho_Y}$ is the image $\Theta_{\Sen, \rho}$ in
  $\Cp \wtimes (\Lie \Pad)(Y)$.
\end{lem}
\begin{proof} One can check this on an affinoid cover trivialising
  $P$, and there the assertion follows from the last part of Lemma
  \ref{tannaka}.
\end{proof}
 
\begin{lem}\label{theta_balaur} The maps $\theta_{U_i}: S(\ghat^*)^{\Ghat} \rightarrow  E \otimes_{\Qp} \OO_X(U_i)$ given by Lemma \ref{thetan} glue to 
  a homomorphism of $L$-algebras
\begin{equation}\label{def_theta}
  \theta: S(\ghat^*)^{\Ghat} \rightarrow  E \otimes_{\Qp} \OO_X(X),
\end{equation}
which satisfies the transformation property in
\eqref{transform2}. Moreover, $\theta$ depends only on the equivalence
class of $\rho$.
\end{lem}
\begin{proof} Since the restrictions of $\Theta_{\Sen, U_i}$ and
  $\Theta_{\Sen, U_j}$ to $U_{ij}$ are conjugate by an element of
  $\Ghat(\OO_X(U_{ij}))$, the evaluation maps at $\Theta_{\Sen, U_i}$
  and $\Theta_{\Sen, U_j}$ coincide when restricted to
  $S(\ghat^*)^{\Ghat}$ and hence glue. The transformation property in
  \eqref{transform2} holds, since it holds on a cover.

  If we choose different sections $s_i'= g_i s_i$ then
  $\Theta_{\Sen, U_i}$ get replaced by
  $g_i \Theta_{\Sen, U_i} g_i^{-1}$, so we obtain the same
  $\theta_{U_i}$ and hence the same $\theta$. Clearly, $\theta$ does
  not change if we refine the cover. Thus $\theta$ does not depend on
  the choices made at the beginning of the subsection.

  If $\rho$ and $\rho'$ are conjugate by an element $u\in \Pad(X)$
  then the representations $\rho'_{U_i}$ and $\rho_{U_i}$ are also
  conjugate. More precisely,
  $\rho'_{U_i}(\gamma)= q_i \rho_{U_i}(\gamma) q_i^{-1}$, for all
  $\gamma \in \Gal_F$ and $i\in I$, where $q_i\in \Ghat(\OO_X(U_i))$
  are uniquely determined by $u(s_i)=q_i^{-1} s_i$, see \cite[\S 1,
  (1.1.4)]{breen_messing}. Hence, we obtain the same $\theta$.
\end{proof} 

\begin{lem}\label{base_change_theta}
  If $f:Y\rightarrow X$ is a map of rigid varieties over $L$ and
  $\theta_Y: S(\ghat^*)^{\Ghat} \rightarrow E \otimes_{\Qp} \OO_Y(Y)$
  is the map \eqref{def_theta} for $\rho_Y$, see Definition
  \ref{base_change}, then $\theta_Y$ is equal to $\theta$ composed
  with the map
  $\id \otimes f^{\sharp}: E\otimes_{\Qp} \OO_X(X)\rightarrow
  E\otimes_{\Qp} \OO_Y(Y)$.
\end{lem}
\begin{proof} It is enough to check it on the cover, where it follows
  from Lemma \ref{extend_scalars}.
\end{proof}

\subsection{Chevalley's restriction theorem} Let $\that$ be the Lie
algebra of $\That$ and let $W$ be the Weyl group associated to the
root system.  We note that $W$ is also the Weyl group of the dual root
system, \cite[\S 2.9]{hump_cox}.

\begin{prop}[Chevalley's restriction theorem]\label{chev} The
  restriction to $\that$ induces an isomorphism of
  rings \begin{equation}\label{eqn_chev}
    S(\ghat^*)^{\Ghat}\overset{\cong}{\longrightarrow} S(\that^*)^W.
  \end{equation}
\end{prop}

\begin{proof} If $\ghat$ is semi-simple both assertions are proved in
  \S 23.1 and the Appendix to \S 23 in \cite{hump}. The same proof
  carries over when $\ghat$ is reductive.
\end{proof}
 
Recall that the group $\Gamma=\Gal(E/F)$ acts on the based root
datum. Recall that by choosing a pinning the group $\Gamma$ acts
through $\Aut(\Ghat, \That)$. If $\theta\in \Aut(\Ghat, \That)$ and
$g\in N_{\Ghat}(\That)$ then
$ \theta(g) t \theta(g)^{-1}= \theta( g \theta^{-1}(t) g^{-1}) \in
\That$ for all $t\in \That$ and thus $\Aut(\Ghat, \That)$ normalizes
$N_{\Ghat}(\That)$ and thus acts on $S(\that^*)^W$.  The isomorphism
\eqref{eqn_chev} is $\Gamma$-equivariant, since it is obtained by
restriction of functions on $\ghat$ to $\that$.
 
\subsection{Harish-Chandra homomorphism.} Since $G_E$ is split we may
assume that the triple $(G_E,B,T)$ is defined over $E$. Let
$\mathfrak{t}$ be the Lie algebra of $T$.  We will now recall the
construction of Harish-Chandra homomorphism
\begin{equation}\label{HC}
  \psi: Z(\mathfrak{g}_E)\overset{\cong}{\longrightarrow} S(\mathfrak{t})^W.
\end{equation} 

Let $(h_i)_{1\leq i \leq n}$ be an $E$-basis of $\mathfrak{t}$. We
extend it to a basis of $\mathfrak{g}_E$,
$\{h_i, x_{\beta}, y_{\beta} : 1\leq i \leq n, \beta\in \Phi^+\}$,
such that $t x_{\beta} t^{-1}= \beta(t) x_{\beta}$ and
$t y_{\beta} t^{-1}= \beta^{-1}(t) y_{\beta}$ for all
$\beta\in \Phi^+$ and all $t\in T$.  By PBW theorem $U(\mathfrak g_E)$
has an $E$-basis consisting of monomials
$\prod_{\beta\in \Phi^+} y_{\beta}^{i_{\alpha}} \prod_{i=1}^n
h_i^{k_i} \prod_{\beta\in \Phi^+} x_{\beta}^{j_{\alpha}}$ with
$i_{\alpha}, k_i, j_{\alpha}\ge 0$.  Let
\[\xi: U(\mathfrak{g}_E)\rightarrow S(\mathfrak{t})\]
be a linear map, which sends all the monomials with $i_{\alpha}>0$ or
$j_{\alpha}>0$ to zero and is identity on the monomials with
$i_{\alpha}=j_{\alpha}=0$.

Let $\lambda\in X^*(T)_+$ and let $V(\lambda)$ be an irreducible
representation of $G_E$ of highest weight $\lambda$. The highest
weight theory implies that $V(\lambda)$ is absolutely irreducible as
representation of $G_E$ and $\mathfrak{g}_E$. In particular,
$\End_{\mathfrak{g}_E}(V(\lambda))=E$ and hence the action of
$U(\mathfrak{g}_E)$ on $V(\lambda)$ induces a ring homomorphism
\begin{equation}\label{central}
 \chi_{\lambda}: Z(\mathfrak{g}_E)\rightarrow E.
\end{equation}
As explained 
in \cite[\S VIII.5]{Knapp} we have
\begin{equation}\label{need} 
  \Lie(\lambda)(\xi(z))=\chi_{\lambda}(z), \quad \forall z\in Z(\mathfrak{g}_E), \quad \forall \lambda\in X^*(T)_+.
\end{equation}
Let $\mathfrak{t}\rightarrow S(\mathfrak t)$ be the $E$-linear map
sending $h\mapsto h-\delta(h)1$ for $h\in\mathfrak{t}$ with
$\delta=\tfrac{1}{2}\sum_{\alpha\in\Phi^+}\alpha$.  This induces an
isomorphism of $E$-algebras
$\eta: S(\mathfrak t)\overset{\cong}{\longrightarrow} S(\mathfrak
t)$. We define
$\psi\coloneqq \eta\circ \xi: Z(\mathfrak g_E)\rightarrow
S(\mathfrak{t})$.  It follows from \cite[Th.~8.18]{Knapp} that $\psi$
induces an isomorphism of $E$-algebras
$Z(\mathfrak{g}_E)\cong S(\mathfrak{t})^W$ which is the isomorphism
\eqref{HC}.

Each $\lambda\in X^*(T)\otimes_\ZZ E= \mathfrak t^*$ defines a linear
map $\mathfrak{t}\rightarrow E$, $t\mapsto \Lie(\lambda)(t)$ and hence
a ring homomorphism $S(\mathfrak{t})\rightarrow E$, which we denote by
$\ev_\lambda$. We note it follows from the construction of $\psi$ that
\begin{equation}\label{Verma}
  \ev_{\lambda+\delta}(\psi(z))= \chi_{\lambda}(z), \quad \forall \lambda\in X^*(T)_+, \quad \forall z\in Z(\mathfrak{g}_E),
\end{equation}
where $\delta$ is a half sum of positive roots. The ring
$S(\mathfrak{t})$ is reduced and the set of closed points
$\ev_\lambda: S(\mathfrak{t})\rightarrow E$, for
$\lambda\in X^*(T)_+$, is Zariski dense in $\Spec
S(\mathfrak{t})$. Thus $\psi$ is the unique isomorphism \eqref{HC}
such that \eqref{Verma} holds.

\begin{lem}\label{HCeq} The isomorphism $\psi$ defined in \eqref{HC}
  is $\Gamma$-equivariant.
\end{lem}
\begin{proof} 
  The group $\Aut(G_E)$ acts on $\mathfrak{g}_E$ and hence on
  $U(\mathfrak{g}_E)$ and $Z(\mathfrak{g}_E)$. If $V$ is an algebraic
  representation of $G_E$, then for $\gamma\in\Gamma$ we define
  $V^\gamma$ to be the action of $G_E$ on $V$ given by
  $(g,v)\mapsto \gamma^{-1}(g)\cdot v$. For $\lambda\in X^*(T)_+$, we
  have $V(\lambda)^\gamma\simeq V(\gamma\cdot\lambda)$. We deduce from
  relation \eqref{Verma} that
  \[\ev_{\lambda+\delta}(\gamma(\psi(z)))=\ev_{\gamma^{-1}\cdot\lambda+\delta}(\psi(z))=\chi_{\gamma^{-1}\cdot\lambda}(z),
    \quad \forall z\in Z(\mathfrak{g}_E).\] It follows from
  \eqref{need} that $Z(\mathfrak{g}_E)$ acts on
  $V(\gamma^{-1}\cdot\lambda)\simeq V(\lambda)^{\gamma^{-1}}$ via
  $\chi_\lambda(\gamma(z))$. Finally we have
  \[ \ev_{\lambda+\delta}(\gamma(\psi(z)))=\chi_\lambda(\gamma(z))\]
  for all $z\in Z(\mathfrak{g}_E)$, this implies
  $\gamma\circ\psi=\psi\circ\gamma$ and the result.
\end{proof}

Let $\mathfrak g$ be the Lie algebra of $G$ as an algebraic group over
$F$, let $U(\mathfrak{g})$ be its universal enveloping algebra and let
$Z(\mathfrak{g})$ be the center of $U(\mathfrak{g})$. Then
$\mathfrak g_E$, where the subscript $E$-denotes the extension of
scalars from $F$ to $E$, is the Lie algebra of $G_E$. It follows from
its universal property \cite[Lem.\,2.1.3]{dix} that
$U(\mathfrak g_E)= U(\mathfrak g)_E$.

\begin{lem}\label{centre_extend} Let $A,B$ be central algebras over a
  field $F$. The natural map
  \[Z(A)\otimes_F Z(B)\to Z(A\otimes_F B)\] is an isomorphism. In
  particular, the centre $Z(\mathfrak{g}_E)$ of $U(\mathfrak g_E)$ is
  equal to $Z(\mathfrak{g})_E$.
\end{lem}

\begin{proof} Let $x\in Z(A\otimes_F B)$ and let $(a_i)_{i\in I}$ be a
  basis of $A$ over $F$. Writing $x=\sum_{i\in I} a_i\otimes b_i$ and
  imposing $x(1\otimes b)=(1\otimes b)x$ yields
  $\sum_{i\in I} a_i\otimes (bb_i-b_ib)=0$, thus $b_i\in Z(B)$ for all
  $i\in I$ and $x\in A\otimes_F Z(B)$. Now pick a basis of $Z(B)$ over
  $F$ and repeat the argument to get $x\in Z(A)\otimes_F Z(B)$.
\end{proof}

\subsection{Definition of infinitesimal character}\label{Definf} Let us fix an embedding 
$\sigma: F\hookrightarrow L$ and choose an embedding
$\tau: E\hookrightarrow L$, such that $\tau|_F=\sigma$.

We have canonical identifications $\mathfrak{t}=X_*(T)\otimes_{\ZZ} E$
and $\that=X_*(\That)\otimes_\ZZ L= X^*(T)\otimes_{\ZZ} L$. Thus an
isomorphism of $L$-vector spaces
\[\mathfrak{t}\otimes_{E, \tau} L \cong \that^*,\] which 
is equivariant for $\Gamma$ and $W$ actions. By base changing
\eqref{HC} along $\tau$ and using Lemma \ref{centre_extend} we obtain
an isomorphism of $L$-algebras
\begin{equation}\label{dazed}
  \kappa_{\tau}: Z(\mathfrak{g})\otimes_{F, \sigma} L \overset{\cong}{\longrightarrow} S(\that^*)^W.
\end{equation}

Since $T$ is split over $E$ we have canonically
$X^*(T)= X^*(T\times_{E, \tau} L)$. If $V$ is an irreducible
representation of $G\times_{F, \sigma} L$ and $\lambda$ is the highest
weight of $V$ with respect to $(B, T)\times_{E, \tau} L$ then the
above identifications allows to view $\lambda_{\tau}$ as an element of
$\that$, and thus induces a homomorphism
$\ev_{\lambda_{\tau}}: S(\that^*)\rightarrow L$. Since
$G\times_{F, \sigma} L$ is split $V$ is absolutely irreducible and
thus by the same argument as in \eqref{central} we obtain a
homomorphism
$\chi_{V}: Z(\mathfrak g)\otimes_{F, \sigma} L \rightarrow L$.

\begin{lem}\label{chara_kappa_tau} The isomorphism $\kappa_{\tau}$ is
  uniquely characterised by the property
  \[ \ev_{\lambda_{\tau}+ \delta}\circ \kappa_{\tau} = \chi_{V}\] for
  all irreducible representations $V$ of $G\times_{F, \sigma} L$.
\end{lem}
\begin{proof} The assertion follows from \eqref{Verma} and Zariski
  density as explained after \eqref{Verma}.
\end{proof}

\begin{lem}\label{transform_kappa} For every $\gamma\in \Gamma$ the composition 
  \[ Z(\mathfrak{g})\otimes_{F, \sigma}
    L\overset{\kappa_{\tau}}{\longrightarrow} S(\that^*)^W
    \overset{\gamma\cdot}{\longrightarrow} S(\that^*)^W\] is equal to
  $\kappa_{\tau\circ\gamma^{-1}}$, where $\cdot$ denotes the action of
  $\Gamma$ on $S(\that^*)^W$.
\end{lem}
\begin{proof} Let $V$ be an irreducible representation of
  $G\times_{F, \sigma} L$. If $\lambda$ is the highest weight of $V$
  with respect to $(B, T)\times_{E, \tau} L$ then $\gamma(\lambda)$ is
  the highest weight of $V$ with respect to
  $(B, T)\times_{E, \tau\circ \gamma^{-1}} L$. This assertion follows
  from the definition of the action of $\Gamma$ on $X^*(T)$ after
  identifying $E$ and $L$ via $\tau$. The claim then follows from
  Lemma \ref{chara_kappa_tau}.
\end{proof}

\begin{cor} If $\lambda\in \that^{\Gamma}$ then the composition 
  \[ Z(\mathfrak{g})\otimes_{F, \sigma}
    L\overset{\kappa_{\tau}}{\longrightarrow} S(\that^*)^W
    \overset{\ev_\lambda}{\longrightarrow} L\] is independent of the
  choice of embedding $\tau$.
\end{cor}

\begin{lem}\label{comp_ind} The composition 
  \[ Z(\mathfrak{g})\otimes_{F, \sigma}
    L\overset{\kappa_{\tau}}{\longrightarrow} S(\that^*)^W
    \underset{\eqref{eqn_chev}}{\overset{\cong}{\longrightarrow}}
    S(\ghat^*)^{\Ghat}
    \underset{\eqref{def_theta}}{\overset{\theta}{\longrightarrow}}
    E\otimes \OO_X(X) \overset{m_\tau}{\longrightarrow} \OO_X(X),\]
  where the last arrow is the map $x\otimes a \mapsto \tau(x) a$, is
  independent of the choice of $\tau$ above $\sigma$.
\end{lem}
\begin{proof} This follows from Lemma \ref{transform_kappa} and
  $\Gamma$-equivariance of \eqref{eqn_chev} and \eqref{def_theta}.
\end{proof} 

\begin{defi}\label{def_ringhom_sigma} Let $P$ be a $\Ghat_X$-torsor over $X$ and let $\rho:\Gal_F\rightarrow \LPad(X)$ be an admissible representation.
  If $\sigma: F\hookrightarrow L$ is an embedding of fields then we
  define a homomorphism of $L$-agebras
\[ \zeta_{\rho, \sigma}: Z(\mathfrak g)\otimes_{F, \sigma} L \rightarrow \OO_X(X)\]
as the composition of the maps in Lemma \ref{comp_ind}.
\end{defi}
\begin{remar} We remind the reader if $P= \Ghat_X$ then
  $\LPad(X)=\LG_f(\OO_X(X))$ and $\rho$ is an admissible
  representation in the sense of section \ref{adm_reps}.
\end{remar}

Let $\Res_{F/\Qp}\mathfrak{g}$ be the Lie algebra of $\Res_{F/\Qp}
G$. We may identify $\Res_{F/\Qp}\mathfrak{g}$ with $\mathfrak g$ as a
$\Qp$-Lie algebra. Since
\[(\Res_{F/\Qp} G)\times_{\Qp} L\cong \prod_{\sigma: F \hookrightarrow
    L} G\times_{F, \sigma} L,\] using Lemma \ref{centre_extend}, we
obtain isomorphisms of $L$-algebras
\[U(\Res_{F/\Qp}\mathfrak g)\otimes_{\Qp} L\cong \bigotimes_{\sigma: F
    \hookrightarrow L} U(\mathfrak g)\otimes_{F, \sigma} L\]
\[ Z(\Res_{F/\Qp}\mathfrak g)\otimes_{\Qp} L\cong \bigotimes_{\sigma:
    F \hookrightarrow L} Z(\mathfrak g)\otimes_{F, \sigma} L.\]

\begin{defi}\label{def_ringhom_abs} Let $P$ be a $\Ghat_X$-torsor over $X$ and let $\rho:\Gal_F\rightarrow \LPad(X)$ be an admissible representation.
  Then we define an $L$-algebra homomorphism
  \[\zeta_{\rho}: Z(\Res_{F/\Qp}\mathfrak g)\otimes_{\Qp} L
    \rightarrow \OO_X(X)\] as
  $\zeta_{\rho}\coloneqq \otimes_{\sigma} \zeta_{\rho, \sigma}$ using
  the isomorphism above and Definition \ref{def_ringhom_sigma}.
\end{defi}

\begin{lem}\label{base_change_L} If $Y\rightarrow X$ is a map of rigid varieties over $L$ then 
  $\zeta_{\rho_Y}$ is equal to $\zeta_{\rho}$ composed with the map
  $\OO_X(X)\rightarrow \OO_Y(Y)$.
\end{lem}

\begin{proof} This follows from Lemma \ref{base_change_theta}.
\end{proof}
 
\subsection{Modification for $C$-groups}\label{sec_Cgr}
There are two issues at hand. Firstly, one expects that the Galois
representations attached to $C$-algebraic automorphic forms take
values in the $C$-group and not the $L$-group, see
\cite[Conj.\,5.3.4]{beegee}. We would like to consider a local
$p$-adic analog of this situation, but our Definition
\ref{def_ringhom_abs} does not cover it. Secondly, if we take
$F=L=\Qp$, $X$ a point, $G=\GL_2$, so that $\LGf= \GL_2$, and
\[\rho: \Gal_{\Qp} \rightarrow \GL_2(\Qp), \quad g\mapsto \bigl
  (\begin{smallmatrix} \chi_{\cyc}(g)^{a} & 0 \\ 0 &
    \chi_{\cyc}(g)^b\end{smallmatrix}\bigr),\] where $a > b$ are
integers, then the Sen operator is the matrix
$\bigl (\begin{smallmatrix} a & 0 \\ 0 & b\end{smallmatrix}\bigr)$,
but the character $\zeta_{\rho}: Z(\mathfrak g)\rightarrow \Qp$ is not
an infinitesimal character of an algebraic representation of
$\GL_2$. The problem is caused is by the shift by $\delta$ in
\eqref{Verma}. We will resolve both of these issues by modifying
Definition \ref{def_ringhom_abs}.

We use the notation introduced in section \ref{sec_Cgr_sub}. Let
$\ghat^T$ be the Lie algebra of $\Ghat^T$. We may identify
$\ghat^T= \mathfrak \ghat \oplus \Lie \Gm$ as $F$-vector spaces. Since
taking square roots out of $2\delta$ on the Lie algebra level is just
dividing by $2$, the map
\[\ghat^T\rightarrow \mathfrak \ghat \oplus \Lie \Gm, \quad (g, t)
  \mapsto (g+t\delta, t)\] is a $\Gamma$-equivariant isomorphism of
Lie algebras. Let $\alpha: \ghat^T\rightarrow \ghat$ be the
composition of the above map with the projection to $\ghat$. The map
induces a $\Ghat^T$-equivariant homomorphism of $L$-algebras
$\alpha: S(\ghat^*)\rightarrow S((\ghat^T)^*)$.  Since $\Gm$ acts
trivially on $\That$, it follows from the Chevalley's restriction
theorem that $S(\ghat^*)^{\Ghat}= S(\ghat^*)^{\Ghat^T}$. We thus
obtain a $\Gamma$-equivariant homomorphism of $L$-algebras
\[\alpha: S(\ghat^*)^{\Ghat} \rightarrow S((\ghat^T)^*)^{\Ghat^T}.\]
Let $P$ be a $\Ghat_X$-torsor on $X$ and let
$\rho: \Gal_F \rightarrow \CPad(X)$ be an admissible representation.
We define
\begin{equation}\label{def_theta_prime}
  \theta': S(\ghat^*)^{\Ghat}\overset{\alpha}{\longrightarrow} S((\ghat^T)^*)^{\Ghat^T}
  \overset{\theta^T}\longrightarrow
  E\otimes \OO_X(X)
\end{equation}
where $\theta^T$ is the map defined in \eqref{def_theta} with $G^T$
instead of $G$ and regarding $\rho$ as a representation of
$\rho: \Gal_F\rightarrow {}^L P^{T, \mathrm{ad}}(X)$, see Remark
\ref{CPad}. We note that $\theta'$ is $\Gamma$-equivariant since both
$\alpha$ and $\theta^T$ are.

\begin{defi}\label{def_ringhom_C} If $\rho: \Gal_F\rightarrow
  \CPad(X)$ is an admissible representation and $\sigma:
  F\hookrightarrow L$ is  an embedding of fields then we define
  a homomorphism of $L$-algebras
  \[ \zeta_{\rho, \sigma}^C: Z(\mathfrak g)\otimes_{F, \sigma} L
    \rightarrow \OO_X(X)\] as the composition of the maps in Lemma
  \ref{comp_ind}, but using $\theta'$ instead of $\theta$.  We define
  a homomorphism of $L$-algebras
  \[\zeta_{\rho}^C: Z(\Res_{F/\Qp} \mathfrak g)\otimes_{\Qp} L
    \rightarrow \OO_X(X),\] as
  $\zeta_{\rho}^C\coloneqq \otimes_{\sigma} \zeta_{\rho, \sigma}^C$.
\end{defi} 

\begin{remar} We remind the reader if $P= \Ghat_X$ then
  $\CPad(X)=\CG_f(\OO_X(X))$ and $\rho$ is an admissible
  representation in the sense of section \ref{adm_reps}.
\end{remar}

\begin{remar}\label{dcyc} Although the definition makes sense for all admissible $\rho$, we will 
  apply it to those $\rho$, where the composition
  $d\circ \rho: \Gal_F \rightarrow \OO_X(X)^{*}$, where $d$ is defined
  in \eqref{twisted_d}, is equal to $\chi_{\cyc}$, since the Galois
  representations associated to $C$-algebraic automorphic forms should
  satisfy this condition according to \cite[Conj.\,5.3.4]{beegee}.  In
  this case, if $U=\Sp A\rightarrow X$ is an affinoid such that $P|_U$
  is trivial, the Sen operator
  $\Theta^T_{\mathrm{Sen}, U} \in (\Cp\wtimes A)\otimes_L \ghat^T$ is
  of the form $(M, 1\otimes 1)$, with
  $M\in (\Cp\wtimes A)\otimes_L \ghat$,
  $1\otimes 1 \in (\Cp\wtimes A)\otimes \Lie \Gm$ and
  $\id\otimes \alpha$ maps $\Theta^T_{\mathrm{Sen}, U}$ to
  $M+1\otimes \delta \in (\Cp\wtimes A)\otimes_L \ghat$.
\end{remar}

\begin{lem}\label{base_change_C} If $Y\rightarrow X$ is a map of rigid varieties over $L$ then 
  $\zeta_{\rho_Y}^C$ is equal to $\zeta_{\rho}^C$ composed with the
  map $\OO_X(X)\rightarrow \OO_Y(Y)$.
\end{lem}

\begin{proof}
  This follows from Lemma \ref{base_change_L}.
\end{proof}

If there is $\tilde{\delta}\in X_*(\That)$ such that its image in
$X_*(\That/Z_{\Ghat})$ is equal to $\delta_{\mathrm{ad}}$ then we have
an isomorphism
\begin{equation}\label{twisting_el} 
  \tw_{\tilde{\delta}}: \Ghat^T\cong \Ghat \times \Gm,\quad (g, t) \mapsto (g \tilde{\delta}(t), t)
\end{equation}
If $\tilde{\delta}$ is $\Gamma$-invariant then \eqref{twisting_el}
yields $\CG_f \cong \LG_f\times \Gm$,
$P_f^T\cong P_f \times \mathbb{G}_{m, X}$ and
\begin{equation}\label{twisting_el2}
  \tw_{\tilde{\delta}}: \CPad(X)\cong \LPad(X)\times \OO_X(X)^*
\end{equation}
 
If $\rho: \Gal_F \rightarrow \LPad(X)$ is an admissible representation
then we may define $\rho^C: \Gal_F \rightarrow \CPad(X)$ via
$\rho^C\coloneqq \tw_{\tilde{\delta}}^{-1} \circ (\rho\boxtimes
\chi_{\cyc})$. We note that $d\circ \rho^C =\chi_{\cyc}$ and, if
$U=\Sp A\rightarrow X$ is an affinoid such that $P|_U$ is trivial,
using Remark \ref{dcyc} we obtain that
\begin{equation}\label{CL}
  (\id\otimes \alpha)(\Theta^T_{\mathrm{Sen},U})= \Theta_{\mathrm{Sen},U} - 1\otimes \tilde{\delta}+ 1\otimes \delta,
\end{equation}
where $\Theta_{\mathrm{Sen},U}$ is the Sen operator attached to $\rho$
in Lemma \ref{tannaka}, and we consider
$\delta, \tilde{\delta}\in \that$ via the identification
$\that=X_*(\That)\otimes_\ZZ L$.

Going back to the $\GL_2$ example at the beginning of the subsection,
we see that if we choose
$\tilde{\delta}(t)\coloneqq \bigl (\begin{smallmatrix} t^{n+1} & 0 \\
  0 & t^n\end{smallmatrix}\bigr)$ for some $n\in \ZZ$ then
$\zeta^C_{\rho^C}$ is equal to the infinitesimal character of
$\Sym^{a-b-1}\otimes \det^{b-n}$.

\subsection{The archimedean case}\label{sec:real}

We check that our definition is compatible with the archimedean case.

Assume now that $F$ is an archimedean local field and let $\overline{F}$ be
some algebraic closure of $F$, so that $[\overline{F}:F]$ is $1$ or $2$. We
recall that the Weil group of $F$ is the semidirect product
$W_F=\overline{F}^\times\rtimes\Gal_F$. If $F=\R$, we denote by $c$ the non
trivial element of $\Gal_F$.

Let $\rho : W_F\rightarrow {}^LG(\C)$ be some admissible
representation. The restriction of $\rho$ to $\overline{F}^\times$
takes values in $\Ghat(\C)$ and, up to conjugate by an element of
$\Ghat(\C)$, we can assume that
$\rho(\overline{F}^\times)\subset\That(\C)$. Let $\sigma_1$ and
$\sigma_2$ be the two isomorphisms of $\overline{F}$ with
$\C$. Identifying $X_*(\That)\otimes\C$ with the Lie algebra of
$\That(\C)$, we see that there exists two elements
$\lambda_{\sigma_1}$ and $\lambda_{\sigma_2}$ in $X_*(\That)\otimes\C$
such that
\[
  \rho(z)=\exp(\log(\sigma_1(z))\lambda_{\sigma_1}+\log(\sigma_2(z))\lambda_{\sigma_2})\]
for $z\in\overline{F}^\times$. We can assume that
$\lambda_{\sigma_1}-\lambda_{\sigma_2}\in X_*(\That)$ so that this
quantity doesn't depend on the choice of the branch of $\log$.

If $F=\R$, then we have $\lambda_\tau=\Ad(\rho(c))\lambda_\sigma$ so
that $\lambda_\tau$ and $\lambda_\sigma$ defines the same element of
$(X_*(\That)\otimes\C/W)$. Evaluation on this element gives us a linear map
$\theta : S(\hat{\mathfrak{t}}^*)^W\rightarrow \C$ which, by
composition with Harish-Chandra isomorphism
$Z(\mathfrak{g})_{\C}\simeq S(\hat{\mathfrak{t}}^*)^W$ gives rise to a
character $\zeta_\rho : Z(\mathfrak{g})_{\C}\rightarrow\C$.

If $F=\C$, then we use the elements $\lambda_\sigma$ and
$\lambda_\tau$ in $X_*(\That)\otimes\C$ to define a character
$\zeta_{\rho,\sigma} :
Z(\mathfrak{g})\otimes_{\overline{F},\sigma}\C\rightarrow\C$ for each
embedding $\sigma$ of $\overline{F}$ in $\C$. We finally define
\[ \zeta_\rho :
  Z(\Res_{F/\R}\mathfrak{g})_{\C}\xrightarrow{\zeta_{\rho,\sigma_1}\otimes\zeta_{\rho,\sigma_2}}\C.\]

\begin{remar}\label{rm:localglobalR}
  If $\rho$ is associated to an irreducible representation of $G(F)$,
  then the center $Z(\Res_{F/\R}\mathfrak{g})$ acts on $\pi$ by
  $\zeta_\rho$ (see Propositions 7.4 and 7.10 in \cite{VoganLL} which
  follow from \cite{LanglandsRparam}).
\end{remar}

\section{Hodge--Tate representations}\label{sec_HT} 
We will study Hodge--Tate representations following \cite{serreHT},
\cite{BC} and \cite[\S 2.4]{beegee}\footnote{We note that the
  published version of \S 2.4 in \cite{beegee} has been updated in
  \url{https://arxiv.org/pdf/1009.0785.pdf}}.  Let $L$ be a finite
extension of $\Qp$ and let $V$ be a finite dimensional $L$-vector
space with a continuous $L$-linear action of $\Gal_{F}$. Consider the
semi-linear diagonal action of $\Gal_F$ on $W=V\otimes_{\Qp}\Cp$ and
let $W(i)$ be the $i$th Tate twist of $W$. The natural map
\[\bigoplus_{i\in \ZZ} \Cp(i)\otimes_{\Qp} W(-i)^{\Gal_{F}}\to W\]
is injective and we say that $V$ is Hodge--Tate if this map is an isomorphism. Suppose that this is the case and let  
$W_i$ be the image of $\Cp(i)\otimes_{\Qp} (W(-i))^{\Gal_{F}}$ in 
$W$, i.e. the $\Cp$-span of the set of $v\in W$ such that 
\[g v= \chi^i(g) v, \quad \forall g\in \Gal_{F},\] where $\chi$ is the
$p$-adic cyclotomic character. Then $W=\oplus_{i\in \Z} W_i$ and the
Hodge--Tate cocharacter $\mu_V$ of $V$ is the algebraic morphism
$\mathbb{G}_{m,\Cp}\rightarrow \GL(W)$ defined by
\[ \mu_V(z)=\sum_{i\in\Z}z^i\id_{W_i}.\] Note that the Sen operator of
$V$ is simply multiplication by $i$ on $W_i$, in particular
\begin{equation} \label{HTcoc1}
\Lie(\mu_V)(1)=\Theta_{\Sen, V}
\end{equation}

Let $H$ be a reductive group defined over $L$ and let $\Rep(H)_L$ be
the category of algebraic representations of $H$ defined over $L$.
Let $\rho: \Gal_{F} \rightarrow H(L)$ be a continuous
representation. We say that $\rho$ is Hodge--Tate if the action of
$\Gal_{F}$ on $V$ obtained by composing $\rho$ with
$r: H\rightarrow \GL(V)$ is Hodge--Tate for every
$(r,V)\in \Rep(H)_L$.

Let $\rho: \Gal_{F} \rightarrow H(L)$ be a Hodge--Tate representation.
Given a $\Cp$-algebra $R$ and $z\in R^{\times}$, for each
$(r,V)\in \Rep(H)_L$ we obtain an automorphism
${\rm id}\otimes \mu_V(z)$ of
$R\otimes_{\Cp}(\Cp\otimes_{\Qp} V)\simeq R\otimes_{\Qp} V\simeq
(R\otimes_{\Qp} L)\otimes_L V$, compatible with tensor products when
we vary $V$. It follows that there is a unique
$\mu_{\rho}(z)\in H(R\otimes_{\Qp} L)$ such that
${\rm id}\otimes \mu_V(z)=r(\mu_{\rho}(z))$ for all
$(V,r)\in \Rep(H)_L$. Varying $R$ and observing that
$ H(R\otimes_{\Qp} L)= (\Res_{L/\Qp}H)_{\Cp}(R)$, we obtain a
cocharacter
\[\mu_\rho : \mathbb{G}_{m,\Cp}\rightarrow (\Res_{L/\Qp}H)_{\Cp}.\]
Using the decomposition
\[
  (\Res_{L/\Qp}H)_{\Cp}\simeq\prod_{\upsilon\in\Hom(L,\Cp)}H\times_{L,
    \upsilon} \Cp,\] we can write
$\mu_\rho=(\mu_{\rho,\upsilon})_{\upsilon\in\Hom(L,\Cp)}$ with
$\mu_{\rho,\upsilon}: \mathbb{G}_{m,\Cp}\rightarrow H\times_{L,
  \upsilon} \Cp$. Note that since $\Gm$ is connected,
$\mu_{\rho,\upsilon}$ factors through the connected component of
$H\times_{L, \upsilon} \Cp$.

Arguing as in the proof of lemma \ref{tannaka} one associates to {\it
  any} continuous representation $\rho: \Gal_{F} \rightarrow H(L)$ a
Sen operator
$\Theta_{\rm Sen, \rho}\in (\Cp\otimes_{\Qp} L)\otimes_{L}
\mathfrak{h}\simeq \Cp\otimes_{\Qp} \mathfrak{h}$, where
$\mathfrak{h}$ is the Lie algebra of $H$. When $\rho$ is Hodge--Tate
relation $(1)$ combined with the Tannakian description of both
$\Theta_{\rm Sen, \rho}$ and $\mu_\rho$ yields\footnote{Once we
  identify the $\Qp$-Lie algebras $\Lie({\rm Res}_{L/\Qp}(H))$ and
  $\mathfrak{h}$.}
\begin{equation} \label{HTcoc2}
\Lie(\mu_\rho)(1)=\Theta_{\rm Sen, \rho}
\end{equation}

\begin{lem}\label{charHT}
  A continuous representation $\rho: \Gal_{F} \rightarrow H(L)$ is
  Hodge--Tate if and only if there is
  $\mu: \mathbb{G}_{m,\Cp}\rightarrow (\Res_{L/\Qp}H)_{\Cp}$ such that
  $\Theta_{\rm Sen, \rho}=\Lie(\mu)(1)$, in which case we necessarily
  have $\mu=\mu_{\rho}$.
\end{lem}

\begin{proof} The only fact which hasn't already been explained is
  that the existence of $\mu$ forces $\rho$ being Hodge--Tate. Let
  $(r,V)\in {\rm Rep}_L(H)$. We want to prove that $V$ endowed with
  the action of $\Gal_F$ via composition with $\rho$ isHodge--Tate.
  The cocharacter $\mu$ and the representation $r$ give rise to an
  algebraic morphism
  $r\circ \mu: \Cp^{\times}\to
  (\Res_{L/\Qp}H)_{\Cp}(\Cp)=H(\Cp\otimes_{\Qp} L)\to
  \GL(\Cp\otimes_{\Qp} V)$ and
  $\Theta_{\Sen, V}=\Lie(r)(\Theta_{\rm Sen, \rho})(1)=\Lie(r\circ
  \mu)(1)$. It follows that $\Theta_{\Sen, V}$ is semi-simple with
  integer eigenvalues and $V$ is Hodge--Tate by the corollary to
  theorem $6$ in \cite{SenInv}.
\end{proof}

\subsection{Galois representations valued in $L$-groups}
Now we assume that $G$ is a connected reductive group over $F$ and let
$\rho : \Gal_{F}\rightarrow \LG(L)$ be an admissible representation,
so that we are in the situation of Section \ref{families_gal} with $X$
equal to a point, so that $\Gamma(X, \OO_X)=L$. We assume that $\rho$
is Hodge--Tate and apply the above discussion to $H=\LGf$.  Let
$T\subset G_{\overline{F}}$ a maximal split torus and let
$\widehat{T}$ be a dual torus in $\widehat{G}$. For each $\upsilon$,
the cocharacter $\mu_{\rho,\upsilon}$ is conjugate inside
$\widehat{G}(\Cp)$ to a cocharacter
$\mathbb{G}_{m,\Cp}\rightarrow \widehat{T}_{\Cp}$, which is uniquely
determined up to the action of the Weyl group $W$. Consequently we
obtain a well defined element
\[ \nu_{\rho,\upsilon}\in X^*(T)/W=X_*(\widehat{T})/W.\]
We may think of $\nu_{\rho, \upsilon}$ as a $W$-orbit of an element in 
$\that\otimes_{L, \upsilon} \Cp= X_*(\widehat{T})\otimes_{\ZZ} \Cp$.

\begin{prop}\label{HTcochar} The composition 
\begin{equation}\label{drinking_beer_on_public_trains}
  S(\that^*)^W 
  \underset{\eqref{eqn_chev}}{\overset{\cong}{\longrightarrow}} S(\ghat^*)^{\Ghat}
  \underset{\eqref{def_theta}}{\overset{\theta}{\longrightarrow}} \Cp\otimes L
  \overset{m_\upsilon}{\longrightarrow}\Cp
\end{equation}
where the last arrow is given by $x\otimes y\mapsto x\upsilon(y)$, is
equal to the evaluation map at
$\nu_{\rho, \upsilon} \in (\that\otimes_{L, \upsilon} \Cp))/W$.
\end{prop} 

\begin{proof} Let
  $\Theta_{\mathrm{Sen}}\in (\Cp\otimes L)\otimes_L \ghat=
  \Cp\otimes_{\Qp} \ghat$ be the Sen operator defined in Lemma
  \ref{tannaka}. Since $X$ is a point we we will suppress the affinoid
  $U$ from the notation. Since
  $\Theta_{\mathrm{Sen}} =\Lie(\mu_{\rho})(1)$, the image of
  $\Theta_{\mathrm{Sen}}$ in $\ghat\otimes_{L, \upsilon} \Cp$ is equal
  to $\Lie(\mu_{\rho, \upsilon})(1)$.  Since $\mu_{\rho, \upsilon}$
  and $\nu_{\rho, \upsilon}$ are conjugate by an element of
  $\Ghat(\Cp)$, the proposition is proved.
\end{proof}

Recall that the group $\Gal_F$ acts on the pinned root datum, which
induces an action on $X_*(\That)/W$, which we denote by $\cdot$. We
can now reprove \cite[Lem.\,2.4.1]{beegee}.

\begin{prop}\label{beegee_eq}
  For $g\in\Gal_F$, we have
  $\nu_{\rho,g \upsilon}=g\cdot \nu_{\rho,\upsilon}$ in
  $X_*(\widehat{T})/W$.
\end{prop}

\begin{proof} 
  Let $W=(\Cp\otimes L)\otimes_L \ghat$. We refer to the discussion
  preceding lemma \ref{equiv} for the definitions of the various
  actions of $\Gal_F$ and $\Ghat(\Cp\otimes L)$ on $W$ used in this
  proof.  For each embedding $\upsilon: L\to \Cp$ let
  $W_v=\Cp\otimes_{L,\upsilon} \ghat$.  Let $X_\upsilon\in W_\upsilon$
  be the image of $X\in W$ via the natural isomorphism
  $W\simeq \prod_{\upsilon: L\to \Cp} W_\upsilon$. One easily checks
  that the map
  $\tilde{\gamma}_\upsilon: W_\upsilon\to W_{\gamma \upsilon}$ defined
  by $\tilde{\gamma}_\upsilon(x\otimes y)=\gamma(x)\otimes y$
  satisfies
  \begin{equation} \label{key}
   (\gamma*X)_{\gamma \upsilon}=\tilde{\gamma}_\upsilon(X_\upsilon) 
 \end{equation}
 In particular, if
 $X=(1\otimes u_\upsilon)_{\upsilon}\in \prod_{\upsilon: L\to \Cp}
 W_\upsilon\simeq W$, with $u_\upsilon\in \ghat$, then for all
 $\gamma\in \Gal_F$ we have
 $[\gamma*(\gamma\cdot X)]_{\gamma \upsilon}=1\otimes\gamma\cdot
 u_\upsilon$, where $\gamma\cdot u_\upsilon$ is the action of $\gamma$
 on $u_\upsilon$ induced by $\mu_G$.

 Now let $\Theta=\Theta_{\Sen, U}$ and consider
 \[X=(1\otimes \Lie(\nu_{\rho, \upsilon})(1))_\upsilon\in
   \prod_{\upsilon: L\to \Cp} W_v\simeq W.\] By definition of
 $\nu_{\rho,\upsilon}$ and relation \eqref{HTcoc2}, there is
 $g\in \Ghat(\Cp\otimes L)\simeq \prod_{\upsilon: L\to \Cp}
 \Ghat(\Cp)$ such that $\Theta={\rm Ad}(g)(X)$. Lemma \ref{equiv}
 combined with relation \eqref{adj} show that there exists
 $h\in \Ghat(\Cp\otimes L)$ (more precisely
 $h=g^{-1}c_{\gamma} sgs^{-1}$, with $s=\mu_G(\bar{\gamma})$) such
 that
\begin{equation}\label{ugly} 
{\rm Ad}(h)(\gamma*(\gamma\cdot X))=X.
\end{equation}
Writing
$h=(h_\upsilon)_\upsilon\in \prod_{\upsilon: L\to \Cp} \Ghat(\Cp)$,
projecting relation \eqref{ugly} in $W_{\gamma \upsilon}$ and taking
into account the discussion above yields
\[{\rm Ad}(h_{\gamma \upsilon})(1\otimes \gamma\cdot \Lie(\nu_{\rho, \upsilon})(1))=1\otimes \Lie(\nu_{\rho, \gamma \upsilon})(1).\]
It follows that $\gamma\cdot \nu_{\rho, \upsilon}$ and 
$\nu_{\rho, \gamma \upsilon}$ are the same in $X_*(\That)/W$.
\end{proof}

\begin{remar}\label{orbits} Let us note that $\Gal_F$-orbits of embeddings $\upsilon: L\hookrightarrow \Cp$ are 
  canonically in bijection with the set of embeddings
  $\tau: F\hookrightarrow L$. Concretely, $F=\upsilon(L)^{\Gal_F}$ and
  $\tau= \upsilon^{-1}|_F$. Similarly, $\Gal_E$-orbits of embeddings
  $\upsilon: L\hookrightarrow \Cp$ are in bijection with the set of
  embeddings $\tau: E\hookrightarrow L$. It follows from Proposition
  \ref{beegee_eq} that if $\upsilon, \upsilon': L\hookrightarrow \Cp$
  lie in the same $\Gal_E$-orbit then
  $\nu_{\rho, \upsilon}=\nu_{\rho, \upsilon'}$.
\end{remar}

\subsection{Galois representations valued in $C$-groups}
We will consider the set up of \S \ref{sec_Cgr_sub}. Since $\Gm$ acts
trivially on $\That$, $\That\times \Gm$ is a maximal torus of
$\Ghat^T$, and $\Bhat\rtimes \Gm$ is a Borel subgroup of $\Ghat^T$. We
also recall that $\CG$ is the $L$-group of $G^T$, which is a central
extension of $G$ by $\Gm$ over $F$, so that the previous discussion in
this section applies.

\begin{prop}\label{inf_HT} Let $\rho: \Gal_F\rightarrow \CG_f(L)$ be an admissible representation such that 
  $d\circ \rho= \chi_{\cyc}$, let $\sigma: F\hookrightarrow L$,
  $\tau: E\hookrightarrow L$, $\upsilon: L\hookrightarrow \Cp$ be
  field embeddings such that $\upsilon\circ\tau=\id_E$ and
  $\tau|_F=\sigma$.

  If $\rho$ is Hodge--Tate and the image of $\nu_{\rho, \upsilon}$
  under
  \begin{equation}\label{S6}
    X_*(\That\times \Gm)/W\rightarrow X_*(\That)/W=X^*(T)/W=X^*(T\times_{E, \tau} L)/W
  \end{equation}
  contains a character $\lambda_{\tau}$, which is dominant with
  respect to $B\times_{E, \tau} L$, then $\zeta_{\rho, \sigma}^C$,
  defined in Definition \ref{def_ringhom_C}, is an infinitesimal
  character of irreducible representation of highest weight
  $\lambda_{\tau}$ with respect to $B\times_{E, \tau} L$.

  If $\zeta^C_{\rho, \sigma}$ is an infinitesimal character of
  irreducible algebraic representation $V_{\sigma}$ of
  $G\times_{F, \sigma} L$ for every $\sigma: F\hookrightarrow L$, and
  the highest weight $\lambda_{\tau}$ of $V_{\sigma}$ with respect to
  $B\times_{E, \tau} L$ is regular for all $\tau$ (equivalently, for
  at least one $\tau$ above every $\sigma$) then $\rho$ is Hodge--Tate
  and $\nu_{\rho, \upsilon}$ maps to the orbit of $\lambda_{\tau}$
  under \eqref{S6}, where $\upsilon\circ \tau|_E=\id_E$.
\end{prop}
\begin{proof} We first recall the definition of
  $\zeta_{\rho, \sigma}^C$. Let
  $\Theta_{\mathrm{Sen}}^T\in (\Cp\otimes L)\otimes_L \ghat^T$. As
  explained in Remark \ref{dcyc}, $\Theta_{\mathrm{Sen}}^T$ is of the
  form $(M, 1\otimes 1)$ with $M\in (\Cp\otimes L)\otimes_L
  \ghat$. Let $\Theta^T_{\mathrm{Sen},\upsilon}$ be the image of
  $\Theta_{\mathrm{Sen}}^T$ in $\ghat^T\otimes_{L, \upsilon}\Cp$ and
  let $M_{\upsilon}$ be the image of $M$ in
  $\ghat^T\otimes_{L, \upsilon}\Cp$ under the base change along
  $m_{\upsilon}: \Cp\otimes L\rightarrow \Cp$, so that
  $\Theta_{\mathrm{Sen}, \upsilon}^T=(M_{\upsilon}, 1\otimes 1)$.

  Then $\zeta_{\rho, \sigma}^C$ is equal to the composition
  \[ Z(\mathfrak{g})\otimes_{F, \sigma}
    L\overset{\kappa_{\tau}}{\longrightarrow} S(\that^*)^W
    \overset{\cong}{\longrightarrow} S(\ghat^*)^{\Ghat}
    \overset{\ev_{M_{\upsilon}+\delta}}{\xrightarrow{\hspace*{1cm}}}
    \upsilon(L) \overset{\upsilon^{-1}}{\longrightarrow} L.\] Let us
  point out that \eqref{drinking_beer_on_public_trains} in our
  situation is the map
  \[S((\that\oplus \Lie \Gm)^*)^W \overset{\cong}{\longrightarrow}
    S((\ghat \oplus \Lie \Gm)^*)^{\Ghat^T}
    \overset{\ev_{(M_{\upsilon}, 1)}}{\xrightarrow{\hspace*{1cm}}}
    \upsilon(L). \]

  Let us assume $\rho$ is Hodge--Tate and that
  $\lambda^T\in \nu_{\rho, \upsilon}$ maps to a dominant weight
  $\lambda \in X^*(T\times_{E, \tau} L)$ with respect to
  $B\times_{E, \tau} L$.  We will denote the images of $\lambda$ in
  $X_*(\That)$ and $X^*(T)$ by the same letter.  It follows from
  Proposition \ref{HTcochar} that
  $\zeta_{\rho, \sigma}^C= \ev_{\lambda+\delta}\circ\kappa_{\tau}$ and
  Lemma \ref{chara_kappa_tau} implies that $\zeta_{\rho, \sigma}^C$ is
  the infinitesimal character of $V(\lambda)$.

  Let us assume that $\zeta_{\rho, \sigma}^C$ is equal to the
  infinitesimal character of irreducible representation $V_{\sigma}$
  of $G\times_{E, \sigma} L$, which is of highest weight
  $\lambda_{\tau}$ with respect to $B\times_{E, \tau} L$. Let us
  additionally assume that $\lambda_{\tau}$ is regular.  Lemma
  \ref{chara_kappa_tau} implies that the maps $\ev_{\lambda_{\tau}}$
  and $\ev_{M_{\upsilon}}$ coincide, when restricted to
  $S(\ghat^*)^{\Ghat}$. This implies that if we write
  $M_{\upsilon}= x_s + x_u\in \ghat\otimes _{L, \upsilon}\Cp$, where
  $x_s$ is semisimple, $x_u$ is nilpotent and $[x_s, x_u]=0$, then
  $x_s$ is conjugate to $\lambda_{\tau}$ by an element of
  $\Ghat(\Cp)$, see for example \cite[Prop.\,10.3.1]{Dmod}. Since
  $\lambda_{\tau}$ is regular, $x_u=0$ and thus $M_{\upsilon}$ is
  conjugate to $\lambda_{\tau}$ by an element of $\Ghat(\Cp)$.  This
  implies that there is
  $\mu'_{\upsilon}: \mathbb{G}_{m,\Cp} \rightarrow \Ghat^T\times_{L,
    \upsilon}\Cp$, such that $d\circ \mu'_{\upsilon}= \id_{\Gm}$ and
  $\Lie(\mu'_{\upsilon})= \Theta^{T}_{\mathrm{Sen},\upsilon}$. Thus
  $\mu\coloneqq (\mu'_{\upsilon})_{\upsilon}: \mathbb{G}_{m,
    \Cp}\rightarrow (\Res_{L/ \Qp} \Ghat^T)_{\Cp}$ satisfies
  $\Lie(\mu')=\Theta_{\mathrm{Sen}}^T$. It follows from lemma
  \ref{charHT} that $\rho$ is Hodge--Tate and $\mu'=\mu_{\rho}$.  On
  then checks as in the proof of Proposition \ref{HTcochar} that
  $\nu_{\rho, \upsilon}$ maps to the $W$-orbit of $\lambda_{\tau}$
  under \eqref{S6}.
\end{proof}

\begin{lem} \label{deptrace} Let $\rho, \rho': \Gal_F\to \CG_f(L)$ be
  admissible representations, such that
  $d\circ \rho=d\circ\rho'=\chi_{\cyc}$. We write
  $\rho(\gamma)=(c_{\gamma}, \bar{\gamma}),
  \rho'(\gamma)=(c'_{\gamma}, \bar{\gamma})$ for all
  $\gamma\in \Gal_F$, where $c_{\gamma}, c'_{\gamma}\in \Ghat^T(L)$.
  Assume that
  \[\tr_V(r(c_{\gamma}))=\tr_V(r(c'_{\gamma})), \quad \forall \gamma \in \Gal_E,\] 
  for all algebraic representations $(r,V)$ of $\Ghat^T$. Then
  $\zeta_{\rho}^C=\zeta_{\rho'}^C$.
\end{lem}

\begin{proof} Let
  $\Theta=\Theta^T_{\Sen, \rho}, \Theta'=\Theta^T_{\Sen, \rho'} \in
  (\Cp\otimes L)\otimes_{L} \ghat^T$ be the respective Sen operators,
  where we consider $\CG_f$ as the $L$-group of $G^T$.  We claim that
  $f(\Theta)=f(\Theta')$ for all $f\in S((\ghat^T)^*)^{\Ghat^T}$. The
  claim implies that the maps $\theta'$ in \eqref{def_theta_prime}
  associated to representations $\rho$ and $\rho'$ coincide, which
  implies the assertion.

  To prove the claim we note that the polynomial function
  $x\mapsto \tr_V(\Lie(r)(x)^k)$ is an element of
  $S((\ghat^T)^*)^{\Ghat^T}$ and it follows from Chevalley's
  restriction theorem that these functions, for all algebraic
  representations $(r, V)$ of $\Ghat^T$ and all $k\ge 1$, span
  $S((\ghat^T)^*)^{\Ghat^T}$ as an $L$-vector space. Thus it suffices
  to show that
  \[\tr (\Lie(r)(\Theta)^k)=\tr (\Lie(r)(\Theta')^k), \quad \forall
    k\geq 1,\] which in turn is equivalent to
  $\tr(\Theta_{{\rm Sen},\, r\circ \rho}^k)=\tr(\Theta_{{\rm Sen}, \,
    r\circ \rho'}^k)$ for all $k\ge 1$.  By assumption the Galois
  representations $r\circ \rho$ and $r\circ \rho'$ restricted to
  $\Gal_E$ have the same semi-simplification. Note that the
  restriction to $\Gal_E$ does not change the Sen operator. It follows
  from \eqref{Sen_compare} that the functor $W\mapsto D_{\Sen}(W)$ is
  exact. This implies that for every embedding
  $\upsilon: L\hookrightarrow \Cp$ the semi-simple parts of
  $\Theta_{\Sen, r\circ \rho, \upsilon}$ and
  $\Theta_{\Sen, r\circ \rho', \upsilon}$ are conjugate in
  $\GL(V\otimes_{L,\upsilon} \Cp)$ and the assertion follows.
\end{proof}

\subsection{Algebraic representations and regular Hodge--Tate weights}\label{sec_HT_reg}
We would like to spell out a part of the proof of Proposition
\ref{inf_HT}, since this is important for the $p$-adic Langlands
correspondence.

We assume the setup at the beginning of section \ref{families_gal} and
let $\rho: \Gal_F \rightarrow \CG_f(L)$ be a continuous admissible
representation, satisfying $d\circ \rho=\chi_{\cyc}$ We assume that
$\rho$ is Hodge--Tate. For each embedding
$\upsilon: L\hookrightarrow \Cp$ we get a $W$-orbit
$\nu^T_{\rho, \upsilon}\in X_*(\That\times \Gm)$. Each
$\lambda^T\in \nu^T_{\rho, \upsilon}$ is of the form
$(\lambda, \id_{\Gm})$ with $\lambda \in X_*(\That)$. We let
$\nu_{\rho, \upsilon}^{C}$ be the image of $\nu^T_{\rho, \upsilon}$ in
$X_*(\That)$. It follows from Remark \ref{orbits} that
$\nu_{\rho, \upsilon}^{C}$ depends only on $\Gal_E$-orbit of
$\upsilon$ and these are canonically in bijection with the set of
embeddings $\tau: E\hookrightarrow L$. Let us reindex the $W$-orbits
by the embeddings $\tau: E\hookrightarrow L$.  We may identify
$X_*(\That)=X^*(T)=X^*(T\times_{E, \tau} L)$, so that
$\nu^C_{\rho, \tau} \subset X^*(T\times_{E, \tau} L)$.

\begin{defi} We say that $\rho$ has regular Hodge--Tate weighs if for
  each $\tau$, $\nu^C_{\rho, \tau}$ contains
  $\lambda_{\tau} \in X^*(T\times_{E, \tau} L)$, which is dominant
  with respect to $B\times_{E, \tau} L$.
\end{defi}
\begin{remar} Since the action of $\Gal_F$ on the root system permutes
  the positive roots, Corollary \ref{beegee_eq} implies that it is
  enough to check this condition for at least one $\tau$ above every
  $\sigma: F\hookrightarrow L$.
\end{remar}

\begin{defi}\label{pi_alg} If $\rho$ has regular Hodge--Tate weights
  then we define an irreducible algebraic representation
  $\pi_{\alg}(\rho)$ of
  $(\Res_{F/\Qp} G)_L\cong \prod_{\sigma} G\times_{F, \sigma} L$ by
  \[ \pi_{\alg}(\rho)\coloneqq \bigotimes_{\sigma: F\hookrightarrow L}
    V_{\sigma},\] where $V_{\sigma}$ is an irreducible algebraic
  representation of $G\times_{F, \sigma} L$ of highest weight
  $\lambda_{\tau}$ with respect to $B\times_{E, \tau} L$, for any
  $\tau: E\hookrightarrow L$ above $\sigma$.
\end{defi}

\begin{remar} The definition of $V_{\sigma}$ is independent of the
  choice of $\tau$: any two lie in the same $\Gal_F$-orbit and
  Corollary \ref{beegee_eq} together with the proof of Lemma
  \ref{transform_kappa} imply the assertion.
\end{remar}

\begin{remar} It follows from the proof of Proposition \ref{inf_HT}
  that under the above assumptions $\zeta^C_{\rho}$, defined in
  Definition \ref{def_ringhom_C}, is the infinitesimal character of
  $\pi_{\alg}(\rho)$.
\end{remar}

We will discuss the notion of regular Hodge--Tate weights in the
presence of the twisting element, see the end of subsection
\ref{sec_Cgr}. Recall that
$\delta_{\mathrm{ad}}\in X_*(\That/Z_{\Ghat})$ is the unique element
such that $2 \delta_{\mathrm{ad}}$ is equal to the image of
$2\delta$. We assume that there is $\tilde{\delta}\in X_*(\That)$ such
that its image in $X_*(\That/Z_{\Ghat})$ is equal to
$\delta_{\mathrm{ad}}$. Let $\rho: \Gal_F \rightarrow \LG(L)$ be an
admissible representation, which we assume to be Hodge--Tate. Let
$\nu_{\rho, \tau}$ be the $W$-orbit corresponding to $\Gal_E$-orbit of
$\upsilon: L\hookrightarrow \Cp$.

\begin{lem}\label{rhoC} Let $\rho^C: \Gal_F\rightarrow \CG_f(L)$ be the representation \[\rho^C\coloneqq \tw_{\tilde{\delta}}^{-1}\circ(\rho \boxtimes \chi_{\cyc}),\] where $\tw_{\tilde{\delta}}$ is defined in \eqref{twisting_el}. Then 
  $\rho^C$ is Hodge--Tate and its Hodge--Tate weights are regular if
  and only if for every $\tau: E\hookrightarrow L$ there is
  $\lambda_{\tau}\in \nu_{\rho, \tau}\subset X^*(\That\times_{E, \tau}
  L)$ such that $\lambda_{\tau}-\tilde{\delta}$ is dominant with
  respect to $B\times_{E, \tau} L$.
\end{lem}
\begin{proof} The assertion follows from \eqref{CL}. 
\end{proof} 

\begin{examp} Let $G$ be an inner form of $\GL_{n, F}$, $B$ the
  upper-triangular matrices over $E$ and $T$ the diagonal matrices
  over $E$. Then $\Ghat=\GL_{n, L}$, $\Bhat$ the upper-triangular
  matrices over $L$ and $T$ the diagonal matrices over $L$.  Let
  $\rho : \Gal_F\rightarrow\GL_n(L)$ be a Hodge--Tate representation.
  Since $G$ is an inner form, the action of $\Gal_F$ on the root
  system is trivial, so all $\nu_{\rho, \tau}$ above depend only on
  $\tau|_F$.  We may identify $X_*(\That)$ with the set of $n$-tuples
  of integers, where $(k_1, \ldots, k_n)$ corresponds to
  $t\mapsto \mathrm{diag}(t^{k_1}, \ldots, t^{k_n})$. The action of
  $W=S_n$ permutes the entries.  With the above identification
  let \[\tilde{\delta}\coloneqq (0, -1, \ldots, 1-n).\] If
  $\lambda \in \nu_{\rho, \sigma}$ corresponds to an $n$-tuple
  $(k_{1, \sigma}, \ldots, k_{n, \sigma})$ then
  $\lambda - \tilde{\delta}$ is dominant with respect to
  $B\times_{E, \tau} L$ if and only if
  \[ k_{1, \sigma} \ge k_{2, \sigma}+1\ge \ldots \ge k_{n,
      \sigma}+n-1,\] which is equivalent to
  $k_{1, \sigma} > k_{2, \sigma}>\ldots> k_{n,\sigma}$. This is the
  usual notion of regular Hodge--Tate weights in the $p$-adic Hodge
  theory.
\end{examp} 
\begin{remar}\label{twisting_GLn} If $G=\GL_{n, F}$  and $\rho: \Gal_F\rightarrow \GL_n(L)$ is a Hodge--Tate representation with regular Hodge--Tate weights 
  in the traditional sense then the algebraic representation
  $\pi_{\alg}(\rho)$ defined in \cite[\S1.8]{6auth} coincides with
  $\pi_{\alg}(\rho^C)$ defined in Definition \ref{pi_alg} with
  $\tilde{\delta}$ as above. We caution the reader that in order to
  verify this one has to bear in mind that in \cite{6auth} the
  Hodge--Tate weight of the cyclotomic character is $-1$ and in this
  paper it is $1$.
\end{remar} 

\subsection{Automorphic representations}\label{sec_auto_reps}

In this section we fix an embedding $\iota :
L\rightarrow\C$. Now let $F$ be a number field, $G$ a 
connected reductive group over $F$ and $\mathfrak{g}$ its Lie
algebra. Let $\pi$ be some automorphic representation of
$G(\mathbb{A}_F)$ which is $C$-algebraic in the sense of
\cite[Def.~3.1.2]{beegee}. Let $\pi_\infty$ be the representation of
$G(F\otimes_{\Q}\R)$ which is the archimedean part of $\pi$. Then the
center
$Z(\Res_{F/\Q} \mathfrak{g})_{\C}$ of the enveloping algebra
$U(\Res_{F/\Q}\mathfrak{g})_{\C}$ acts on the subset of smooth
vectors of $\pi_\infty$ by a character $\chi$. The decomposition
\[ Z(\Res_{F/\Q} \mathfrak{g})_{\C}\simeq\bigotimes_{\kappa :
    F\hookrightarrow \C} Z(\mathfrak{g})\otimes_{F,\kappa}\C \]
gives a decomposition $\chi=\bigotimes_\kappa \chi_\kappa$ of the
character $\chi$.

\begin{prop}\label{C-alg-inf}
  Let $\rho_{\pi,\iota} : \Gal_F\rightarrow{}^CG(L)$ be a
  continuous admissible representation satisfying the assumptions of
  Conjecture $5.3.4$ in \cite{beegee}. Let $v$ be a place of $F$
  dividing $p$ and let
  $\rho_v\coloneqq\rho_{\pi,\iota}|_{\Gal_{F_v}}$. Then we have
  \[ \zeta^C_{\rho_v}\otimes_{L,\iota}\C=\bigotimes_{\kappa :
      F_v\hookrightarrow\overline{\Q}_p}
    \chi_{\iota\circ\kappa}.\]
\end{prop}

\begin{proof}
  Let $\xi : \mathbb{G}_m\rightarrow {}^CG$ be the cocharacter defined
  by $t\mapsto (2\delta(t^{-1}),t^2)$.
  
  If
  $v$ is an archimedean place of $F$, the local Langlands
  correspondence gives us an admissible morphism $\rho_v : W_{F_v}
  \rightarrow {}^LG(\C)$. We recalled in
  section \ref{sec:real} how to associate to an embedding $\kappa$ of $F$
  in $\C$ above $v$, an
  element of $\lambda_\kappa \in X_*(\That)\otimes\C$, well defined up to the
  action of the Weyl group. This gives us $\lambda_\kappa\in
  X_*(\That)\otimes\C$ for each embedding $\kappa$ of $F$ in
  $\C$. Saying that $\pi$ is $C$-algebraic is equivalent to the
  condition $\lambda_\sigma+\tfrac{1}{2}\xi\in X_*(\That\times\mathbb{G}_m)$ for each
  $\kappa : F\hookrightarrow\C$. Note that we have
  $\lambda_\kappa \in X_*(\That)\otimes\Q$.
  
  Now let $v$ be a place of $F$
  dividing $p$. According to Conjecture 5.3.4 and Remark 5.3.5 in
  \cite{beegee}, the representation $\rho$ is Hodge--Tate at $v$ and the
  Hodge--Tate cocharacter of $\rho$ associated to an embedding $\kappa :
  F_v\hookrightarrow L$ is
  $\lambda_{\iota\circ\kappa}+\tfrac{1}{2}\xi$.

  Recall from Remark \ref{rm:localglobalR} that for an embedding
  $\kappa : F \hookrightarrow\C$ the character $\chi_\kappa$ is the composite
  \[ Z(\mathfrak{g})_\C \xrightarrow{\sim}
    S(\that_\C^*)^W\xrightarrow{\lambda_\kappa}\C\]
  where the first map is the Harish-Chandra isomorphism over $\C$
  which can be obtained from the Harish-Chandra isomorphism over $L$
  after base change along $\iota : L \hookrightarrow \C$.

  If $\kappa : F_v\hookrightarrow L$, the character
  $\zeta^C_{\rho,\kappa}$ is the composite
  \[ Z(\mathfrak{g})\otimes_{F,\kappa}L \xrightarrow{\sim}
    S(\that^*)^W\xrightarrow{\lambda_\kappa}L\]
  since $\lambda_\kappa$ is exactly the image of
  $(\lambda_\kappa+\tfrac{1}{2}\xi,1)$ in $\ghat_L$ by the map
  $\alpha$ of section \ref{sec_Cgr} (note that since $\lambda_\kappa
  \in X_*(\That)\otimes\Q$, we land in $L\subset E\otimes L$). It follows that
  $\chi_{\iota\circ\kappa}=\zeta^C_{\rho,\kappa}\otimes_{L,\iota}\C$
  and hence the result.
\end{proof}

\section{Fibres of Cohen--Macaulay modules}\label{FibreCM}
                     
Let $(R, \mm)$ be a complete local noetherian $\OO$-algebra with
residue field $k$, which we assume to be $\OO$-torsion free.  Let
$R^{\rig}$ be the ring of global functions on the rigid space $\Xrig$,
associated to the formal scheme $\Spf R$, see \cite[\S 7.1]{dejong}.

\begin{prop}\label{commutative_algebra} Let $M$ be a faithful, finitely generated $R$-module and let $M^{\rig}\coloneqq  M\otimes_R R^{\rig}$.
  If $R$ is reduced and $M$ is Cohen-Macaulay then the natural map
\[M^{\rig}\to \prod_{x\in \Xrig} M^{\rig}\otimes_{R^{\rig}}
  \kappa(x)\] is injective.
 \end{prop}

 \begin{proof} Since $R$ is $\OO$-torsion free we may choose a finite
   injective map of $\OO$-algebras
   $S\coloneqq \OO[\![x_1,\ldots,x_d]\!]\to R$, where necessarily
   $d+1=\dim R$, see \cite[Thm.\,29.4 (iii)]{matsumura} and the Remark
   following it. The assumptions on $M$ imply that $M$ is a faithful,
   Cohen--Macaulay module over $S$. Since $S$ is regular it follows
   from the Auslander--Buchsbaum theorem,
   \cite[Thm.\,19.1]{matsumura}, that $M$ is free over $S$. The finite
   injective map $S[1/p]\to R[1/p]$ induces a finite surjective
   morphism
   \[\varphi: X\coloneqq \Spec R[1/p]\to Y\coloneqq \Spec S[1/p].\] 
   The fibre of the generic point $\eta$ of $Y$ is geometrically
   reduced. Indeed, it suffices to check that it is reduced (since it
   lives in characteristic zero), which is clear since it is a
   localization of the reduced ring $R[1/p]$. Thus Lemma 37.24.4 in
   \cite[\href{https://stacks.math.columbia.edu/tag/0574}{Tag
     0574}]{stacks-project} yields a nonzero $f\in R[1/p]$ such that
   all fibres over the non-vanishing locus of $f$ are reduced.

   Let $\Sigma$ be the set of closed points $y$ of $Y$ such that
   $f(y)\ne 0$. For any $y\in \Sigma$ and any $x\in \varphi^{-1}(y)$
   we have $x\in \mSpec R[1/p]$, thus $x$ corresponds to a point of
   $\Xrig$ with the same residue field, see
   \cite[Lem.\,7.1.9]{dejong}.  Moreover, for such $y$ we have an
   isomorphism
   $\kappa(y)\otimes_{S} R\cong\prod_{\varphi(x)=y} \kappa(x)$, since
   the fibre over $y$ is finite and reduced.  Using the isomorphism
   $R^{\rig}\cong R\otimes_S S^{\rig}$ provided by
   \cite[Lem.\,7.2.2]{dejong}, we obtain an isomorphism
   \[\prod_{\varphi(x)=y} M^{\rig}\otimes_{R^{\rig}} \kappa(x)\cong
     M^{\rig}\otimes_{S^{\rig}} \kappa(y).\] It suffices therefore to
   prove that the natural map
   $M^{\rig}\to \prod_{y\in \Sigma} M^{\rig}\otimes_{S^{\rig}}
   \kappa(y)$ is injective. Since $M$ is finite free over $S$, this
   reduces to the case $M=S$, i.e. we need to prove that a nonzero
   rigid analytic function $g$ on the open unit disc
   $\mathfrak{Y}^{\rig}$ cannot vanish at all $y\in \Sigma$. In that
   case, $fg$ would vanish at all points of $\mathfrak{Y}^{\rig}$. It
   follows from \cite[Prop.\,5.1.3/3]{BGR} that the restriction of
   $fg$ to every closed disc of radius $1/p^{1/n}$ is zero. Since the
   union of such discs for $n\ge 1$ give an admissible covering of the
   open unit disc, we deduce that $fg=0$. Since $f\neq 0$, $g\neq 0$
   and $S^{\rig}$ is a domain, we get a contradiction.
\end{proof}                     

\section{Density of algebraic vectors}\label{sec_density_algebraic}

Let $X$ be a smooth affine scheme over $\Spec \Qp$ of finite type of
dimension $d$ such that $X(\Qp)$ is Zariski dense in $X$.  Let
$A= \Gamma(X, \OO_X)$ be the ring of global sections on $X$. We may
choose a presentation $A= \Qp[T_1,\ldots, T_n]/I$. This allows us to
consider $X(\Qp)$ as a closed subset of $\Qp^n$ with the induced
topology.  If $U$ is an open subset of $X(\Qp)$ we let
$\mathcal C(U, L)$ be the space of continuous functions from $U$ to
$L$. Since $X(\Qp)= \Hom_{\Qp\text{-}\alg}(A, \Qp)$, evaluation
induces a bilinear map $X(\Qp)\times A \rightarrow \Qp$. This induces
a map $A\otimes_{\Qp} L \rightarrow \mathcal C(X(\Qp), L)$ and we
denote the image by $\mathcal C^{\alg}(X(\Qp), L)$.
\begin{defi} Let $U$ be an open and closed subset of $X(\Qp)$. We let
  $\mathcal C^{\alg}(U, L)$ be the image of
  $\mathcal C^{\alg}(X(\Qp), L)$ in $\mathcal C(U, L)$ under the
  restriction map $f\mapsto f|_U$.
\end{defi}

We will equip $X(\Qp)$ with the structure of a locally analytic
manifold. Let $f_1,\dots,f_d$ be a sequence of generators of $I$. They
induce a map $\varphi : \Qp^n\rightarrow\Qp^d$ which is polynomial and
in particular map of locally analytic manifold. The smoothness of $X$
and the Jacobian criterion imply that $\varphi$ is a subimmersion in
the sense of \cite[Part.~II, chap.~III.4)]{serre_Lie} at all points of
$X(\Qp)$. It follows from the Theorem in III.11C) of \emph{loc.~cit.}
that $X(\Qp)=\varphi^{-1}(0)$ is a submanifold of $\Qp^n$ and in
particular carries a canonical structure of locally analytic
manifold. Moreover it follows from III.11.A) in \emph{loc.~cit.} that,
for each point $x\in X(\Qp)$, there exists an open neighborhood $U_x$
of $x$ in $\Qp^n$ and a locally analytic isomorphism
$\alpha_x=(\alpha_{x,1},\dots,\alpha_{x,n})$ from $U_x$ onto an open
subset $V_x$ of $\Qp^r$ such that
\begin{equation}\label{localchar} U_x\cap X(\Qp)=\{y\in U_x \mid
  \alpha_{x,r+1}(y)=\cdots=\alpha_{x,n}(y)=0\}.\end{equation}
It follows that the inverse $\beta_x$ of the map $(\alpha_{x,1},\dots,\alpha_{x,r})$ induces a locally analytic map from $V_x$ to $U_x$ whose image is $U_x\cap X(\Qp)$.

The main result of the section, Theorem \ref{density1}, says that
$\mathcal C^{\alg}(U, L)$ is a dense subspace of
$\mathcal C^{\la}(U, L)$ locally $\Qp$-analytic functions on $U$, when
$U$ is an open and closed subspace of $X(\Qp)$.  This holds if for
example $U$ is open and compact. We then look closer at the example,
when $X$ is a connected reductive group scheme and $U$ is a compact
open subgroup of $X(\Qp)$.

\begin{lem}\label{clopen}
  Let $M$ be a paracompact manifold, $U\subset M$ an open and closed
  subspace, and let $S$ be a dense subset of
  $\mathcal{C}^{\la}(M, L)$. Then the image of $S$ is dense in
  $\mathcal{C}^{\la}(U, L)$.
\end{lem}
\begin{proof} Since $U$ is open, it naturally carries a structure of a
  paracompact manifold. Since $M\setminus U$ is also open, we have a
  homeomorphism:
  \[ \mathcal{C}^{\la}(M, L)\cong \mathcal{C}^{\la}(U, L) \times
    \mathcal{C}^{\la}(M\setminus U, L)\] for the product topology on
  the right-hand side, \cite[Prop.\,12.5]{schneider_pLie}.  Hence, the
  restriction map
  $r: \mathcal{C}^{\la}(M, L)\rightarrow \mathcal{C}^{\la}(U, L)$,
  $f\mapsto f|_U$ is continuous and surjective.  Hence, if
  $V \subset \mathcal{C}^{\la}(U, L)$ is non-empty and open, then
  $r^{-1}(V) \subset \mathcal{C}^{\la}(M, L)$ is non-empty and
  open. The intersection $S\cap r^{-1}(V)$ is non-empty, since $S$ is
  dense in $\mathcal{C}^{\la}(M, L)$. Hence, $r(S)\cap V$ is also
  non-empty and so $r(S)$ is dense in $\mathcal{C}^{\la}(U, L)$.
\end{proof}

\begin{lem}\label{denseAn} $\mathcal C^{\alg}(\mathbb A^n(\Qp), L)$ is dense in 
  $\mathcal C^{\la}(\mathbb A^n(\Qp), L)$.
\end{lem}
\begin{proof} Let $\mathbb A^n(\Qp)=\cup_{i\in I} U_i$ be a disjoint
  covering by open subsets, where each $U_i$ is a closed ball of
  radius $\epsilon$ around some point in $\mathbb A^n(\Qp)$. By
  \cite[Prop.\,12.5]{schneider_pLie} we have a homeomorphism of
  topological vector spaces:
  \[\mathcal C^{\la}(\mathbb A^n(\Qp), L)\cong \prod_{i\in I} \mathcal
    C^{\la}(U_i, L)\] with the product topology on the right-hand
  side. Thus it is enough to show that the image of
  $\mathcal C^{\alg}(\mathbb A^n(\Qp), L)$ in
  $\prod_{i\in J} \mathcal C^{\la}(U_i, L)$ is dense for every finite
  subset $J\subset I$. Given such $J$ there will exist an $m\ge 0$
  such that $\cup_{i\in J} U_i$ is contained in $p^{-m} \Zp^n$.

  It follows from the theorem of Amice, see \cite[(III.1.3.8)]{Lazard}
  that the image of $\mathcal C^{\alg}(\mathbb A^n(\Qp), L)$ is dense
  in the Banach space of locally analytic functions on $p^{-m} \Zp^n$
  of radius of convergence $p^{-h}$, for any $h\ge 0$. These Banach
  spaces are denoted by $\mathcal F_{\mathcal I}(E)$ in
  \cite{schneider_pLie}, with $E=L$ in our situation. The topology on
  $\mathcal C^{\la}(p^{-m}\Zp^n, L)$ is the locally convex inductive
  limit topology defined by the family of Banach spaces
  $\{\mathcal F_{\mathcal I}(L)\}_{\mathcal I}$, see \cite[\S
  II.12]{schneider_pLie}. If
  $\ell: \mathcal C^{\la}(p^{-m}\Zp^n, L)\rightarrow L$ is a
  continuous linear form, then its restriction to
  $\mathcal F_{\mathcal I}(L)$ is continuous and hence if $\ell$
  vanishes on $\mathcal C^{\alg}(\mathbb A^n(\Qp), L)$ it vanishes on
  $\mathcal F_{\mathcal I}(L)$ for all $\mathcal I$, and thus
  $\ell=0$.  Hahn--Banach implies that the image of
  $\mathcal C^{\alg}(\mathbb A^n(\Qp), L)$ is dense in
  $\mathcal C^{\la}(p^{-m}\Zp^n, L)$.

  Since $\cup_{i\in J} U_i$ is both open and compact, it will be
  closed in $p^{-m}\Zp^n$. Lemma \ref{clopen} implies that the image
  of $\mathcal C^{\alg}(\mathbb A^n(\Qp), L)$ in
  $\mathcal C^{\la}(\cup_{i\in J} U_i, L)=\prod_{i\in J} \mathcal
  C^{\la}(U_i, L)$ is dense.
\end{proof} 

\begin{thm}\label{density1} $\mathcal C^{\alg}(X(\Qp), L)$ is dense in 
  $\mathcal C^{\la}(X(\Qp), L)$.
\end{thm}
\begin{proof} We consider a closed embedding
  $X\hookrightarrow \mathbb A^n$ as before.  Since the inclusion
  $X(\Qp) \hookrightarrow \mathbb A^n(\Qp)$ is a map of locally
  analytic manifolds, it induces a continuous map
  \begin{equation}\label{map1}
    \mathcal C^{\la}( \mathbb A^n(\Qp), L) \rightarrow \mathcal C^{\la}(X(\Qp), L),
  \end{equation}
  by \cite[Prop. 12.4 (ii)]{schneider_pLie}. Since
  $\mathcal C^{\alg}(X(\Qp), L)$ is equal to the image
  $\mathcal C^{\alg}(\mathbb A^n(\Qp), L)$ via \eqref{map1} and
  $\mathcal C^{\alg}(\mathbb A^n(\Qp), L)$ is dense in
  $\mathcal C^{\la}( \mathbb A^n(\Qp), L)$ by Lemma \ref{denseAn}, it
  is enough to show that \eqref{map1} is surjective, see the proof of
  Lemma \ref{clopen}.

  To show the surjectivity of \eqref{map1} let
  $\mathbb A^n(\Qp)= \bigcup_{i\in I} U_i$ be a covering by pairwise
  disjoint open subsets such that each $U_i$ is of the form $U_x$ as
  in \eqref{localchar}. It is enough to prove that the induced map
  \begin{equation}\label{map2}
    \mathcal C^{\la}( U_x, L)\rightarrow \mathcal C^{\la}( X(\Qp)\cap U_x, L)
  \end{equation}
  is surjective. After composition with
  $\alpha_x : U_x\xrightarrow{\sim} V_x$ it is induced by a map of the
  form $(x_1,\dots,x_r)\mapsto (x_1,\dots,x_r,0,\dots,0)$ which has a
  locally analytic section. Hence \eqref{map2} is surjective.
\end{proof} 

\begin{cor}\label{dense_clopen} If $U$ is an open and closed subset of $X(\Qp)$ then 
  $\mathcal C^{\alg}(U, L)$ is dense in $\mathcal C^{\la}(U, L)$.
\end{cor}
\begin{proof} This follows from Theorem \ref{density1} and Lemma
  \ref{clopen}.
\end{proof}

Let $G$ be a connected reductive group scheme defined over $\Qp$. It
follows from \cite[Cor.\,18.3]{Borel} that $G(\Qp)$ is dense in $G$.
Let $\Irr_G(L)$ be the set of isomorphism classes of irreducible
algebraic representations of $G_L$.  We assume that $G$ splits over
$L$, which implies that all representations in $\Irr_G(L)$ are
absolutely irreducible.  Let $K$ be a compact open subgroup of
$G(\Qp)$. By evaluating $V$ at $L$ we get an an action on $V$ by
$G(\Qp)$, and hence by $K$, via the embedding
$G(\Qp)\hookrightarrow G(L)$.

\begin{cor}\label{A3} The evaluation map
  \begin{equation}\label{map3}
    \bigoplus_{[V]\in \Irr_G(L)} \Hom_{K}(V, \mathcal C^{\la}(K, L))\otimes V \rightarrow \mathcal C^{\la}(K, L),
  \end{equation}
  is injective and the image is equal to $\mathcal C^{\alg}(K, L)$. In
  particular, the image of \eqref{map3} is a dense subspace of
  $\mathcal C^{\la}(K, L)$.
\end{cor}
\begin{proof} The first part follows from \cite[Prop.\,A.3]{durham},
  where we have shown an analogous statement with
  $\mathcal C^{\la}(K, L)$ replaced with $\mathcal C(K, L)$. The
  density result follows from Corollary \ref{dense_clopen}.
\end{proof}

Let $\tau$ be a continuous representation of $K$ on a finite
dimensional $L$-vector space.  If $V$ is an irreducible representation
of $G_L$ we will write $V(\tau)\coloneqq V\otimes \tau$ with the
diagonal $K$-action.
\begin{cor}\label{summand} Let $\Pi$ be an admissible unitary $L$-Banach space representation of $K$, which 
  is a direct summand of $\mathcal C(K, L)^{\oplus m}$ for some
  $m\ge 1$. Let $\Pi^{\la}$ be the subspace of locally analytic
  vectors in $\Pi$. Then the evaluation map
  \begin{equation}\label{map4}
    \bigoplus_{[V]\in \Irr_G(L)} \Hom_{K}(V(\tau), \Pi^{\la})\otimes V(\tau) \rightarrow \Pi^{\la},
  \end{equation}
  is injective and its image is a dense subspace of $\Pi^{\la}$.
\end{cor}
\begin{proof} We note that the subspace of locally analytic vectors in
  $\mathcal C(K, L)$ is equal to $\mathcal C^{\la}(K, L)$ by
  \cite[Prop.\,3.3.4]{em1}. Thus $\Pi^{\la}$ is a direct summand of
  $\mathcal C^{\la}(K, L)^{\oplus m}$. If $\tau$ is the trivial
  representation then the assertion can be easily deduced from
  Corollary \ref{A3}. We will reduce the assertion to this case.

  Let $W$ be the closure of the image of \eqref{map4} in
  $\Pi^{\la}$. Then $W$ can be characterised as the smallest closed
  subspace of $\Pi^{\la}$ such that
  \[\Hom_{K}(V(\tau), W)= \Hom_{K}(V(\tau), \Pi^{\la}), \quad \forall
    V\in \Irr_{G}(L).\] We deduce that
  \[\Hom_{K}(V, W\otimes \tau^*)= \Hom_{K}( V, \Pi^{\la}\otimes
    \tau^*), \quad \forall V\in \Irr_{G}(L).\] Thus $W\otimes \tau^*$
  contains the closure of the image of the evaluation map
  \begin{equation}\label{434}
    \bigoplus_{[V]\in \Irr_{G}(L)} \Hom_{K}(V, \Pi^{\la}\otimes \tau^*)\otimes  V \rightarrow \Pi^{\la}\otimes \tau^*.
  \end{equation}
  The projection formula gives us an isomorphism
  $\mathcal C(K, L)\otimes \tau^*\cong \mathcal C(K, L)^{\oplus \dim
    \tau}$. Thus $\Pi\otimes \tau^*$ is a direct summand of a direct
  sum of finitely many copies of $\mathcal C(K, L)$.  Since $\tau$ is
  finite dimensional we have
  $(\Pi\otimes \tau^*)^{\la}= \Pi^{\la} \otimes \tau^*$. From the
  special case considered above we deduce that
  $W\otimes \tau^*= \Pi^{\la} \otimes \tau^*$ and hence $W=\Pi^{\la}$.
\end{proof}

We want to have a variant of Corollary \ref{summand} with fixed
central character. Let $Z$ be the centre of $G$ and let
$\psi: K\cap Z(\Qp)\rightarrow \OO^{\times}$ be a continuous group
homomorphism. We let $\mathcal C_{\psi}(K, L)$ be the closed subspace
of $\mathcal C(K,L)$ consisting of functions on which $K\cap Z(\Qp)$
acts by $\psi$. We assume that $\tau$ has central character $\psi$.

\begin{cor}\label{summand_central} Let $\Pi$ be an admissible unitary $L$-Banach space representation of $K$, which 
  is a direct summand of $\mathcal C_{\psi}(K, L)^{\oplus m}$ for some
  $m\ge 1$. Let $\Pi^{\la}$ be the subspace of locally analytic
  vectors in $\Pi$. Then the image of the evaluation map
  \begin{equation}\label{map423}
    \bigoplus_{[V]\in \Irr_{G/Z}(L)} \Hom_{K}(V(\tau), \Pi^{\la})\otimes V(\tau) \rightarrow \Pi^{\la},
  \end{equation}
  is a dense subspace of $\Pi^{\la}$.
\end{cor}
\begin{proof} Since the central character of $\tau^*$ is equal to
  $\psi^{-1}$, projection formula gives us an isomorphism
  $\mathcal C_{\psi}(K, L)\otimes \tau^*\cong \mathcal C(K/K\cap
  Z(\Qp), L)^{\oplus \dim \tau}$.  Thus arguing as in the proof of
  previous Corollary we may reduce the assertion to the case when
  $\tau$ and $\psi$ are both trivial. Since the image of $K$ in
  $(G/Z)(\Qp)$ is a compact open subgroup, the assertion follows from
  Corollary \ref{A3} applied to $G/Z$.
\end{proof}

\section{Families of Banach space representations}\label{sec_fam_Ban}

Let $G$ be a connected reductive group over $\Qp$, split over $L$ and
let $K$ be a compact open subgroup of $G(\Qp)$. An example relating
this set up to section \ref{families_gal}, would be
$G=\Res_{F/\Qp} G'$, where $G'$ is a connected reductive group over
$F$, and $K$ compact open subgroup of $G'(F)$.

Let $(R, \mm)$ be a complete local noetherian $\OO$-algebra with
residue field $k$. Let $M$ be a finitely generated
$R[\![K]\!]$-module. If $V$ is an irreducible algebraic representation
of $G_L$ then by evaluating at $L$ we get an an action on $V$ by
$G(\Qp)$, and hence by $K$, via the embedding
$G(\Qp)\hookrightarrow G(L)$. Let $\tau$ be a representation of the
form $\sigma \otimes \eta$, where $\sigma$ is a smooth absolutely
irreducible representation of $K$ and
$\eta: K\rightarrow \OO^{\times}$ is a continuous character. Then
$V(\tau)\coloneqq V \otimes \tau$ is an absolutely irreducible
representation of $K$. Since $K$ is compact there is a $K$-invariant
$\OO$-lattice $\Theta$ in $V(\tau)$. We let
\[M(\Theta)\coloneqq \Hom_{\OO[\![K]\!]}^{\cont}(M, \Theta^d)^d\cong
  \Theta \otimes_{\OO[\![K]\!]} M,\] where
$(\cdot)^d\coloneqq \Hom_{\OO}^{\cont}(\cdot, \OO)$ and in the tensor
product $\Theta$ is considered as a left $\OO[\![K]\!]$-module via
$g\mapsto g^{-1}$.  Since $M$ is a finitely generated
$R[\![K]\!]$-module, $M(\Theta)$ is a finitely generated $A$-module,
see \cite[Prop.\,2.15]{duke}. Moreover, $M(\Theta)[1/p]$ is
independent of the choice of lattice $\Theta$. Let $R_{V(\tau)}$ be
the quotient of $R$, which acts faithfully on $M(\Theta)$ and let
\[\Sigma_{V(\tau)}\coloneqq \mSpec R_{V(\tau)}[1/p].\]
We note that $R_{V(\tau)}$ and $\Sigma_{V(\tau)}$ also depend on $M$. 
Since $M$ is finitely generated over $R[\![K]\!]$, it is compact. Thus 
\[\Pi\coloneqq \Hom_{\OO}^{\cont}(M, L)\] 
equipped with the supremum norm is a unitary $L$-Banach space
representation of $K$. If $M$ is not finitely generated over
$\OO[\![K]\!]$ then $\Pi$ is not admissible. However, one can get
around this by a trick introduced in \cite[\S 2.1]{duke}.  We choose a
presentation $\OO[\![x_1, \ldots, x_s]\!]\twoheadrightarrow R$.  This
induces a surjection
$\OO[\![\Zp^s\times K]\!]\twoheadrightarrow R[\![K]\!]$, and thus
$\Pi$ is an admissible unitary $\Zp^s\times K$-Banach space
representation. Following \cite[Def.\,3.2]{BHS} we define $\Pi^{\Rla}$
as the subspace of of locally analytic vectors for
$\Zp^s \times K$-action on $\Pi$. It will become apparent in the proof
of the Theorem below that it is actually better to think of
$\OO[\![x_1, \ldots, x_s]\!]$ as the completed group algebra of
$((1+2p \Zp)^s, \times)$ instead of $(\Zp^s, +)$.

If $x\in \mSpec R[1/p]$ then we denote by $\mm_x$ the corresponding
maximal ideal and by $\kappa(x)$ its residue field.  Let $\Pi[\mm_x]$
be the subspace of $\Pi$ consisting of vectors killed by
$\mm_x$. Since $R$ acts on $\Pi$ by continuous endomorphisms,
$\Pi[\mm_x]$ is a closed subspace of $\Pi$. It follows from
\cite[Lem.\,2.20, 2.21]{duke} applied with $\md= \kappa(x)$ that
\[\Pi[\mm_x]\cong \Hom_{\OO}^{\cont}(M\otimes_R\OO_{\kappa(x)}, L)\]
and $\Pi[\mm_x]$ is an admissible unitary Banach space representation of $K$.

We will now explain how to put a topology on
$\Hom_K(V(\tau), \Pi^{\Rla})\otimes V(\tau)$. The module
$M(\Theta)\otimes \Theta^d$ is finitely generated over
$R[\![K]\!]$. Hence,
$\Hom_{\OO}^{\cont}(M(\Theta)\otimes \Theta^d, L)$ is an admissible
unitary Banach space representation of $\Zp^s\times K$. It follows
from \cite[Eq.\,(10), (11)]{duke} that we have an isomorphisms of
Banach space representations:
\begin{equation}\label{banach}
\begin{split}
  \Hom_{\OO}^{\cont}(M(\Theta)\otimes \Theta^d, L)&\cong
  \Hom_{\OO[\![K]\!]}^{\cont}(M, V(\tau)^*)\otimes V(\tau)\\&\cong
  \Hom_{K}(V(\tau), \Pi)\otimes V(\tau).
\end{split}
\end{equation}
Since both $\Hom_{K}(V(\tau), \Pi)\otimes V(\tau)$ and $\Pi$ are
admissible unitary Banach space representations of $\Zp^s\times K$,
the evaluation map
\begin{equation}\label{evaluation}
  \Hom_{K}(V(\tau), \Pi)\otimes V(\tau) \rightarrow \Pi
\end{equation}
is continuous and the image is closed in $\Pi$.

\begin{lem} The evaluation map \eqref{evaluation} is injective.
\end{lem}
\begin{proof} Since as $K$-representation
  $\Hom_{K}(V(\tau), \Pi)\otimes V(\tau)$ is isomorphic to a direct
  sum of copies of $V(\tau)$, which is irreducible, the kernel
  $\mathcal K$ is also isomorphic to a direct sum of copies of
  $V(\tau)$.  Since $V(\tau)$ is absolutely irreducible and finite
  dimensional we have $\Hom_K(V(\tau),V(\tau))=L$. By applying the
  functor $\Hom_K(V(\tau), \cdot)$ to \eqref{evaluation} we obtain an
  isomorphism. Hence, $\Hom_K(V(\tau), \mathcal K)=0$ and so
  $\mathcal K=0$.
\end{proof}

By applying the functor of $R$-locally analytic vectors to
\eqref{evaluation} we obtain injection
\begin{equation}\label{swapAan}
  \Hom_{K}(V(\tau), \Pi)^{\Rla}\otimes V(\tau)=(\Hom_K(V(\tau), \Pi)\otimes V(\tau))^{\Rla} \hookrightarrow \Pi^{\Rla}
\end{equation}
where $\Hom_{K}(V(\tau), \Pi)^{\Rla}$ is the subspace of locally
analytic vectors for $\Zp^s$-action on $\Hom_{K}(V(\tau), \Pi)$, thus
is naturally endowed with a topology, see \cite{ST}.

\begin{lem}\label{swap} $\Hom_{K}(V(\tau), \Pi)^{\Rla}\otimes V(\tau)= \Hom_K(V(\tau), \Pi^{\Rla})\otimes V(\tau)$,
\end{lem}
\begin{proof} We identify $\Hom_K(V(\tau), \Pi)\otimes V(\tau)$ with
  its image in $\Pi$. Then its $R$-locally analytic vectors are equal
  to $(\Hom_K(V(\tau), \Pi)\otimes V(\tau))\cap \Pi^{\Rla}$. The
  inclusions
  \[\Hom_K(V(\tau), \Pi^{\Rla})\otimes V(\tau) \subset
    (\Hom_K(V(\tau), \Pi)\otimes V(\tau))\cap \Pi^{\Rla} \subset
    \Pi^{\Rla}\] become isomorphisms after applying
  $\Hom_K(V(\tau), \cdot)$. Since as representations of $K$ both
  $\Hom_K(V(\tau), \Pi^{\Rla})\otimes V(\tau)$ and
  $(\Hom_K(V(\tau), \Pi)\otimes V(\tau))\cap \Pi^{\Rla}$ are
  isomorphic to a direct sum of copies of $V(\tau)$, we conclude that
  they are equal.
\end{proof}

The Lemma above and \eqref{swapAan} allows us to put a topology on
$\Hom_{K}(V(\tau), \Pi^{\Rla})\otimes V(\tau)$. It follows from the
theory of admissible locally analytic representations, see
\cite[Prop.\,6.4]{ST}, that the image of
$\Hom_{K}(V(\tau), \Pi^{\Rla})\otimes V(\tau)$ under
\eqref{evaluation} is closed in $\Pi^{\Rla}$ and \eqref{evaluation}
induces a homeomorphism between
$\Hom_{K}(V(\tau), \Pi^{\Rla})\otimes V(\tau)$ and its image in
$\Pi^{\Rla}$.

\begin{prop}\label{long_march} If $R_{V(\tau)}$ is reduced and
  $M(\Theta)$ is Cohen-Macaulay then
  \[\bigoplus_{x\in \Sigma_{V(\tau)}}\Hom_{K}(V(\tau), \Pi[
    \mm_x]^{\la})\otimes V(\tau)\] is a dense subspace of
  $\Hom_{K}(V(\tau), \Pi^{\Rla})\otimes V(\tau)$.
\end{prop} 
\begin{proof} Let us fix $x\in \Sigma_{V(\tau)}$. Since $\Pi[\mm_x]$
  is admissible as $K$-representation \cite[Prop.\,3.8]{BHS} implies
  that $\Pi[\mm_x]^{\la}=\Pi[\mm_x]^{\Rla}$. We thus have an inclusion
  $\Pi[\mm_x]^{\la} \subset \Pi^{\Rla}$.  If
  $x, x_1, \ldots, x_n \in \Sigma_{V(\tau)}$ are pairwise distinct
  then $(\oplus_{i=1}^n \Pi[\mm_{x_i}] ) \cap \Pi[\mm_x]=0$, as
  $1 \in \mm_x + \cap_{i=1}^n \mm_{x_i}$. We thus have an inclusion
  $\oplus_{x\in \Sigma_{V(\tau)}} \Pi[\mm_x]^{\la} \subset
  \Pi^{\Rla}$. By applying $\Hom_K(V(\tau), \cdot)$ we deduce that
\begin{equation}\label{prismatic}
  \bigoplus_{x\in \Sigma_{V(\tau)}}\Hom_{K}(V(\tau), \Pi[ \mm_x]^{\la}) \subset \Hom_K(V(\tau), \Pi^{\Rla}).
\end{equation}
By Hahn--Banach theorem, see \cite[Cor.\,9.3]{schneider_nfa}, it is
enough to show that any continuous linear form on
$\Hom_K(V(\tau), \Pi^{\Rla})$, which vanishes on the left hand-side of
\eqref{prismatic}, is zero. Below we will denote by $W'$ the
continuous linear dual of a locally convex $L$-vector space $W$.

Going back to \eqref{banach} we think of $\Hom_K(V(\tau), \Pi^{\Rla})$
as locally analytic vectors for the $\Zp^s$-action on
$\Hom_{\OO}^{\cont}( M(\Theta), L)$. Thus
\[ \Hom_K(V(\tau), \Pi^{\Rla})'\cong
  M(\Theta)\otimes_{\OO[\![\Zp^s]\!]} D(\Zp^s)\cong
  M(\Theta)\otimes_{R_{V(\tau)}} R_{V(\tau)}^{\rig}=:
  M(\Theta)^{\rig}.\] We may think of
$\Hom_{K}(V(\tau), \Pi[ \mm_x]^{\la})$ as
$\Hom_K(V(\tau), \Pi^{\Rla})[\mm_x]$. Thus
\[\Hom_{K}(V(\tau), \Pi[ \mm_x]^{\la})'\cong M(\Theta)^{\rig}\otimes_{R_{V(\tau)}^{\rig} }\kappa(x).\]
The assertion follows from Proposition \ref{commutative_algebra}.
\end{proof}

\begin{thm}\label{patched_density} Assume that there is an $M$-regular sequence $y_1, \ldots, y_h$ contained in $\mm$, such that 
  $M/(y_1, \ldots, y_h)M$ is a finitely generated projective
  $\OO[\![K]\!]$-module.  Then the evaluation map
  \begin{equation}\label{map412}
    \bigoplus_{[V]\in \Irr_G(L)} \Hom_{K}(V(\tau), \Pi^{\Rla})\otimes V(\tau) \rightarrow \Pi^{\Rla},
  \end{equation}
  is injective and its image is a dense subspace of
  $\Pi^{\Rla}$. Moreover, the map
  \begin{equation}\label{map5}
    \bigoplus_{[V]\in \Irr_G(L)} \bigoplus_{x\in \Sigma_{V(\tau)}}\Hom_{K}(V(\tau), \Pi[ \mm_x]^{\la})\otimes V(\tau) \rightarrow \Pi^{\Rla},
  \end{equation}
  is injective and if additionally $R_{V(\tau)}$ is reduced for all
  $V\in \Irr_G(L)$ then the image of \eqref{map5} is a dense subspace
  of $\Pi^{\Rla}$.
\end{thm}
\begin{proof}
  Since $\Hom_{K}(V(\tau), \Pi^{\Rla})\otimes V(\tau)$ is
  $V(\tau)$-isotypic to prove the injectivity of \eqref{map412} it is
  enough to do so for a single summand. That follows from Lemma
  \ref{swap} and injectivity of \eqref{swapAan}. Since the left
  hand-side of \eqref{map5} is a subspace of the left hand-side of
  \eqref{map412}, see Proposition \ref{long_march}, we deduce that
  \eqref{map5} is injective. We will now show that image of
  \eqref{map412} is dense by exhibiting a dense subspace.

  Let $S=\OO[\![x_1, \ldots, x_h]\!]$. By mapping $x_i$ to $y_i$ we
  obtain a ring homomorphism $S\rightarrow R$, which makes $M$ into a
  finitely generated $S[\![K]\!]$-module. We claim that $M$ is a
  projective as $S[\![K]\!]$-module. To prove the claim it is enough
  to show that $M$ is a free $S[\![P]\!]$-module, where $P$ is a
  pro-$p$ Sylow of $K$. By topological Nakayama's lemma we may choose
  a surjection $S[\![P]\!]^{\oplus m} \twoheadrightarrow M$, which
  becomes an isomorphism after applying the functor
  $k\otimes_{S[\![P]\!]}$, and let $\mathcal K$ denote the kernel. We
  argue by induction on $h$ that such map has to be an isomorphism. If
  $h=0$ then $S=\OO$ and $M$ is projective as $S[\![K]\!]$-module, and
  so the surjection has a section.  Hence,
  $k\otimes_{S[\![P]\!]} \mathcal K=0$ and topological Nakayma's lemma
  implies that $\mathcal K=0$. Let $h$ be arbitrary. Since $x_h$ is
  $M$-regular, we have an exact sequence of $S/(x_h)[\![P]\!]$-modules
  \[0 \rightarrow \mathcal K/x \mathcal K \rightarrow
    S/(x_h)[\![P]\!]^{\oplus m}\rightarrow M/ x_h M\rightarrow 0.\]
  Since $S/(x_h)\cong \OO[\![x_1, \ldots, x_{h-1}]\!]$ and the
  sequence $x_1, \ldots, x_{h-1}$ in $M/ x_h M$-regular, we deduce
  that $M/ x_h M$ is a free $S/(x_h)[\![P]\!]$-module and so the
  surjection has a splitting. By the same argument as above we deduce
  that $\mathcal K/x_h \mathcal K =0$. Topological Nakayama's lemma
  implies that $\mathcal K=0$.  This finishes the proof of the claim.

  The claim implies that $M$ is a direct summand of
  $S[\![K]\!]^{\oplus m}$ for some $m\ge 1$. By dualizing and
  identifying $S$ with the completed group algebra of $(1+2p \Zp)^h$,
  we may consider $\Pi$ as an admissible unitary Banach space
  representation of $K'\coloneqq (1+2p \Zp)^h \times K$, which is a
  direct summand $\mathcal C(K', L)^{\oplus m}$. The locally analytic
  vectors for the action of $K'$ on $\Pi$ are equal to $\Pi^{\Sla}$
  and since $M$ is finitely generated over $S[\![K]\!]$,
  \cite[Prop.\,3.8]{BHS} implies that $\Pi^{\Rla}= \Pi^{\Sla}$. Since
  $K'$ is a compact open subgroup of $G'\coloneqq \Gm^h \times G$,
  Corollary \ref{summand} implies that the evaluation map
  \begin{equation}\label{map6}
    \bigoplus_{[V']\in \Irr_{G'}(L)} \Hom_{K'}(V'(\tau), \Pi^{\Rla})\otimes V'(\tau) \rightarrow \Pi^{\Rla},
  \end{equation}
  is injective and its image is a dense subspace of
  $\Pi^{\Rla}$. Every $V'$ is of the form $\chi\boxtimes V$, where
  $V\in \Irr_G(L)$ and $\chi: \Gm^h \rightarrow \Gm$ is an algebraic
  representation.  Such $\chi$ induces a continuous group homomorphism
  $\chi: (1+2p \Zp)^s\rightarrow L^{\times}$ and hence a maximal ideal
  of $S[1/p]$, which we denote by $\mm'_{\chi}$.  We thus may rewrite
  \eqref{map6} as a $K$-equivariant map
  \begin{equation}\label{map7}
    \bigoplus_{[V]\in \Irr_{G}(L)} \bigoplus_{\chi\in \Irr_{ \Gm^h}(L)} \Hom_{K}(V(\tau), \Pi[\mm_{\chi}']^{\la})\otimes V(\tau)\rightarrow \Pi^{\Rla}.
  \end{equation}
  Since
  $ \Pi[\mm_{\chi}']^{\la}= \Pi[\mm_{\chi}']^{\Rla} \subset
  \Pi^{\Rla}$, the image of \eqref{map7} will be contained in the
  image of \eqref{map412}. Since the image of \eqref{map7} is dense,
  we conclude that the image of \eqref{map412} is dense.

  The proof of \cite[Lem.\,4.18 1)]{6auth} shows that $M(\Theta)$ is a
  free $S$-module of finite rank. Thus $y_1, \ldots, y_h$ is a regular
  sequence of parameters for $M(\Theta)$ and hence $M(\Theta)$ is
  Cohen--Macaulay as $R_{V(\tau)}$-module.  If $R_{V(\tau)}$ is
  reduced then Proposition \ref{long_march} implies that the closure
  of the image of
  \[\bigoplus_{x\in \Sigma_{V(\tau)}}\Hom_{K}(V(\tau), \Pi[
    \mm_x]^{\la})\otimes V(\tau)\rightarrow \Pi^{\Rla}\] is equal to
  $\Hom_K(V(\tau), \Pi^{\Rla})\otimes V(\tau)$. Hence the closure of
  the image of \eqref{map5} will contain the image of \eqref{map412}
  and hence is equal to $\Pi^{\Rla}$.
\end{proof} 

Let $Z(\mathfrak g)_L$ be the centre of the enveloping algebra of the
Lie algebra of $G_L$ and let
$\zeta: Z(\mathfrak g)_L \rightarrow R^{\rig}$ be a homomorphism of
$L$-algebras, where $R^{\rig}$ is the ring of global functions on the
rigid space $\Xrig$, associated to the formal scheme $\Spf R$, see
\cite[\S 7.1]{dejong}. The map $R[1/p]\rightarrow R^{\rig}$ induces a
bijection between $\mSpec R[1/p]$ and points of $\Xrig$ and
isomorphism on the completions of local rings by
\cite[Lem.\,7.1.9]{dejong}.  Thus we may specialise $\zeta$ at any
$x\in \mSpec R[1/p]$ to obtain a ring homomorphism:
\[\zeta_x: Z(\mathfrak g)_L \overset{\zeta}{\longrightarrow}
  R^{\rig}\rightarrow \kappa(x).\]

\begin{thm}\label{infinitesimal} Assume that the following hold:
  \begin{itemize}
  \item[(i)] there is an $M$-regular sequence $y_1, \ldots, y_h$ in
    $\mm$, such that $M/(y_1, \ldots, y_h)M$ is a finitely generated
    projective $\OO[\![K]\!]$-module;
  \item[(ii)] the rings $R_{V(\tau)}$ are reduced,
    $\forall V\in \Irr_G(L)$;
  \item[(iii)] the action of $Z(\mathfrak g)_L$ on $V(\tau)$ is given
    by $\zeta_x$, $\forall V\in \Irr_G(L)$,
    $\forall x\in \Sigma_{V(\tau)}$.
 \end{itemize}
 Then $Z(\mathfrak g)_L$ acts on $\Pi^{\Rla}$ via $\zeta$. In
 particular, for all $y\in \mSpec R[1/p]$, $Z(\mathfrak g)_L$ acts on
 $\Pi[\mm_y]^{\la}$ via the infinitesimal character $\zeta_y$.
\end{thm}
\begin{proof} It follows from \cite[Lem.\,3.3]{BHS} that we have a
  continuous $R^{\rig}\wtimes D(K,L)$ action on $\Pi^{\Rla}$ by
  continuous endomorphisms. This induces an action of
  $R^{\rig} \otimes Z(\mathfrak g)_L$, see \cite[Prop.\,3.7]{ST2}. Let
  $\mathfrak a$ be the ideal of $R^{\rig}\otimes Z(\mathfrak g)_L$
  generated by the elements of the form
  $\zeta(z) \otimes 1 - 1 \otimes z$ for $z\in Z(\mathfrak
  g)_L$. Assumption (iii) implies that such elements will
  kill
  \[\bigoplus_{[V]\in \Irr_G(L)} \bigoplus_{x\in
      \Sigma_{V(\tau)}}\Hom_{K}(V(\tau), \Pi[ \mm_x]^{\la})\otimes
    V(\tau).\] Since the assumptions (i) and (ii) imply via Theorem
  \ref{patched_density} that the subspace is dense we conclude that
  the action of $R^{\rig}\otimes Z(\mathfrak g)_L$ on $\Pi^{\Rla}$
  factors through the action of
  $(R^{\rig}\otimes Z(\mathfrak g)_L)/\mathfrak a\cong R^{\rig}$. Thus
  the action of $Z(\mathfrak g)$ on $\Pi^{\Rla}$ factors through
  $\zeta$. If $y\in \mSpec R[1/p]$ then the action of
  $(R^{\rig}\otimes Z(\mathfrak g)_L)/\mathfrak a$ on
  $\Pi[\mm_y]^{\la}= \Pi[\mm_y]^{\Rla}=\Pi^{\Rla}[\mm_y]$ factors
  through the specialisation of $\zeta$ at $y$.
\end{proof}

We want to prove a variant of Theorem \ref{infinitesimal} with a fixed
central character.  Let $Z$ be the centre of $G$ and let
$\psi: K\cap Z(\Qp)\rightarrow \OO^{\times}$ be a continuous group
homomorphism. Let $\Mod^{\pro}_{K, \psi}(\OO)$ be the category of
linearly compact $\OO[\![K]\!]$-modules on which $K\cap Z(\Qp)$ acts
by $\psi^{-1}$.  We assume that the central character of $\tau$ is
equal to $\psi$.
\begin{thm}\label{infinitesimal_central} Assume that the following
  hold:
  \begin{itemize}
  \item[(o)] $M$ is in $\Mod^{\pro}_{K, \psi}(\OO)$;
  \item[(i)] there is an $M$-regular sequence $y_1, \ldots, y_h$ in
    $\mm$, such that $M/(y_1, \ldots, y_h)M$ is a finitely generated
    $\OO[\![K]\!]$-module, which is projective in
    $\Mod^{\pro}_{K, \psi}(\OO)$;
  \item[(ii)] the rings $R_{V(\tau)}$ are reduced,
    $\forall V\in \Irr_{G/Z}(L)$;
  \item[(iii)] the action of $Z(\mathfrak g)_L$ on $V(\tau)$ is given
    by $\zeta_x$, $\forall V\in \Irr_{G/Z}(L)$,
    $\forall x\in \Sigma_{V(\tau)}$.
  \end{itemize}
  Then $Z(\mathfrak g)_L$ acts on $\Pi^{\Rla}$ via $\zeta$. In
  particular, for all $y\in \mSpec R[1/p]$, $Z(\mathfrak g)_L$ acts on
  $\Pi[\mm_y]^{\la}$ via the infinitesimal character $\zeta_y$.
\end{thm}
\begin{proof} One proves the analog of Theorem \ref{patched_density}
  using Corollary \ref{summand_central} instead of Corollary
  \ref{summand}. Then the proof is the same as the proof of Theorem
  \ref{infinitesimal}.
\end{proof}

\section{Global applications}\label{sec_global}
The abstract Theorem \ref{infinitesimal} is motivated by the study of
Hecke eigenspaces in the completed cohomology. We will show that if we
are in the situation where the Hecke eigenvalues contribute only in
one cohomological degree and we can attach Galois representations to
the Hecke eigenspaces in a way to be made precise below, then the
locally analytic vectors in the Hecke eigenspaces in the completed
cohomology afford an infinitesimal character.

\subsection{Locally symmetric spaces}
We are motivated by Emerton's ICM talk \cite{emerton_icm} and his
paper \cite{interpolate}.  Let $G$ be a connected reductive group over
$\QQ$, let $Z$ be its centre, let $A$ denote the maximal split torus
in $Z$. Let $G_p=G(\Qp)$, let $G_{\infty}=G(\RR)$, let
$A_{\infty}=A(\RR)$, let $Z_{\infty}=Z(\mathbb R)$ be the centre of
$G_{\infty}$ and let $K_{\infty}$ denote a choice of maximal compact
subgroup of $G_{\infty}$. For any Lie group $H$, we let $H^{\circ}$
denote the subgroup consisting of the connected component at the
identity.

The quotient $G_{\infty}/ Z_{\infty}^{\circ} K_{\infty}^{\circ}$ is a
symmetric space on which $G_{\infty}$ acts. We denote its dimension by
$d$.
  
Let $\mathbb A$ denote the ring of adeles over $\QQ$, $\mathbb A_f$
the ring of finite adeles, and $\mathbb A^{p}_f$ the ring of
prime-to-$p$ adeles. \begin{defi} A compact open subgroup
  $K_f \subset G(\mathbb A_f)$ is sufficiently small if $G(\QQ)$ acts
  on $G(\mathbb A)/ A_{\infty}^{\circ} K_{\infty}^{\circ} K_f$ with
  trivial stabilisers.
\end{defi}

\begin{lem}\label{crit_suffsm} If $K_f=\prod_{\ell} K_{f, \ell}$ with $K_{f, \ell} \subset G(\mathbb{Q}_{\ell})$ compact open and 
  $K_{f, \ell}$ is torsion free for some $\ell$ then $K_f$ is
  sufficiently small.
 \end{lem} 
 \begin{proof} Let us first assume that $Z$ is
   $\mathbb Q$-anisotropic, so that $A_{\infty}^{\circ}$ is
   trivial. The stabiliser of $g K_{\infty}^{\circ} K_f$ is equal to
   $ G(\mathbb Q) \cap g K_{\infty}^{\circ} K_f g^{-1}$ and is both
   compact and discrete, hence finite. By considering the projection
   onto $g K_{f, \ell} g^{-1}$, we deduce it is trivial. Let us return
   to the general case. Since the centre of $G/A$ is anisotropic, we
   deduce that the $G(\mathbb Q)$-stabilizer of
   $g K_{\infty}^{\circ} K_f A_{\infty}^{\circ}$ is contained in
   $A(\mathbb A)$ and hence in $A(\QQ)$.  Since
   $A(\mathbb Q)\cap g K_{\infty}^{\circ}A_{\infty}^{\circ} K_f
   g^{-1}$ is finite we conclude by the same argument.
 \end{proof}   

 \begin{remar}\label{easy_suffsm} An easy way to ensure that the
   condition in Lemma \ref{crit_suffsm} holds both for $K_f$ and its
   image in $G(\mathbb A_f)/ Z(\mathbb A_f)$ is to embedd
   $G \subset \GL_n$ by choosing a faithful representation over $\QQ$.
   If $\ell > n+1$ then the pro-$\ell$ Sylow in
   $\GL_n(\mathbb Q_{\ell})$ is $\ell$-saturable by \cite[III
   (3.2.7)]{Lazard} and hence torsion free. If we assume that
   $K_{f, \ell}$ is pro-$\ell$, for example a pro-$\ell$ Iwahori
   subgroup if $G$ is split over $\mathbb Q_{\ell}$, then it will be
   torsion free. To make sure that the assertion also holds for the
   image of $K_f$ in $G(\mathbb A_f)/ Z(\mathbb A_f)$, pick an
   embedding $G/ Z\hookrightarrow \GL_{n'}$ and repeat the argument.
   Thus if $\ell > \max(n, n') +1$ and $K_{f, \ell}$ is pro-$\ell$
   then both conditions are satisfied.
\end{remar}

Let us assume that $K_f$ satisfies the condition of Lemma
\ref{crit_suffsm} and its image in $(G/ Z)(\mathbb A_f)$ is
sufficiently small with respect to $G/Z$ and let
\[\Lambda=\Lambda(K_f)\coloneqq Z(\QQ)\cap K_f K_{\infty}^{\circ}
  Z_{\infty}^{\circ}.\]

\begin{lem}\label{control_Gamma} $\Lambda$ is a discrete cocompact
  subgroup of $Z_{\infty}/ A_{\infty}$.  In particular,
  $\Lambda\cong \mathbb Z^r$, where $r$ is the split $\RR$-rank of
  $Z_{\infty}/ A_{\infty}$.
\end{lem}
\begin{proof} By assumption on $K_f$, $\Lambda$ is torsion free and
  $Z(\QQ)\cap K_f K_{\infty}^{\circ} A_{\infty}^{\circ}$ is
  trivial. Thus the map induced by the projection to the
  $\infty$-component induces an injection
  $\Lambda \hookrightarrow Z_{\infty}/ A_{\infty}$. Since the torus
  $Z/A$ is $\QQ$-anisotropic, it follows from
  \cite[Thm.\,4.11]{platonov_rapinchuk} that the quotient is compact.
\end{proof} 

\begin{remar} Emerton in \cite[\S 2.4]{interpolate} assumes that
  $r=0$. We don't do this because we would like to cover the case of
  Shimura curves in the applications.
\end{remar}

\begin{lem}\label{quotient_gamma} The action of $G(\QQ)$ on $G(\mathbb A)/ Z_{\infty}^{\circ} K_{\infty}^{\circ} K_f$ factors through the action of 
  $G(\QQ)/\Lambda$, which acts with trivial stabilisers.
\end{lem}
\begin{proof} The first part is clear. Since the image of $K_f$ in
  $G(\mathbb A_f)/Z(\mathbb A_f)$ is sufficiently small,
  $G(\QQ)/ Z(\QQ)$ acts with trivial stabilisers on
  $G(\mathbb A)/ Z(\mathbb A)K_{\infty}^{\circ} K_f$. The rest of the
  proof is an exercise for the reader.
\end{proof}

If $K_f$ is a compact open subgroup of $G(\mathbb A)$, we write
\[ Y(K_f)\coloneqq G(\QQ)\backslash G(\mathbb A)/ A_{\infty}^{\circ}
  K_{\infty}^{\circ} K_f, \quad \Ytilde(K_f)\coloneqq G(\QQ)\backslash
  G(\mathbb A)/ Z_{\infty}^{\circ} K_{\infty}^{\circ} K_f.\] It
follows from Lemma \ref{quotient_gamma} that the map
$q:Y(K_f)\rightarrow \Ytilde(K_f)$ makes $Y(K_f)$ into a torus bundle
for the compact torus
\[\mathcal T\coloneqq (Z_{\infty}^{\circ}/ A_{\infty}^{\circ} (Z_{\infty}^{\circ}\cap K_{\infty}^{\circ}))/\Lambda\cong \RR^r_{>0}/\Lambda\cong \RR^r/\ZZ^r.\]
Let $G^1_{\infty}$ be the intersection of $\Ker \chi^2$ for all
$\chi: G_{\RR} \rightarrow \mathbb G_{m, \RR}$. We then have
\begin{equation}\label{park_city}
  G_{\infty}= G^1_{\infty} Z^{\circ}_{\infty},
\end{equation} 
see \cite[4.3.1]{borel_park}.  Let
$G(\mathbb A)^1\coloneqq G(\mathbb A_f) G^1_{\infty}
A_{\infty}^{\circ}$ and let
$G(\QQ)^1\coloneqq G(\QQ) \cap G(\mathbb A)^1$ and let
\[ Y^1(K_f)\coloneqq G(\QQ)^1\backslash G(\mathbb A)^1/
  A_{\infty}^{\circ} K_{\infty}^{\circ} K_f.\]
\begin{lem} The natural map $Y^1(K_f)\rightarrow Y(K_f)$ identifies
  the source with the closed subset of the target.  The map
  $\mathcal T \times Y^1(K_f)\rightarrow Y(K_f)$, $(t, y)\mapsto ty$
  is a homeomorphism. In particular, we obtain a
  homeomorphism
  \[Y^1(K_f) \overset{\cong}{\longrightarrow} Y(K_f)/\mathcal T =
    \Ytilde(K_f).\]
\end{lem}
\begin{proof} The inclusion $G(\mathbb A)^1 \subset G(\mathbb A)$ induces an injection 
  \[G(\QQ)^1\backslash G(\mathbb A)^1 \subset G(\QQ)\backslash
    G(\mathbb A),\] and hence an injection
  $Y^1(K_f)\hookrightarrow Y(K_f)$. Since the map
  $G(\mathbb A)\rightarrow Y(K_f)$ is open and $G(\mathbb A)^1$ is a
  closed subset of $G(\mathbb A)$ we deduce that $Y^1(K_f)$ is a
  closed subset of $Y(K_f)$.  Since $Y(K_f)$ is a $\mathcal T$-bundle
  over $\Ytilde(K_f)$, it is enough to show that the map
  $Y^1(K_f)\rightarrow \Ytilde(K_f)$ is bijective.  From
  \eqref{park_city} we obtain an equality:
  \[ Y^1(K_f)= G(\QQ)^1\backslash G(\mathbb A)/ Z_{\infty}^{\circ}
    K_{\infty}^{\circ} K_f=G(\QQ)^1Z_{\infty}^{\circ}\backslash
    G(\mathbb A)/ Z_{\infty}^{\circ} K_{\infty}^{\circ} K_f.\] Thus it
  is enough to show that $G(\QQ)\subset G(\QQ)^1 Z_{\infty}^{\circ}$,
  which follows by considering the map
  $ G(\mathbb A)\rightarrow G_{\infty}\rightarrow G_{\infty}/
  G^1_{\infty} A_{\infty}^{\circ}$ and using \eqref{park_city} again.
\end{proof}

If $M$ is a $K_f$-module then following \cite[Def.\,2.2.3]{interpolate} we define a local system 
\[ \mathcal M\coloneqq ( M\times ( G(\QQ)\backslash G(\mathbb A)/
  A_{\infty}^{\circ} K_{\infty}^{\circ}))/K_f, \quad \Mtilde\coloneqq
  ( M\times ( G(\QQ)\backslash G(\mathbb A)/ Z_{\infty}^{\circ}
  K_{\infty}^{\circ}))/K_f,\] on $Y(K_f)$ and $\Ytilde(K_f)$
respectively. We have $q_{\ast} \mathcal M= \Mtilde$.  If $N$ is the
maximal quotient of $M$ on which $\Lambda$ acts trivially then
$\Mtilde=\widetilde{\mathcal N}$.

\begin{lem}\label{Kuenneth} If $\Lambda$ acts trivially on $M$ then
  for all $n\ge 0$ we have a natural isomorphism
  \[H^n(Y(K_f), \mathcal M)\cong \bigoplus_{i+j=n} H^i(\Ytilde(K_f),
    \Mtilde)\otimes H^j(\mathcal T, \ZZ).\]
\end{lem} 
\begin{proof} Since $H^j(\mathcal T, \ZZ)$ is a free $\ZZ$-module for
  all $j\ge 0$, the assertion follows from K\"unneth formula, see for
  example \cite[Thm.\,15.10]{demailly}.
\end{proof} 

\subsection{Completed cohomology}
It follows from Lemma \ref{quotient_gamma} that if $K'_f$ is an open
normal subgroup of $K_f$ then $\Ytilde(K_f')\rightarrow \Ytilde(K_f)$
is Galois covering with Galois group $K_f/ \Lambda(K_f) K_f'$. We want
to vary $K_f$ by shrinking the subgroup at $p$.  We fix a compact open
subgroup $K^p_f \subset G(\mathbb A^p_f)$ and assume that for some
sufficiently large $\ell$, as in Remark \ref{easy_suffsm},
$K^p_{f, \ell}$ is pro-$\ell$. As a consequence we have that for all
compact open subgroups $K_p \subset G(\Qp)$ both $ K^p_fK_p$ and its
image in $G(\mathbb A_f)/ Z(\mathbb A_f)$ are sufficiently small and
$K^p_fK_p \cap G(\QQ)$ is torsion free. We set up things this way,
because we don't want to put any restrictions on the subgroup $K_p$,
but all three properties hold if $K^p_f$ is arbitrary and $K_p$ is
small enough.

We then define the completed (co)ho\-mo\-logy by 
\[\widetilde{H}^i(\OO/\varpi^s)\coloneqq   \varinjlim_{K_p} H^i(\Ytilde(K^p_f K_p), \OO/\varpi^s), \quad 
  \widetilde{H}^i\coloneqq \varprojlim_{s}
  \widetilde{H}^i(\OO/\varpi^s)\]
\[ \widetilde{H}_i(\OO/\varpi^s)\coloneqq \varprojlim_{K_p}
  H_i(\Ytilde(K^p_f K_p), \OO/\varpi^s),\quad \widetilde{H}_i\coloneqq
  \varprojlim_{s} \widetilde{H}_i(\OO/\varpi^s).\] where in the limits
$K_p$ ranges over all compact open subgroups of $G(\Qp)$.  The
topologies are as follows: discrete on
$\widetilde{H}^i(\OO/\varpi^s)$, $p$-adic on $\widetilde{H}^i$,
projective limit topology on $\widetilde{H}_i(\OO/\varpi^s)$ and
$\widetilde{H}_i$. Note that $\widetilde{H}^i(\OO/\varpi^s)$ and
$\widetilde{H}_i(\OO/\varpi^s)$ are related by the Pontryagin duality.
If $K_p$ is a compact open subgroup of $G(\Qp)$ then $\widetilde{H}_i$
is a finitely generated $\OO[\![K_p]\!]$-module,
\cite[Thm.\,2.2]{emerton_icm}.

We will consider $K_p$-modules $M$ as $K^p_f K_p$-modules by making
$K^p_f$ act trivially.  We can then obtain local systems $\mathcal M$
and $\Mtilde$ as in the previous subsection.  Let $\Lambda_p$ be the
closure of $\Lambda(K^p_f K_p)$ in $K_p$.  If
$M= \Ind_{K_p'\Lambda_p}^{K_p} M'$ then $\Mtilde$ is equal to
$\Mtilde'$ defined with respect to $K_p'$.  Since cohomology commutes
with direct limits, see for example \cite[Lem.\,2.5]{HH}, by writing
$\mathcal C(K_p/\Lambda_p, \OO/\varpi^s)\cong \varinjlim_{K_p'}
\Ind_{K_p'\Lambda_p}^{K_p} \OO/\varpi^s$ we obtain
\begin{equation}\label{hill}
  \widetilde{H}^i(\OO/\varpi^s)\cong H^i( \Ytilde(K^p_f K_p),  \widetilde{\mathcal C(K_p/\Lambda_p, \OO/\varpi^s)}).
\end{equation}
\begin{lem}\label{spectral} Let $M$ be a finite $\OO/\varpi^s$-module
  with continuous $K_p/ \Lambda_p$-action. Then there is a spectral
  sequence
  \[ E_2^{ij}= \Ext^i_{K_p/\Lambda_p}(M,
    \widetilde{H}^j(\OO/\varpi^s))\Longrightarrow
    H^{i+j}(\Ytilde(K^p_f K_p), \Mtilde^{\vee}),\] where
  $\Mtilde^{\vee}$ is the local system associated to the Pontryagin
  dual of $M$.
\end{lem}
\begin{proof} We use an argument due to Hill \cite{hill} in an easier
  setting, alternatively see \cite[(2.1.10)]{interpolate}.  We pick a
  triangulation of $\Ytilde(K^p_f K_p)$ and write down the \v{C}ech
  complex computing the cohomology of the local system associated to
  $\mathcal C(K_p/\Lambda_p, \OO/\varpi^s)$.  We then apply
  $\Hom_{K_p/\Lambda_p}(M, \ast)$ to obtain a complex computing the
  cohomology of $\Mtilde^{\vee}$. This yields the desired spectral
  sequence.
\end{proof}

\subsection{Hecke algebra}\label{sec_hecke}
We fix a finite set of places $S$ containing $p$ and $\infty$, such
that $G$ is unramified over $\QQ_\ell$ for all $\ell\not\in S$. We
assume that we may factor $K^p_{f}= K^{p, \ell}_f K_{\ell}$ with
$K_{\ell}$ hyperspecial for all $\ell \not \in S$. Let
$\mathcal H_{\ell}$ be the double coset algebra
$\OO[K_\ell\backslash G(\QQ_\ell)/K_\ell]$ and let
\[\mathbb T^{\univ}\coloneqq  \bigotimes_{\ell\not \in S} \mathcal H_{\ell}\]
where the tensor product (defined by its universal property) is taken
over $\OO$. One may think of $\mathbb T^{\univ}$ as a polynomial ring
over $\OO$ in infinitely many variables. If $M$ is an
$\OO[K_p]$-module then $\mathbb T^{\univ}$ acts on the cohomology of
the associated local system. Following \cite{emerton_icm} we define
$\mathbb T$ to be the closure of the image of $\mathbb T^{\univ}$ in
\begin{equation}\label{fat_module}
  \prod_{K_p} \prod_M \prod_i \End_{\OO}(H^i(\Ytilde(K^p_f K_p), \Mtilde)),
\end{equation}
where the product is taken over all compact open subgroups
$K_p\subset G(\Qp)$, all finite $\OO$-torsion modules $M$ with
continuous $K_p/\Lambda_p$-action, and all degrees $i$, where the
target is equipped with the profinite topology.

\begin{lem}\label{neat} Let $K_p$ be an open pro-$p$ subgroup of  $G(\Qp)$, 
  let $\mm$ be an open maximal ideal of $\mathbb T$ and let $i$ be a
  non-negative integer.  Then the following are equivalent:
  \begin{itemize}
  \item[(i)] $H^i(\Ytilde(K^p_f K_p), \OO/\varpi)_{\mm}=0$;
  \item[(ii)] $H^i(\Ytilde(K^p_f K_p), \mathcal M)_{\mm}=0$, for all
    representations of $K_p$ on finitely generated $\OO$-torsion
    modules $M$;
  \item[(iii)] $H^i(\Ytilde(K^p_f K_p'), \OO/\varpi)_{\mm}=0$, for all
    open normal subgroups $K'_p \subset K_p$.
  \end{itemize}
\end{lem}
\begin{proof} Since $K_p$ is pro-$p$ any such $M$ has a composition
  series with graded pieces isomorphic to $\OO/\varpi$ with the
  trivial $K_p$-action. Since localisation is exact part (i) implies
  (ii). Part (ii) implies (iii) by considering the induced module.
  Part (iii) trivially implies (i).
\end{proof}
\begin{cor} $\mathbb T$ has only finitely many open maximal ideals.  
\end{cor}
\begin{proof} If $\mm$ is open then it has to lie in the support of
  some $H^i(\Ytilde(K^p_f K_p'), \mathcal M)$. By embedding $M$ into
  in an induction from the trivial representation of a smaller
  subgroup, we may assume that $K_p'$ is an open normal subgroup of
  $K_p$. Lemma \ref{neat} implies that $\mm$ lies in the support of
  $H^i(\Ytilde(K^p_f K_p), \OO/\varpi)$. Thus if we let
  $\overline{\mathbb T}$ be the image of $\mathbb T$ in the finite
  dimensional $k$-algebra
  $\End_k(\oplus_{i=0}^{d} H^i(\Ytilde(K^p_f K_p), \OO/\varpi))$ then
  the set of maximal ideals of $\overline{\mathbb T}$ coincide with
  the set of open maximal ideals of $\mathbb T$.
\end{proof}

\begin{remar}\label{chinese} By applying the Chinese remainder theorem at each finite level of \eqref{fat_module}, we obtain 
  an isomorphism
  $\mathbb T\cong \prod_{\mm} \widehat{\mathbb T}_{\mm}$, where the
  product is taken over the open maximal ideals and
  $\widehat{\mathbb T}_{\mm}\coloneqq \varprojlim_n \mathbb T/\mm^n$
  with the topology induced by the limit. Since the product is finite
  the completion coincides with localisation, so that
  $\widehat{\mathbb T}_{\mm}=\mathbb T_{\mm}$. This explains why
  localisation at $\mm$ behaves well for various topological modules
  of $\mathbb T$. Below subscript $\mm$ always means localisation at
  $\mm$.  If $\mathbb T_{\mm}$ is noetherian then the topology will
  coincide with the $\mm$-adic topology.
\end{remar}

\subsection{Weakly non-Eisenstein ideals} 

\begin{defi}\label{strongly_non} We say that an open maximal ideal $\mm$ of 
  $\mathbb T$ is weakly non-Eisenstein if there is an integer $q_0$
  such that the equivalent conditions of Lemma \ref{neat} hold for all
  $i\neq q_0$.
\end{defi}

\begin{remar}If $G_{\infty}/Z_{\infty}$ is compact then every open
  maximal ideal is weakly non-Eisenstein.
\end{remar}
\begin{remar} We note that one does not expect weakly non-Eisenstein
  ideals to exist unless the rank of $G_{\infty}$ is equal to the rank
  of $Z_{\infty} K_{\infty}$, which corresponds to the assumption
  $l_0=0$ for the derived subgroup of $G_{\infty}$. In that case, $d$
  is even and $q_0=d/2$.
\end{remar}

\begin{lem}\label{need_this} If $\mm$ is weakly non-Eisenstein then
  the following hold:
  \begin{itemize}
  \item[(i)] $\widetilde{H}_{q_0, \mm}$ is a projective finitely
    generated $\OO[\![K_p/\Lambda_p]\!]$-module;
  \item[(ii)] $\mathbb T_{\mm}$ acts faithfully on
    $\widetilde{H}_{q_0, \mm}$;
  \item[(iii)] There are natural $\mathbb T_{\mm}[G(\Qp)]$-equivariant
    homeomorphisms:
    \[ \widetilde{H}^{q_0}_{\mm}\cong
      \Hom^{\cont}_{\OO}(\widetilde{H}_{q_0, \mm}, \OO), \quad
      \widetilde{H}_{q_0, \mm}\cong
      \Hom^{\cont}_{\OO}(\widetilde{H}^{q_0}_{\mm}, \OO).\]
  \end{itemize}
\end{lem}
\begin{proof} Since the functor
  $M\mapsto H^i(\Ytilde(K^p_f K_p), \Mtilde)$ commutes with direct
  limits, we obtain from \eqref{hill} that
  $\widetilde{H}^j(\OO/\varpi^s)_{\mm}=0$ for all $j\neq q_0$ and all
  $s\ge 0$. This implies that $\widetilde{H}^j_{\mm}=0$ for
  $j\neq q_0$ and the spectral sequence in Lemma \ref{spectral}
  degenerates to give:
\begin{equation}\label{degenerate}
  \Hom_{K_p}(M, \widetilde{H}^{q_0}(\OO/\varpi^s)_{\mm})\cong H^{q_0}( \Ytilde(K^p_f K_p), \Mtilde)_{\mm}.
 \end{equation}
 The assumption on $\mm$ implies that
 $M\mapsto H^i(\Ytilde(K^p_f K_p), \Mtilde)_{\mm}$ is exact, thus we
 deduce that $\widetilde{H}^{q_0}(\OO/\varpi^s)_{\mm}$ is injective in
 the category smooth representations of $K_p/\Lambda_p$ on
 $\OO/\varpi^s$-modules. By Pontryagin duality we get that
 $\widetilde{H}_{q_0}(\OO/\varpi^s)_{\mm}$ is projective in the
 category of compact $\OO/\varpi^s[\![K_p/\Lambda_p]\!]$-modules. By
 passing to the limit we obtain part (i). Any $M$ as in Lemma
 \ref{spectral} may be embedded into
 $\mathcal C(K_p/\Lambda_p, \OO/\varpi^s)^{\oplus m}$ for some
 $m$. Using \eqref{hill} we get an embedding
 $H^{q_0}( \Ytilde(K^p_f K_p), \Mtilde)_{\mm}\hookrightarrow
 \widetilde{H}^{q_0}(\OO/\varpi^s)_{\mm}^{\oplus m}$.  This yields
 part (ii). It follows from part (i) that $\widetilde{H}_{q_0, \mm}$
 is $\OO$-torsion free. Part (iii) follows from Schikhof duality, see
 the discussion in \cite[\S 2]{Ludwig}, together with the identity
 $\widetilde{H}^{q_0}(\OO/\varpi^s)_{\mm}^{\vee}\cong
 \widetilde{H}_{q_0}(\OO/\varpi^s)_{\mm}\cong
 \widetilde{H}_{q_0,\mm}/\varpi^s$; alternatively one could use
 \cite[Thm.\,1.1 3)]{cc_survey}.
\end{proof}

\subsection{Automorphic forms}
Let $V$ be an irreducible algebraic representation of $G$ over $L$. As
in the previous section we evaluate $V$ at $L$, and make $G(\Qp)$ and
$K_p$ act on it via $G(\Qp)\hookrightarrow G(L)$. We assume that
$\Lambda_p$ acts trivially on $V$. We fix a $K_p$-invariant lattice
$M$ in $V$.  Let $\mathcal M^d$ the local system associated to
$M^d\coloneqq \Hom_{\OO}(M, \OO)$, let $\mathcal V^*$ the local system
associated to $V^*$ on $Y(K_p^f K_p)$ and let $i$ be a non-negative
integer.  We let
\[ H^{i}(\mathcal V^*)\coloneqq \varinjlim_{K_p'} H^i(Y(K^p_f K_p'),
  \mathcal V^*), \quad H^i(\mathcal M^d)\coloneqq \varinjlim_{K_p'}
  H^i(Y(K^p_f K_p'), \mathcal M^d).\] Then
$H^i(\mathcal V^*)= H^i(\mathcal M^d)\otimes_{\OO} L$.
 
Let $\tau$ be a smooth absolutely irreducible representation of $K_p$
on an $L$-vector space with the property that if $\pi_p$ is a smooth
irreducible $\overline{L}$-representation of $G(\Qp)$ then
$\Hom_{K_p}(\tau, \pi_p)\neq 0$ implies that $\pi_p$ is
supercuspidal. We will call such representations \textit{supercuspidal
  types}. If $p$ is bigger than the Coxeter number of $G$ then
examples of supercuspidal types for every open $K_p$ can be obtained
by inducing representations considered in
\cite[Thm.\,2.2.15]{fintzen_shin} and letting $\tau$ be an irreducible
subquotient; if $G=\GL_n$ then there is no need to put restrictions on
$p$, see \cite[Prop.\,3.19]{EP}. We assume that $\Lambda_p$ acts
trivially on $\tau$.

We fix an embedding $L\hookrightarrow \mathbb C$. Following \cite[\S
2.1.6]{emerton_icm}, we let
\[\mathcal A(K^p_f)\coloneqq \varinjlim_{K_p'}\mathcal A(K^p_f K_p'),\] 
where $\mathcal A(K^p_f K_p')$ is the space of automorphic forms on
$G(\QQ)\backslash G(\mathbb A)/K^p_f K_p'$. Let $\chi$ be the
character through which $A_{\infty}^{\circ}$ acts on $V_{\mathbb
  C}$. Let $\mathcal A(K^p_f)_{\chi}$ be the subspace of
$\mathcal A(K^p_f)$ on which $A_{\infty}^{\circ}$ acts through $\chi$.
Franke's theorem \cite{franke} implies, see
\cite[Thm.\,2.3]{franke_schwermer}, that
\begin{equation}\label{Eisenstein}
  H^i(\mathcal V^*)\otimes_L \mathbb C\cong H^i( \widetilde{\mathfrak g}, \mathfrak k; \mathcal A(K^p_f)_{\chi} \otimes V^*_{\mathbb C}),
 \end{equation}
 where $\widetilde{G}_{\infty}$ is the group of real points of the
 intersection of the kernels of all the rational characters of $G$,
 $\widetilde{\mathfrak g}$ is its Lie algebra and $\mathfrak k$ is the
 Lie algebra of $K_{\infty}$. The space $\mathcal A(K^p_f)_{\chi}$
 decomposes into the cuspidal part
 $\mathcal A_{\mathrm{cusp}}(K^p_f)_{\chi}$ and its orthogonal
 complement $\mathcal A_{\mathrm{Eis}}(K^p_f)_{\chi}$. Since $\tau$ is
 a supercuspidal type, it follows from the description of
 $\mathcal A_{\mathrm{Eis}}(K^p_f)_{\chi}$ in the course of the proof
 of \cite[Prop.\,3.3, Eqn.\,(3), (4)]{franke_schwermer} as a quotient
 of a direct sum of parabolically induced representations, that
 $\Hom_{K_p}(\tau_{\mathbb C}, \mathcal
 A_{\mathrm{Eis}}(K^p_f)_{\chi})=0$. Thus we obtain a
 $\mathbb T^{\univ}_{\mathbb C}$-equivariant isomorphism:
\begin{equation}\label{bookkeeping1}
  \Hom_{K_p}(\tau, H^i(\mathcal V^*))\otimes_{L} \mathbb C\cong \Hom_{K_p}(\tau_{\mathbb C}, 
  H^i( \widetilde{\mathfrak g}, \mathfrak k;\mathcal A_{\mathrm{cusp}}(K^p_f)_{\chi} \otimes V^*_{\mathbb C})).
\end{equation}
If we let
$\mathcal A_{\mathrm{cusp}}=\varinjlim_{K^p_f} \mathcal
A_{\mathrm{cusp}}(K^p_f)$ then $\mathcal A_{\mathrm{cusp}}$ decomposes
into a direct sum of irreducible representations
$\pi=\otimes'_v \pi_v$ of $G(\mathbb A)$. Moreover,
$\mathcal A_{\mathrm{cusp}}^{K^p_f}=\mathcal
A_{\mathrm{cusp}}(K^p_f)$. We thus obtain a
$\mathbb T^{\univ}_{\mathbb C}$-equivariant isomorphism
\begin{equation}\label{bookkeeping2}
\begin{split}
  \Hom_{K_p}(\tau_{\mathbb C},
  H^i&( \widetilde{\mathfrak g}, \mathfrak k; \mathcal A_{\mathrm{cusp}}(K^p_f)_{\chi} \otimes V^*_{\mathbb C}))\\
  &\cong \bigoplus_{\pi} H^i( \widetilde{\mathfrak g}, \mathfrak k;
  \pi_{\infty}\otimes V^*_{\mathbb C})\otimes \Hom_{K_p}(\tau_{\mathbb
    C}, \pi_p)\otimes (\pi^{p, \infty})^{K^p_f},
\end{split}
\end{equation}
where the sum is taken over all irreducible subrepresentations $\pi$
of of $\mathcal A_{\mathrm{cusp}}$ counted with multiplicities, such
that $A_{\infty}^{\circ}$ acts by $\chi$ on $\pi_{\infty}$, and, where
$\pi^{\infty, p}= \otimes'_{\ v\nmid p \infty} \pi_v$.
\begin{remar} The supercuspidal type $\tau$ is only used to ensure
  that $\mathcal A_{\mathrm{Eis}}(K^p_f)_{\chi}$ does not contribute
  to Hecke eigenvalues. If the group $G$ is anisotropic then
  $\mathcal A_{\mathrm{Eis}}(K^p_f)_{\chi}$ is zero and the use of
  supercuspidal type is redundant. See also the discussion in
  \cite[3.1.2]{emerton_icm}.
\end{remar} 
Let $\Mtilde^d$ the local system associated to
$M^d\coloneqq \Hom_{\OO}(M, \OO)$ and $\widetilde{\mathcal V}^*$ the
local system associated to $V^*$ on $\Ytilde(K_p^f K_p)$. If $K_p'$ is
an open normal subgroup of $K_p$ then
\begin{equation}
\begin{split}
  \Hom_{K_p'}(M, \widetilde{H}^{q_0}_{\mm})&\cong \varprojlim_s \Hom_{K_p'}(M, \widetilde{H}^{q_0}(\OO/\varpi^s)_{\mm})\\
  &\overset{\eqref{degenerate}}{\cong} \varprojlim_s H^i(\Ytilde(K^p_f
  K_p'), \widetilde{(\mathcal M/\varpi^s)^{\vee}})_{\mm}\\ &\cong
  H^i(\Ytilde(K^p_f K_p'), \widetilde{\mathcal M}^d)_{\mm},
\end{split}
\end{equation} 
where the last isomorphism follows from Mittag-Leffler. Thus we have
an isomorphism of $\mathbb T$-modules
\begin{equation}\label{direct_summand}
  H^{q_0}(\Ytilde(K^p_f K_p'), \widetilde{\mathcal M}^d)\cong \Hom_{K_p'}(M, \widetilde{H}^{q_0}_{\mm})\oplus \bigoplus_{\mm'\neq \mm} H^{q_0}(\Ytilde(K^p_f K_p'), \widetilde{\mathcal M}^d)_{\mm'},
\end{equation}
where the direct sum is over open maximal ideals of $\mathbb T$
different from $\mm$.

It follows from \eqref{direct_summand} that
$\Hom_{K_p}(V(\tau), \widetilde{H}^{q_0}_{\mm}\otimes_{\OO} L)$ is a
direct summand as a $\mathbb T$-module of
$\Hom_{K_p}(\tau, H^{q_0}(\widetilde{\mathcal V}^*))$, where
$V(\tau)\coloneqq V\otimes_L \tau$. It follows from Lemma
\ref{Kuenneth} that we have a $G(\mathbb A_f)$-equivariant isomorphism
\[\varinjlim_{K_f} H^{q_0}(Y(K_f), \mathcal V^*)\cong \bigoplus_{i+j=q_0} (\varinjlim_{K_f} H^{i}(\Ytilde(K_f), \widetilde{\mathcal V}^*))\otimes_{L} H^j( \RR^r/\ZZ^r, L).\]
After taking $K^p_f$-invariants we deduce that
$\Hom_{K_p}(\tau, H^{q_0}(\widetilde{\mathcal V}^*))$ is a direct
summand of $\Hom_{K_p}(\tau, H^{q_0}(\mathcal V^*))$ as a
$\mathbb T$-module.

Let $\mathbb T_{\mm, V(\tau)}$ be the quotient of $\mathbb T_{\mm}$ acting  faithfully on 
$\Hom_{K_p}(V(\tau), \widetilde{H}^{q_0}_{\mm}\otimes_{\OO} L)$.

\begin{lem}\label{TmVtau} The algebra $\mathbb T_{\mm, V(\tau)}$ is reduced. Moreover, if we fix an isomorphism $\Qpbar\overset{\cong}{\rightarrow}\mathbb C$ then for every
  $L$-algebra homomorphism
  $x: \mathbb T_{\mm, V(\tau)}[1/p]\rightarrow \Qpbar$ there is a
  cuspidal automorphic representation $\pi_x=\otimes'_v \pi_{x,v}$ of
  $G(\mathbb A)$, such that the following hold:
  \begin{itemize}
  \item[(i)] $\mathbb T_{\mm}$ acts on $\pi_x^{K^p_f}$ via $x$;
  \item[(ii)] $\Hom_{K_p}(\tau, \pi_{x,p})\neq 0$;
  \item[(iii)] $A_{\infty}^{\circ}$ acts on $\pi_{\infty}$ via $\chi$;
  \item[(iv)]
    $H^{q_0}( \widetilde{\mathfrak g}, \mathfrak k; \pi_{x, \infty}
    \otimes V^*_{\mathbb C})\neq 0$.
  \end{itemize}
\end{lem}
\begin{proof} Let $\mathbb T'_{V(\tau)}$ be the $L$-subalgebra of
  $\End_L( \Hom_{K_p}(\tau, H^{q_0}(\mathcal V^*)))$ generated by the
  image of $\mathbb T$. As explained above
  $\Hom_{K_p}(V(\tau), H^{q_0}_{\mm}\otimes_{\OO} L)$ is a direct
  summand of $\Hom_{K_p}(\tau, H^{q_0}(\mathcal V^*))$ as a
  $\TT$-module and hence it is enough to prove the assertion for
  $\mathbb T'_{V(\tau)}$. Since
  $\Hom_{K_p}(\tau, H^{q_0}(\mathcal V^*))$ is finite dimensional,
  $\mathbb T'_{V(\tau)}$ is a quotient of $\mathbb T^{\univ}[1/p]$.
  It follows from \eqref{bookkeeping1} and \eqref{bookkeeping2} that
  the assertion of the Lemma holds for $\mathbb T'_{V(\tau)}$.
\end{proof}

\begin{lem}\label{inf_match} For all $\pi_x$ in Lemma \ref{TmVtau},
  $\pi_{x,\infty}$ and $V$ have the same infinitesimal character.
\end{lem} 
\begin{proof} It follows from Lemma \ref{TmVtau} (iv) and
  \cite[Cor.\,I.4.2]{BW} that $Z(\widetilde{\mathfrak g})$ acts on
  $(\mathfrak g, \mathfrak k)$-module of $\pi_{x,\infty}$ and $V$ by
  the same infinitesimal character. Since $A_{\infty}^{\circ}$ acts on
  both representations by $\chi$ and
  $ \widetilde{G}_{\infty}A_{\infty}^{\circ}$ is of finite index in
  $G_{\infty}$, see for example \cite[4.3.1]{borel_park}, the
  assertion follows.
\end{proof}

\subsection{Main result}\label{main_result}
Part (iv) of Lemma \ref{TmVtau} implies that $\pi_x$ is cohomological,
and such automorphic forms are $C$-algebraic by
\cite[Lem.\,7.2.2]{beegee}. Thus according to the Conjecture 5.3.4 of
\cite{beegee}, that there should be an admissible Galois
representation $\rho_x: \Gal_{\QQ}\rightarrow \CG(\Qpbar)$ attached to
$\pi_x$ (or rather to the local information detailed in Lemma
\ref{TmVtau}).

\begin{thm}\label{lzero} 
  We assume that the following hold:
  \begin{itemize}
  \item[(o)] $Z_{\infty}/ A_{\infty}$ is compact;
  \item[(i)] $\mathbb T_{\mm}$ is noetherian;
  \item[(ii)] there is an admissible representation
    $\rho: \Gal_{\QQ}\rightarrow \CG_f(\mathbb T_{\mm}^{\rig})$, such
    that for all $V \in \Irr_{G}(L)$ and all
    $x: \mathbb T_{\mm, V(\tau)}[1/p]\rightarrow \Qpbar$, the
    specialisation of $\rho$ at $x$ matches $\pi_x$ according
    \cite[Conj.\,5.3.4]{beegee};
  \item[(iii)] the composition $d\circ \rho$ is equal to the $p$-adic
    cyclotomic character.
  \end{itemize}
  Then for all $y\in \mSpec \mathbb T_{\mm}[1/p]$ the centre
  $Z(\mathfrak g)$ acts on
  $(\widetilde{H}^{q_0}_{\mm} \otimes_{\OO} L)[\mm_y]^{\la}$ by the
  infinitesimal character $\zeta^C_{\rho, y}$.
\end{thm}
\begin{proof} We will show that the conditions of Theorem
  \ref{infinitesimal} holds for $M=\widetilde{H}_{q_0, \mm}$,
  $R=\mathbb T_{\mm}$ and $\zeta=\zeta^C_{\rho}$. The assumption that
  $Z_{\infty}/ A_{\infty}$ is compact implies that $\Lambda_p$ is
  trivial via Lemma \ref{control_Gamma}. Part (i) of Theorem
  \ref{infinitesimal} holds by Lemma \ref{need_this} with $h=0$. Part
  (ii) is given by Lemma \ref{TmVtau}. To show that part (iii) holds
  we denote by $\rho_x$ the specialisation of $\rho$ at $x$. It
  follows from Lemma \ref{base_change_C} that the specialisation of
  $\zeta^C_{\rho}$ at $x$ is equal to $\zeta^C_{\rho_x}$. Proposition
  \ref{C-alg-inf} implies that $\zeta^C_{\rho, x}$ is equal to the
  infinitesimal character of $\pi_{x,\infty}$ which is equal to the
  infinitesimal character of $V$ by Lemma \ref{inf_match}. Since
  $\tau$ is a smooth representation, part (iii) of Theorem
  \ref{infinitesimal} is satisfied.
\end{proof}

\begin{remar}\label{need_less} Let us note that we do not need the full force of the \cite[Conj.\,5.3.4]{beegee}, just the part relating the Hodge--Tate cocharacter of the Galois representation at places above $p$ to the infinitesimal 
  character of the automorphic representation at the archimedean
  places, as discussed in \cite[Rem.\,5.3.5]{beegee} in the amended
  version of the paper.
\end{remar}

The most difficult condition to check is part (ii). In the known cases
one obtains such representations by expressing $\mathbb T_{\mm}$ as a
quotient of a Galois deformation ring. This then automatically implies
part (i). Part (iii) is forced upon us by the conjecture of
Buzzard--Gee, since according to it part (iii) should hold for
representations associated to $x$ in part (ii), and it follows from
Theorem \ref{patched_density} that such points are Zariski dense in
$\mathbb T_{\mm}[1/p]$.

We will now prove a version of the theorem above with a fixed central
character.  It will enable us to remove the assumption that
$Z_{\infty}/ A_{\infty}$ is compact.

\begin{lem}\label{finite_index} The closure of the image of $Z(\mathbb A^S)$ inside 
  $Z(\mathbb A)/Z(\QQ)$ is a subgroup of finite index. Moreover, if
  $Z=\Res_{F/\QQ} Z'$, where $Z'$ is a split torus defined over a
  number field $F$, then the image is dense.
\end{lem}
\begin{proof} Since the map
  $Z(\mathbb A)\rightarrow Z(\mathbb A)/Z(\QQ)$ is continuous the
  inverse image of the closure is closed in $Z(\mathbb A)$ and
  contains $Z(\mathbb A^S) Z(\QQ)$. Thus it is enough to show that the
  closure of $Z(\mathbb A^S) Z(\QQ)$ is of finite index in
  $Z(\mathbb A)$.  Since $Z(\mathbb A^S)$ is closed in $Z(\mathbb A)$,
  this closure is equal to $Z(\mathbb A^S) C$, where $C$ is the
  closure of $Z(\mathbb Q)$ in $Z(\mathbb A_S)$. Since $C$ is of
  finite index in $Z(\mathbb A_S)$ by \cite[Cor.\,3.5]{sansuc}, we are
  done. If $Z=\Res_{F/\QQ} Z'$, where $Z'$ is split over $F$, then it
  follows from \cite[Prop.\,7.8]{platonov_rapinchuk} applied to $Z'$
  that $C=Z(\mathbb A_S)$. Thus $Z(\mathbb A^S) Z(\QQ)$ is dense in
  $Z(\mathbb A)$.
\end{proof} 
For each $\ell \not \in S$ let $\mathcal Z_{\ell}$ be the
$\OO$-subalgebra of $\mathcal H_{\ell}$ generated by functions with
support in $Z(\QQ_{\ell}) K_{\ell}$, let
$\mathcal Z^{\univ} \subset \mathbb T^{\univ}$ be the tensor product
of these algebras over $\OO$.  If
$\psi: Z(\mathbb A)\rightarrow \OO^{\times}$ is a character then we
will denote by $(\widetilde{H}^{q_0}_{\mm}\otimes_{\OO} L)_{\psi}$ the
$\psi$-eigenspace for the action of $Z(\mathbb A)$ on
$\widetilde{H}^{q_0}_{\mm}\otimes_{\OO} L$.

\begin{lem}\label{decomp_char} Let $\mathfrak a$ be the kernel of an $\OO$-algebra homomorphism $\varphi: \mathcal Z^{\univ} \rightarrow \OO$, such that $(\widetilde{H}^{q_0}_{\mm}\otimes_{\OO} L)[\mathfrak a]\neq 0$.
  There exist finitely many characters
  $\psi_i: Z(\mathbb A)\rightarrow \OO^{\times}$, $1\le i\le n$, such
  that
  \[(\widetilde{H}^{q_0}_{\mm}\otimes_{\OO} L)[\mathfrak a]=
    \bigoplus_{i=1}^n (\widetilde{H}^{q_0}_{\mm}\otimes_{\OO}
    L)_{\psi_i}.\] Moreover, if $Z=\Res_{F/\QQ} Z'$, where $Z'$ is a
  split torus defined over a number field $F$, then $n=1$.
\end{lem}
\begin{proof} The action of $Z(\mathbb A)$ on
  $\widetilde{H}^{q_0}_{\mm}\otimes_{\OO} L$ factors through
  $Z(\mathbb A)/ Z(\QQ)$.  Let $J$ be the closure of the image of
  $Z(\mathbb A^S)$ inside $Z(\mathbb A)/ Z(\QQ)$. Lemma
  \ref{finite_index} says that $J$ is of finite index in
  $Z(\mathbb A)/ Z(\QQ)$. Thus it is enough to show that there is a
  character $\psi: J\rightarrow \OO^{\times}$, such that
  $(\widetilde{H}^{q_0}_{\mm}\otimes_{\OO} L)[\mathfrak a]$ is equal
  to the $\psi$-eigenspace for the action of $J$ on
  $\widetilde{H}^{q_0}_{\mm}\otimes_{\OO} L$.

  Let $(\widetilde{H}^{q_0}_{\mm}\otimes_{\OO} L)[\mathfrak a]^0$ be
  the unit ball and let
  $v\in (\widetilde{H}^{q_0}_{\mm}\otimes_{\OO} L)[\mathfrak
  a]^0/\varpi^n$. Since the topology on
  $(\widetilde{H}^{q_0}_{\mm}\otimes_{\OO} L)[\mathfrak a]^0/\varpi^n$
  is discrete and the action of $J$ is continuous, the $J$-stabiliser
  $\Stab(v)$ of $v$ is open in $J$. Since $Z(\mathbb A^S)$ is dense in
  $J$, the map $Z(\mathbb A^S)\rightarrow J/\Stab(v)$ is surjective.
  On the other hand if $\ell \not\in S$ then $Z(\mathbb Q_{\ell})$
  acts on $(\widetilde{H}^{q_0}_{\mm}\otimes_{\OO} L)[\mathfrak a]$ by
  the character $g\mapsto \varphi(g K_{\ell})$, thus $J$ acts on $v$
  by the character $\psi_n$, uniquely determined by the formula
  $\psi_n(g) \equiv \varphi(g K_{\ell})\pmod{\varpi^n}$ for all
  $g\in Z(\mathbb Q_{\ell})$ and for all $\ell\not\in S$. The
  uniqueness implies that $J$ acts on
  $(\widetilde{H}^{q_0}_{\mm}\otimes_{\OO} L)[\mathfrak a]^0/\varpi^n$
  via $\psi_n$. By passing to the limit we obtain a character
  $\psi: J\rightarrow \OO^{\times}$, by which $J$ acts on
  $(\widetilde{H}^{q_0}_{\mm}\otimes_{\OO} L)[\mathfrak a]$, and which
  satisfies $\psi(g)= \varphi(g K_{\ell})$ for all
  $g\in Z(\QQ_{\ell})$ for all $\ell\not\in S$. In particular,
  $\psi$-eigenspace is contained in
  $(\widetilde{H}^{q_0}_{\mm}\otimes_{\OO} L)[\mathfrak a]$ and thus
  the two coincide.

  If $Z=\Res_{F/\QQ} Z'$, where $Z'$ is a split torus defined over
  $F$, then $J= Z(\mathbb A)/ Z(\QQ)$ by Lemma \ref{finite_index} and
  the assertion follows.
\end{proof}

We fix $V_0$ such that $\Lambda(K^p_fK_p)$ acts trivially on it and choose 
an $\OO$-algebra homomorphism $x_0: \mathbb T_{\mm, V_0(\tau)}\rightarrow \OO$. 
Let $\mathfrak a$ be the kernel of the composition 
\[\mathcal Z^{\univ}\rightarrow  \mathbb T_{\mm, V_0(\tau)}\overset{x_0}{\longrightarrow} \OO.\]
Then $(\widetilde{H}^{q_0}_{\mm}\otimes_{\OO} L)[\mathfrak a]$ is
non-zero, since it contains
$(\widetilde{H}^{q_0}_{\mm}\otimes_{\OO} L)[\mm_{x_0}]$, which is
non-zero as $x_0$ lies in the support
$\Hom_{K_p}(V_0(\tau), \widetilde{H}^{q_0}_{\mm}\otimes_{\OO} L)$.
Let $\psi: Z(\mathbb A)\rightarrow \OO^{\times}$ be one of the
characters in Lemma \ref{decomp_char}, such that
$(\widetilde{H}^{q_0}_{\mm}\otimes_{\OO} L)_{\psi}$ is non-zero.  Note
that $\psi|_{Z(\Qp)\cap K_p}$ is the central character of $V_0(\tau)$.
Let $\mathbb T_{\mm}^{\psi}$ be the quotient of $\mathbb T_{\mm}$
acting faithfully on
$(\widetilde{H}^{q_0}_{\mm}\otimes_{\OO} L)_{\psi}$, and let
$\mathbb T^{\psi}_{\mm, V(\tau)}$ be the quotient of
$\mathbb T_{\mm}^{\psi}$ acting faithfully on
\[\Hom_{K_p}(V(\tau), (\widetilde{H}^{q_0}_{\mm}\otimes_{\OO} L)_{\psi}).\]
If the action of $Z(\Qp)\cap K_p$ on $V(\tau)$ is not given by $\psi$ then 
$\mathbb T_{\mm, V(\tau)}^{\psi}$ will be zero. 
\begin{thm}\label{lzero_central} 
  We assume that the following hold:
  \begin{itemize}
  \item[(i)] $\mathbb T_{\mm}^{\psi}$ is noetherian;
  \item[(ii)] there is an admissible representation
    $\rho: \Gal_{\QQ}\rightarrow \CG_f(\mathbb T_{\mm}^{\psi, \rig})$,
    such that for all $V \in \Irr_{G}(L)$ and all
    $x: \mathbb T_{\mm, V(\tau)}^{\psi}[1/p]\rightarrow \Qpbar$, the
    specialisation of $\rho$ at $x$ matches $\pi_x$ according
    \cite[Conj.\,5.3.4]{beegee};
  \item[(iii)] the composition $d\circ \rho$ is equal to the $p$-adic
    cyclotomic character.
  \end{itemize}
  Then for all $y\in \mSpec \mathbb T_{\mm}^{\psi}[1/p]$ the centre
  $Z(\mathfrak g)$ acts on
  $(\widetilde{H}^{q_0}_{\mm} \otimes_{\OO} L)[\mm_y]^{\la}$ by the
  infinitesimal character $\zeta^C_{\rho, y}$.
\end{thm}
\begin{proof} We first note that since $\mm_x$ is an ideal of
  $\mathbb T_{\mm}^{\psi}[1/p]$, the subspaces annihilated by $\mm_x$
  in $\widetilde{H}^{q_0}_{\mm} \otimes_{\OO} L$ and in
  $(\widetilde{H}^{q_0}_{\mm} \otimes_{\OO} L)_{\psi}$ coincide. It
  follows from Lemma \ref{need_this} that
  $(\widetilde{H}^{q_0}_{\mm})_{\psi}$ is a direct summand of
  $\mathcal C_{\psi}(K_p, \OO)^{\oplus m}$. Thus its Schikhof dual $M$
  is projective in $\Mod^{\pro}_{K, \psi}(\OO)$ and is equal to the
  largest quotient of $\widetilde{H}_{q_0, \mm}$ on which
  $Z(\Qp)\cap K_p$ acts by $\psi^{-1}$. The algebras
  $\TT_{\mm, V(\tau)}$ are finite dimensional over $L$ and reduced by
  Lemma \ref{TmVtau}, hence products of finite field extensions of
  $L$. Thus the quotients $\TT_{\mm, V(\tau)}^{\psi}$ are also
  reduced. The rest of the proof is the same as the proof of Theorem
  \ref{lzero} using Theorem \ref{infinitesimal_central}.
\end{proof}

In sections \ref{sec_modular},\ref{sec_shimura}, \ref{sec_unitary} and
\ref{CS} we will discuss some examples, where the conditions of
Theorems \ref{lzero} and \ref{lzero_central} are satisfied.

\subsection{Modular curves}\label{sec_modular} Let $G=\GL_2$. Then $Z_\infty=A_{\infty}$, $d=2$, $q_0=1$, 
$K_{\ell}=\GL_2(\ZZ_{\ell})$ and
$\mathbb T^{\univ}=\OO[ T_{\ell}, S_{\ell}^{\pm 1}: \ell\not \in S]$.
Let $\mm$ be an open maximal ideal of $\mathbb T$. After extending
scalars we may assume that the residue field of $\mm$ is $k$. Since
$H^2$ is dual to $H^0_c$, which is contained in $H^0$, $\mm$ is weakly
non-Eisenstein if and only if $H^0(Y(K^p_f K_p), \OO/\varpi)_{\mm}=0$.

\begin{lem}\label{non-non}$H^0(Y(K_f^p K_p), \OO/\varpi)_{\mm}\neq 0$
  if and only if there is a character
  $\psi: \widehat{\ZZ}^*/ \det(K_f^pK_p)\rightarrow k^*$ such that
  \[T_{\ell} \equiv (\ell+1) \psi(\ell)\pmod{\mm}, \quad
    S_{\ell}\equiv \psi(\ell^2)\pmod{\mm}, \quad \forall \ell\not\in
    S.\]
\end{lem} 
\begin{proof} We may identify $H^0(Y(K_f^p K_p), \OO/\varpi)$ with the
  set of maps from the set of connected components of $Y(K_f^p K_p)$
  to $k$. The set of connected components can be identified with
  $\widehat{\ZZ}^*/ \det(K_f^p K_p)$. The action of
  $\GL_2(\mathbb A_f^p)$ on
  $\varinjlim_{K_p^f} H^0(Y(K_f^p K_p), \OO/\varpi)$ factors through
  the determinant. Since $H^0(Y(K_f^p K_p), \OO/\varpi)$ is the
  $K^p_f$-invariants of this representation we obtain the assertion.
\end{proof}

We assume that $\mm$ is weakly non-Eisenstein. It follows from
Deligne--Serre lemma that after extending scalars there is a
homomorphism of $\OO$-algebras
$x:\mathbb T_{\mm, V(\tau)}\rightarrow \OO$, lifting
$\mathbb T^{\univ}_{\mm} \rightarrow k$.  By composing it with our
fixed embedding $L\hookrightarrow \mathbb C$ we obtain a cuspidal
eigenform $f$. To it Deligne associates a Galois representation
$\rho_f: \Gal_{\mathbb Q, S}\rightarrow \GL_2(\Zpbar)$ such that
\[\tr \rho_f(\Frob_{\ell})\equiv T_{\ell}\pmod{\mm_x}, \quad \det
  \rho_f(\Frob_{\ell})\equiv \ell S_{\ell} \pmod{\mm_x}, \quad \forall
  \ell\not\in S.\] Let
$\rhobar: \Gal_{\mathbb Q, S}\rightarrow \GL_2(\Fpbar)$ be the
reduction of $\rho_f$ modulo $p$.  Then
\[\tr \rhobar(\Frob_{\ell})\equiv T_{\ell}\pmod{\mm}, \quad \det \rhobar(\Frob_{\ell})\equiv \ell 
  S_{\ell} \pmod{\mm}, \quad \forall \ell\not\in S.\] After extending
scalars we may assume that $\rhobar$ takes values in $\GL_2(k)$.

We assume that $\rhobar$ is absolutely irreducible. It is explained in
\cite[\S 5.2]{emerton_lg} there is a surjection
$R_{\rhobar} \twoheadrightarrow \mathbb T_{\mm}$, where $R_{\rhobar}$
is the universal Galois deformation ring of
$\rhobar:\Gal_{\mathbb Q, S}\rightarrow \GL_2(k)$. This implies that
$\mathbb T_{\mm}$ is noetherian. Let
$\rho: \Gal_{\mathbb Q, S}\rightarrow \GL_2(\mathbb T_{\mm})$ be the
corresponding Galois representation. This is a representation into an
$L$-group.  By using the twisting element
$\tilde{\delta}(t)=\bigl ( \begin{smallmatrix} 1 & 0 \\ 0 &
  t^{-1} \end{smallmatrix}\bigr)$ we obtain a representation
$\rho^C: \Gal_{\QQ, S}\rightarrow \CG_f(\mathbb T_{\mm})$ as in
section \ref{sec_Cgr}. This representation satisfies the condition
(ii) of Theorem \ref{lzero}, but it is not so easy to extract this
statement from \cite{beegee}. Instead, we observe that it follows from
\cite[Thm.\,7]{faltings} that the conditions in Remark \ref{need_less}
are satisfied. Thus Theorem \ref{lzero} holds when $G=\GL_2$ and $\mm$
is associated to an absolutely irreducible Galois representation
$\rhobar: \Gal_{\QQ, S}\rightarrow \GL_2(k)$.

\subsection{Shimura curves}\label{sec_shimura} Let $D$ be a quaternion algebra over a totally real field $F$ split at only one infinite place. 
Let $G'$ be the group over $F$ defined by $G'(A)= (D\otimes_{F} A)^*$
for $F$-algebras $A$ and let $G=\Res_{F/\Qp} G'$.  Then $d=2$, $q_0=1$
and the split $\RR$-rank of $Z_{\infty}/A_{\infty}$ is equal to
$[F:\QQ]-1$.  We assume that $S$ contains the ramification primes of
$F/\QQ$ and the primes below the ramification places of $D$. Let $S'$
be all the places of $F$ above the places in $S$.  Then
$\mathbb T^{\univ}= \OO[ T_v, S_v^{\pm 1}: v\not \in S']$. Let $\mm$
be a weakly non-Eisenstein ideal of $\mathbb T$. Arguing as in the
previous section and using the results of Carayol \cite{carayol} we
may assume that there is a Galois representation
$\rhobar: \Gal_{F, S'}\rightarrow \GL_2(k)$, such that
\begin{equation}\label{congruence}
  \tr \rhobar(\Frob_{v})\equiv T_{v}\pmod{\mm}, \quad \det \rhobar(\Frob_v)\equiv q_v
  S_{v} \pmod{\mm}, \quad \forall v\not\in S',
\end{equation}
where $q_v$ denotes the number of elements in the residue field of
$F_v$.

Let $x_0: \mathbb T_{\mm, V_0(\tau)}\rightarrow \OO$ be an
$\OO$-algebra homomorphism and let
$\psi: \mathbb A^{\times}_{f,F}/F^{\times} \rightarrow \OO^{\times}$
be as in section \ref{main_result}. If $\rho_0$ is the Galois
representation attached to the automorphic form corresponding to
$x_0$, then $\det \rho_0= \psi \chi_{\cyc}$.

Let $R_{\tr \rhobar}$ be the universal pseudodeformation ring
deforming the characteristic polynomial of $\rhobar$ as in
\cite{chenevier_det}, and let $R_{\tr \rhobar}^{\psi}$ be the quotient
of $R_{\tr \rhobar}$ which corresponds to the determinant
$\chi_{\cyc}\psi$. Arguing as in \cite{appendix} we obtain a
surjection
$R_{\tr \rhobar}^{\psi}\twoheadrightarrow \mathbb T_{\mm}^{\psi}$,
which implies that $\mathbb T_{\mm}^{\psi}$ is noetherian.

We will assume that $\rhobar$ is absolutely irreducible. Then
$R_{\tr \rhobar}^{\psi}$ is the universal deformation ring of
$\rhobar$ and thus by specialising the universal deformation along
$R_{\tr \rhobar}^{\psi}\twoheadrightarrow \mathbb T_{\mm}^{\psi}$ we
obtain a Galois representation
$\rho: \Gal_{F, S'}\rightarrow \GL_2(\mathbb T_{\mm}^{\psi})$.

As explained in \cite{borel_corvalis} there is a canonical bijection
between the equivalence classes of admissible representations of
$\Gal_{\QQ, S}$ into the $L$-group of $G$ and the admissible
representations of $\Gal_{F, S'}$ into the $L$-group of $G'$.  So we
will work with $\Gal_{F, S'}$.

By using the twisting element
$\tilde{\delta}(t)=\bigl ( \begin{smallmatrix} 1 & 0 \\ 0 &
  t^{-1} \end{smallmatrix}\bigr)$ we obtain a representation
$\rho^C: \Gal_{F, S'}\rightarrow \CG_f(\mathbb T_{\mm}^{\psi})$ as in
section \ref{sec_Cgr}. Specialisations of this representation at
$x\in \mSpec \mathbb T_{\mm,V(\tau)}^{\psi}[1/p]$ satisfies
\cite[Conj.\,5.3.4]{beegee}, but again it is not so easy to extract
this statement from \cite{beegee}. Instead, we observe that the
compatibility at $p$ and $\infty$ discussed in Remark \ref{need_less}
follows from \cite{skinner_documenta}.

We conclude that the conditions of Theorem \ref{lzero_central} are
satisfied, when $\mm$ is associated to an absolutely irreducible
Galois representation $\rhobar: \Gal_{F, S'}\rightarrow \GL_2(k)$, and
$Z(\mathfrak g)$ acts on
$(\widetilde{H}^{q_0}_{\mm} \otimes_{\OO} L)[\mm_x]^{\la}$ by the
infinitesimal character $\zeta^C_{\rho^C, x}$, for all
$x\in \mSpec \mathbb T_{\mm}^{\psi}[1/p]$.

\begin{remar} Let $\mm$ be an open maximal ideal of $\mathbb T$ and
  let $\rhobar: \Gal_{F, S'}\rightarrow \GL_2(k)$ be a Galois
  representation satisfying \eqref{congruence}. If $\rhobar$ is
  absolutely irreducible then $\mm$ is weakly non-Eisenstein, see
  \cite{newton}.
\end{remar}

\subsection{Definite unitary groups}\label{sec_unitary} Let $F$ be a totally real field and let $E$ be a 
quadratic totally imaginary extension of $F$. We assume that $[F:\QQ]$
is even, every place of $F$ above $p$ is split in $E$ and every finite
place of $F$ is unramified in $E$. Since $[F : \QQ]$ is even there
exists a unitary group $G'$ over $F$ which is an outer form of $\GL_n$
with respect to the quadratic extension $E/F$ such that $G'$ is
quasi-split at all finite places and anisotropic at all infinite
places. Let $G=\Res_{F/\QQ} G'$.  Then $d=0$ and $Z_\infty$ is
compact. We assume that $S$ contains the ramification places of
$F/\QQ$. Let $S'$ be the set of places of $E$ above the places in $S$.
Let $\mathbb T^{\mathrm{Spl},\univ}$ be the subalgebra of
$\mathbb T^{\univ}$ generated by subalgebras
$\mathcal H(G'(F_v) // K_{v})$ for those places $v$ of $F$, which
split completely in $E$ and are not above places in $S$. Let
$\mathbb T^{\mathrm{Spl}}$ be the closure of the image of
$\mathbb T^{\mathrm{Spl}, \univ}$ in $\mathbb T$.

Since $d=0$ every open maximal ideal of $\mathbb T^{\mathrm{Spl}}$ is
weakly non-Eisenstein.  As explained in Remark \ref{chinese} the set
of such ideals is finite and we have an isomorphism
$\mathbb T^{\mathrm{Spl}}\cong \prod_{\mm} \mathbb
T^{\mathrm{Spl}}_{\mm}$. As explained in \cite[Lem.\,C.7]{appendix}
every open maximal ideal of $\mm$ corresponds to a characteristic
polynomial of a Galois representation
$\rhobar: \Gal_{E, S'}\rightarrow \GL_n(\Fpbar)$. Let
$R^{\mathrm{ps}}_{\tr\rhobar}$ be the universal deformation ring
parameterising pseudorepresentations (or determinants) lifting the
characteristic polynomial of $\rhobar$ as in \cite{chenevier_det}. It
is shown in \cite[Thm.\,C.3]{appendix} that
$\mathbb T^{\mathrm{Spl}}_{\mm}$ is a quotient of
$R^{\mathrm{ps}}_{\tr\rhobar}$ and hence is noetherian.

Let us assume that $\mm$ corresponds to an absolutely irreducible
representation $\rhobar$. Then $R^{\mathrm{ps}}_{\tr\rhobar}$ is also
the universal deformation ring of $\rhobar$. Let
$\rho: \Gal_{E, S'}\rightarrow \GL_n(\mathbb T^{\mathrm{Spl}}_{\mm})$
be the representation obtained from the universal deformation
representation of $\rhobar$ by extending scalars along the surjection
$R^{\mathrm{ps}}_{\tr\rhobar}\twoheadrightarrow \mathbb
T^{\mathrm{Spl}}_{\mm}$.  If
$x: \mathbb T^{\mathrm{Spl}}_{\mm, V(\tau)}\rightarrow \Qpbar$ is a
homomorphism of $\OO$-algebras then $\rho_x$ is the Galois
representation associated to $\pi_x$ by Clozel, see
\cite[Thm.\,3.3.1]{fintzen_shin}. The local-global compatibility
between $\pi_x$ and $\rho_x$ is summarised in \cite[Thm.\,7.2.1]{EGH},
see the references in its proof for proper attributions.  In
particular, $\rho_x$ is potentially semi-stable with Hodge--Tate
cocharacter $\lambda+\tilde{\delta}$ (note that the convention
concerning Hodge--Tate numbers is different in \emph{loc.~cit.}),
where $\lambda$ is the highest weight of $V$ and
$\tilde{\delta}=(0,-1,\dots,1-n)$.  Let
$\rho^C: \Gal_{E, S'}\rightarrow \CG_f'(\mathbb
T^{\mathrm{Spl}}_{\mm})$ be the representation obtained from $\rho$
using the twisting element $\tilde{\delta}$.  It follows from
Proposition \ref{inf_HT} that $Z(\mathfrak{g})$ acts on $V$ through
the character $\zeta_{\rho_y}^C$.

It follows from Theorem \ref{lzero} that if $\mm$ is associated to an
absolutely irreducible $\rhobar$ then for all
$x\in \mSpec \mathbb T^{\mathrm{Spl}}_{\mm}[1/p]$, $Z(\mathfrak g)$
acts on $(H^0_{\mm}\otimes_{\OO}L)[\mm_x]^{\la}$ via
$\zeta^C_{\rho^C_x}$.

\subsection{Compact unitary Shimura varieties}\label{CS}

Let \(F\) be totally real field and \(E\) a CM quadratic extension of
\(F\). We assume that \(E\) contains a quadratic imaginary number
field. Let \(G\) be some anisotropic similitude unitary group over
\(\Q\) with similitude factor \(c : G\rightarrow\mathbb{G}_{m,\Q}\)
and let \(H=\ker(c)\). We have \(H=\Res_{F/\Q}H'\) with \(H'\) an
unitary group over \(F\) such that \(H'_E\) is an inner form of
\(\GL_n\). Note that we have a decomposition
\(\mathfrak{g}=\mathfrak{h}\oplus\mathfrak{s}\) where \(\mathfrak{s}\)
is the Lie algebra of the maximal \(\Q\)-split normal torus \(S\)
contained in \(G\).

As in section
\ref{sec_unitary}, we define the Hecke algebra \(\mathbb{T}^{\Spl}\)
using places \(v\) of \(F\) which are split in \(E\) and such that \(H'\) is
quasi-split at \(v\). Let \(\mm\) be some open maximal ideal in
\(\mathbb{T}^{\Spl}\) and let \(\rho_{\mm} :
\Gal_E\rightarrow\GL_n(\overline{\F}_p)\) be Galois representation associated to
\(\mm\) in \cite[Thm.~1.1]{CS}. Up to enlarging \(L\), we can assume
that \(\rho_{\mm}\) takes values in \(\GL_n(k)\). We assume from now
that \(\rho_{\mm}\) is decomposed generic (see Definition 1.9 in
\emph{loc.~cit.}) so that, by Thm.~1.1 in \emph{loc.~cit.}, the ideal
\(\mm\) is non Eisenstein.

Let \(V\) be some algebraic irreducible representation of \(G\) such
that \(\mathbb{T}^{\Spl}_{\mm,V}\neq0\). If
\(x : \mathbb{T}^{\Spl}_{\mm,V}\rightarrow\C\) is a character of
\(\mathbb{T}^{\Spl}_{\mm,V}\), we fix some automorphic representation
\(\pi_x\) of \(G(\mathbb{A}_\Q)\) whose existence is assured by Lemma
\ref{TmVtau}. It follows from \cite[Thm.~A.1]{Shin_appendix_Goldring}
that we can define the base change \(\Pi\) of \(\pi\) which is an
isobaric sum of conjugate self dual \(C\)-algebraic cuspidal
automorphic representations of \(\GL_{n,E}\). The main Theorem of
\cite{CheHarris} gives us some admissible representation
\(\rho_x : \Gal_E\rightarrow \GL_n(\overline{\Q}_p)\) associated to
\(\Pi\). Assume now that \(\rho_{\mathfrak{m}}\) is absolutely
irreducible and that \(p>2\). Reasoning as in the proof of
\cite[Prop.~3.4.4]{CHT}, we can construct a continuous representation
\(\rho :
\Gal_{E,S}\rightarrow\GL_n(\mathbb{T}^{\Spl}_{\mathfrak{m}})\) such
that, for all
\(x : \mathbb{T}^{\Spl}_{\mathfrak{m},V}\rightarrow\overline{\Q}_p\),
\(\rho_x\) is associated to \(\pi_x\) by the previous construction. As
in subsection \ref{sec_unitary}, for all
\(x\in\mSpec\mathbb{T}^{\Spl}_{\mathfrak{m}}[1/p]\), the center
\(Z(\mathfrak{h})\) acts on
\((\widetilde{H}^{q_0}\otimes_{\OO}L)[\mm_x]^{\la}\) by
\(\zeta^C_{\rho_x^C}\), the representation \(\rho_x^C\) being obtained
from \(\rho_x\) using \(\tilde{\delta}=(0,-1,\dots,1-n)\). As
\(S(\Qp)\) acts on
\((\widetilde{H}^{q_0}\otimes_{\OO}L)[\mm_x]^{\la}\) by a character by
Lemma \ref{decomp_char}, it follows that \(Z(\mathfrak{g})\) acts by a
character on this space.

\subsection{Patched module}\label{sec_patch} In this subsection let $F$ be a finite extension of $\Qp$.
Let $n\geq1$ be such that $p\nmid 2n$, let
$\rhobar : \Gal_F\rightarrow\GL_n(k)$ be a Galois representation and
let $R_{\overline{\rho}}^\Box$ be the framed deformation ring of
$\rhobar$.

Let $K=\GL_n(\mathcal{O}_F)$ and let $\mathfrak{g}$ be the
$\Qp$-linear Lie algebra of $G=\GL_n(F)$.  Let $M_\infty$ be the
patched module constructed in \cite{6auth} by patching automorphic
forms on definite unitary groups.  This is a compact
$R_\infty[\![K]\!]$-module carrying an $R_\infty$-linear action of $G$
extending the action of $K$, where $R_{\infty}$ is a complete local
noetherian $R_{\overline{\rho}}^{\Box}$-algebra with residue field
$k$, which is flat over $R_{\overline{\rho}}^{\Box}$.

We let $\Pi_{\infty}\coloneqq \Hom_{\OO}^{\cont}(M_{\infty}, L)$. If
$y\in \mSpec R_{\infty}[1/p]$ then $\Pi_{\infty}[\mm_y]$ is an
admissible unitary $\kappa(y)$-Banach space representation of $G$.  By
letting $x$ be the image of $y$ in
$\mSpec R_{\overline{\rho}}^\Box[1/p]$, we obtain a Galois
representation $\rho_x: \Gal_F \rightarrow \GL_n(\kappa(x))$. The
expectation is that $\Pi_{\infty}[\mm_y]$ and $\rho_x$ should be
related by the hypothetical $p$-adic Langlands correspondence, see
\cite[\S 6]{6auth}.

Let $\rho:\Gal_F \rightarrow \GL_n(R_{\infty})$ be the Galois
representation obtained from the universal framed deformation of
$\rhobar$ by extending scalars to $R_{\infty}$.  Let
\[\rho^C\coloneqq \tw_{\tilde{\delta}}^{-1} \circ (\rho\boxtimes
  \chi_{\cyc}),\] where $\tilde{\delta}$ is the twisting element
$(0,-1,\dots,2-n,1-n)$ of $\GL_n$, see Remark \ref{twisting_GLn}.

\begin{thm}\label{patched}
  The algebra $Z(\mathfrak{g})$ acts on $\Pi_{\infty}[\mm_y]^{\la}$
  through the character $\zeta_{\rho^C_x}^C$.
\end{thm}

\begin{proof}
  We will show that Theorem \ref{infinitesimal} applies with
  $M=M_\infty$, $R=R_\infty$ and $\zeta=\zeta^C_{\rho^C}$.  By
  \cite[Prop.\,2.10]{6auth}, there is a morphism of local rings
  \[ S_\infty=\OO[\![y_1,\dots,y_h]\!] \rightarrow R_\infty\]
  such that $M_\infty$ is a finitely generated projective
  $S_\infty[\![K]\!]$-module. Thus the sequence $(y_1,\dots,y_h)$ is
  $M_\infty$-regular and $M_\infty/(y_1,\dots,y_h)$ is a finitely generated projective
  $\mathcal{O}[\![K]\!]$-module. This gives hypothesis (i).

  Let $V$ be an irreducible algebraic representation of
  $\Res_{F/\Qp}\GL_n$ over $L$. Let $R_{\infty}(V)$ be the quotient of
  $R_{\infty}$ acting faithfully on $\Hom_K(V, \Pi_{\infty})$. Then
  $R_{\infty}(V)$ is reduced by \cite[Lem.\,4.17]{6auth} and we have
  (ii).

  If $y \in\Sigma_V$, the $\OO$-algebra homomorphism
  $x : R_\infty\rightarrow\mathcal{O}_{\kappa(y)}$ factors through
  $R_{\infty}(V)$ and, by \cite[Prop.\,4.33]{6auth}, the
  representation $\rho_x$ is crystalline with Hodge--Tate cocharacter
  $\lambda+\tilde{\delta}$ where $\lambda$ is the highest weight of
  $V$ and $x$ is the image of $y$ in $\mSpec
  R_{\rhobar}^{\Box}[1/p]$. By Proposition \ref{inf_HT}, see also
  Remark \ref{twisting_GLn}, the algebra $Z(\mathfrak{g})$ acts on $V$
  via the character $\zeta^C_{\rho_x^C}$. Thus part (iii) of Theorem
  \ref{infinitesimal} is satisfied.
 
  The specialisation of $\zeta^C_{\rho^C}$ at
  $y\in \mSpec R_{\infty}[1/p]$ is equal to $\zeta^C_{\rho^C_x}$ by
  Lemma \ref{base_change_C}, where $x$ is the image of $y$ in
  $\mSpec R_{\rhobar}^{\Box}[1/p]$, thus we obtain the result.
  \end{proof}
  
  \begin{remar} It is not known in general whether the Banach space
    representation $\Pi_{\infty}[\mm_y]$ depends only on the Galois
    representation $\rho_x$. However, this is expected to be true, as
    the $p$-adic Langlands correspondence should not depend on the
    choices made in the patching process. The theorem above shows that
    the infinitesimal character of $\Pi_{\infty}[\mm_y]^{\la}$ depends
    only on $\rho_x$, thus adding evidence that the expectation should
    be true.
 \end{remar}
 
 \subsection{The $p$-adic Langlands correspondence for $\GL_2(\Qp)$}\label{padicLL} It is shown in \cite{image} for $p\ge 5$ and in \cite{CDP} in general that 
 Colmez's Montreal functor $\Pi \mapsto \cV(\Pi)$ induces a bijection
 between the equivalence classes of absolutely irreducible admissible
 $L$-Banach space representations $\Pi$ of $G=\GL_2(\Qp)$, which do
 not arise as subquotients of parabolic inductions of unitary
 characters, and the equivalence classes of absolutely irreducible
 Galois representations $\rho: \Gal_{\Qp} \rightarrow \GL_2(L)$.  The
 correspondence is normalised so that local class field theory matches
 the central character of $\Pi(\rho)$ with
 $ \chi_{\cyc}^{-1}\det \rho$.  Let
 $\rho^C: \Gal_{\Qp} \rightarrow \CG_f(L)$ be the Galois
 representation attached to $\rho$ using the twisting element
 $\tilde{\delta}=(1, 0)$.

 The following result proved by one of us (G.\,D.\,) was an important
 motivation for this paper and we will give a new proof of it.

 \begin{thm}[\cite{GDinf}, Thm.\,1.2]\label{tung} Let $\Pi$ be as
   above and let $\rho=\cV(\Pi)$ then the action of $Z(\mathfrak g)$
   on $\Pi^{\la}$ is given by $\zeta^C_{\rho^C}$.
 \end{thm}
 \begin{proof} We will use the recent results of Shen-Ning Tung
   \cite{Tung1}, \cite{Tung2} proved in his thesis. We may assume that
   $\rho$ is the specialisation at
   $x\in \mSpec R_{\rhobar}^{\Box}[1/p]$ of the universal framed
   deformation of some $\rhobar:\Gal_{\Qp} \rightarrow \GL_2(k)$. If
   $p>2$ then let $M_{\infty}$ and $R_{\infty}$ be as in the previous
   section with $n=2$ and $F=\Qp$. Since $R_{\infty}$ is flat over
   $R_{\rhobar}^{\Box}$ there is an $\OO$-algebra homomorphism
   $y: R_{\infty}\rightarrow \OO$ extending $x$. It follows from
   \cite[Thm.\,4.1]{Tung1} that there is an irreducible subquotient
   $\Pi'$ of $\Pi_{\infty}[\mm_y]$, such that
   $\cV(\Pi')=\rho$. Theorem \ref{patched} implies that the action of
   $Z(\mathfrak g)$ on $\Pi_{\infty}[\mm_y]^{\la}$, and hence also on
   $(\Pi')^{\la}$, is given by $\zeta^C_{\rho^C}$. Using the
   bijectivity of the correspondence we obtain that $\Pi=\Pi'$.  If
   $p=2$ then in \cite{Tung2} Tung carries out the patching
   construction himself to obtain the analog of $M_{\infty}$, see
   \cite[Prop.\,6.1.2]{Tung2}, which says that the patched module
   satisfies parts (o) and (i) of Theorem
   \ref{infinitesimal_central}. The same argument as in the case $p>2$
   can be carried over using \cite[Thm.\,6.3.7]{Tung2}.
\end{proof} 

\begin{remar} The proof of bijectivity in \cite{CDP} uses the results
  of \cite{GDinf} in an essential way. However, if one is willing to
  assume that $p\ge 5$ or if $p=2$ or $p=3$ then
  $\rhobar^{\mathrm{ss}}\neq \chi \oplus \chi \omega$, for any
  character $\chi: \Gal_{\Qp}\rightarrow k^{\times}$, where $\omega$
  is the reduction of $\chi_{\cyc}$ modulo $p$, then the bijectivity
  follows from \cite[Thm.\,1.3]{image},
  \cite[Cor.\,1.4]{blocksp2}. The papers \cite{image} and
  \cite{blocksp2} use only Colmez's functor $\cV$, which goes from
  $\GL_2(\Qp)$-representations to Galois representations, see \cite[\S
  IV]{colmez}, and not the construction $V\mapsto \Pi(V)$, which uses
  the $p$-adic Hodge theory in a deeper way and is used in
  \cite{GDinf}.
\end{remar}

\subsection{ A conjectural picture}\label{eyeswideshut} In this subsection we formulate a conjecture, which describes 
the infinitesimal characters of the subspace of locally analytic
vectors of Hecke eigenspaces in completed cohomology in the general
setting, i.e. in the setting when one does not expect weakly
non-Eisenstein ideals to exist.

Let $\mathbb T$ be the Hecke algebra defined in subsection
\ref{sec_hecke}.  Let $x: \mathbb{T}[1/p]\rightarrow \Qpbar$ be a
continuous homomorphism of $\OO$-algebras with kernel $\mm_x$, such
that the image of $x$ is a finite extension of $\Qp$. Note that this
condition is satisfied if $\mathbb T$ is noetherian. If $\ell$ a prime
number which is not in $S$, then we will denote by
$x_{\ell}: \mathcal H_{\ell}\rightarrow \Qpbar$ the composition of $x$
with the natural map $\mathcal H_{\ell}\rightarrow \mathbb T$.

There is a version of the Satake isomorphism using the $C$-group which
is defined in \cite{zhu} (see Proposition $5$ and Remark $6$ in
\emph{loc.~cit.}). As $\ell$ is invertible in $\OO$, it takes the form
of an isomorphism of $\OO$-algebras
\[ \mathcal H_\ell \simeq
  \OO[\Ghat^T|_{d=\ell}\rtimes\set{\Frob_{\ell}}]^{\Ghat}\] where
$\Frob_{\ell}$ is a geometric Frobenius at $\ell$ and
$\Ghat^T|_{d=\ell}$ is the subscheme of $\Ghat^T$ which is the inverse
image of $\ell$ under $d : \Ghat^T\rightarrow\mathbb{G}_m$.

Using this isomorphism, we can associate to $x_\ell$ a semisimple
$\Ghat(\Qpbar)$-conjugation class $\mathrm{CC}(x_\ell)$ in
$\Ghat(\Qpbar)\rtimes (\set{\ell}\times\set{\Frob_{\ell}})$.

Inspired by \cite[Conj.\,5.3.4]{beegee} we have the (very) optimistic
conjecture :

\begin{conj}\label{dreamconj}
  There exists an admissible representation
  \[ \rho : \Gal_{\QQ}\longrightarrow \CG_f(\Qpbar)\]
  such that
  \begin{enumerate}[(i)]
  \item $d\circ\rho$ is the cyclotomic character;
  \item $\rho$ is unramified outside of $S$;
  \item for $\ell\notin S$, the semisimplification of $\rho(\Frob_{\ell})$ is in
    \[\mathrm{CC}(x_{\ell})\xi(\chi_{\cyc}(\Frob_\ell))\subset\Ghat(\Qpbar)\rtimes(\set{\ell^{-1}}\times\set{\Frob_{\ell}})\]
    where $\xi : \mathbb{G}_m\rightarrow \widehat{G}^T$ is the
    cocharacter is $t\mapsto ((2\delta)(t^{-1}),t^2)$, where $2\delta$
    is the sum of positive roots.
  \end{enumerate}
\end{conj}
\begin{remar} The cocharacter $\xi$ is central in $\Ghat^T$, thus is
  independent of the choice of $\Bhat$ and $\That$ and
  $\mathrm{CC}(x_{\ell})\xi(\chi_{\cyc}(\Frob_\ell))$ is a
  $\Ghat^T(\Qpbar)$--conjugacy class.
\end{remar}

\begin{remar} If $\iota \circ x: \mathbb T\rightarrow \C$, where
  $\iota: \Qpbar \cong \C$ is a fixed isomorphism, is associated to a
  $C$-algebraic automorphic form then the existence of $\rho$
  satisfying the conditions of Conjecture \ref{dreamconj} is
  conjectured in \cite[Conj.\,5.3.4]{beegee}.
\end{remar}

Assume that $\rho$ and $\rho'$ are two admissible representations
associated to $x$ as in Conjecture \ref{dreamconj}, then
$\rho(\Frob_\ell)^{\mathrm{ss}}$ and $\rho'(\Frob_\ell)^{\mathrm{ss}}$
are conjugate by an element of $\Ghat(\Qpbar)$. Let $E$ be a finite
Galois extension of $\QQ$ unramified outside $S$ such that $\Gal_E$
acts trivially on the root datum of $\Ghat$. If $\gamma\in \Gal_E$,
then $\rho(\gamma)=(c_{\gamma}, 1)$ and
$\rho'(\gamma)=(c'_{\gamma}, 1)$ with
$c_{\gamma}, c'_{\gamma}\in \Ghat^T(\Qpbar)$ and thus
\[ g \rho(\gamma) g^{-1}= ( g c_{\gamma} g^{-1}, 1), \quad \forall g\in \Ghat(\Qpbar).\]
Thus if  $\ell$ splits completely in $E$ then 
\[ \tr_V(r(c_{\Frob_v}))=\tr_V(r(c'_{\Frob_v})),\] for all places $v$
of $E$ above $\ell$ and all algebraic representations $(r, V)$ of
$\Ghat^T$.  By \v{C}ebotarev density we have
$\tr_V(r(c_{\gamma}))=\tr_V(r(c'_{\gamma}))$ for all
$\gamma \in \Gal_E$ and thus $\zeta^C_{\rho}=\zeta^C_{\rho'}$, by
Lemma \ref{deptrace}.  This proves that if Conjecture \ref{dreamconj}
is true, there is a well defined character
$Z(\mathfrak{g})\rightarrow\Qpbar$ associated to $x$.

\begin{conj}\label{dream_padic}
  Let $\rho$ be an admissible representation associated to
  $x: \mathbb T[1/p] \rightarrow \Qpbar$ as in Conjecture
  \ref{dreamconj}, and let $n$ be a non-negative integer. Then
  $Z(\mathfrak{g})$ acts on
  $(\widetilde{H}^n\otimes_{\OO} L)[\mm_x]^{\la}\otimes_{\mathbb T,
    x}\Qpbar$ via $\zeta^C_{\rho}$.
\end{conj}

\begin{remar} Note that $(\widetilde{H}^n\otimes_{\OO} L)[\mm_x]$ can
  be zero, in which case the statement of the Conjecture
  \ref{dream_padic} holds trivially.
\end{remar}

\begin{remar}\label{know_it_all} In the general case, even after localising $\widetilde{H}_n$ at a maximal ideal of the Hecke algebra, we do not expect to get a projective $\OO[\![K_p]\!]$-module. So 
  Theorems \ref{infinitesimal} and \ref{infinitesimal_central} cannot
  be applied directly. However, one might hope to be able to apply our
  results to the patched homology groups obtained via the patching
  method of Calegary--Geraghty \cite{cale-ger}. The most accessible
  case, when weakly non-Eisenstein maximal ideal are not expected to
  exist, is when $G= \PGL_2$ over a quadratic imaginary field $F$,
  such that $p$ splits completely in $F$, studied by Gee--Newton in
  \cite{gee-newton}. It follows from \cite[Prop.\,5.3.1]{gee-newton}
  and its proof that under the assumptions made there the patched
  homology, denoted by $H_{q_0}(\widetilde{\mathcal{C}}(\infty))$ in
  \cite{gee-newton}, satisfies the conditions of Theorem
  \ref{infinitesimal}.  We do not pursue this further, just remark
  that in that setting instead of applying Theorem \ref{infinitesimal}
  to $H_{q_0}(\widetilde{\mathcal{C}}(\infty))$ it might be easier to
  use local--global compatibility at $p$ and appeal to the results on
  the infinitesimal character in the $p$-adic Langlands correspondence
  for $\GL_2(\Qp)$, see Theorem \ref{tung}.
\end{remar}

\end{document}